\documentclass[10pt]{preprint}
\usepackage{color}

\usepackage{tikz}
\usepackage{amsmath}
\usepackage{amssymb}
\usepackage{amsthm}
\usepackage{mhfig}
\usepackage{mhequ}
\usepackage[margin=4cm]{geometry}
\usepackage{times}
\usepackage{microtype}
\usepackage{shuffle}

\usepackage{verbatim}

\newcommand{\fixmap}{\mathcal{M}}

\newcommand{\defin}{\stackrel{{\rm{def}}}{=}}
\newcommand{\bigz}{\mathbf{Z}}
\newcommand{\bigy}{\mathbf{Y}}

\newcommand{\bigybar}{\bar{\bigy}}

\newcommand{\rtilde}{\tilde{r}}
\newcommand{\lambdatilde}{\tilde{\lambda}}

\newcommand{\CK}{\mathcal{K}}

\newcommand{\Ztilde}{\tilde{Z}}
\newcommand{\Ntensor}{T^{(N)}}

\def\forest{\mathcal{F}}
\def\hopf{\mathcal{H}}

\def\X{{\bf X}}
\def\Y{{\bf Y}}
\def\reals{\mathbb{R}}
\def\innerprod#1{\langle #1 \rangle}
\def\inner#1{\langle #1 \rangle}
\def\innerbig#1{\left\langle #1 \right\rangle}
\def\liloh{{o}}
\def\hom{{\rm Hom}}
\def\del{\partial}
\def\norm#1{\|#1\|}
\def\S{\mathcal{S}}
\def\G{\mathcal{G}}

\def\fbar{\bar{f}}

\def\trees{\mathcal{T}}

\def\sup{{\rm sup}}
\def\Id{{\rm Id}}

\def\spn{{\rm{span}}}
\def\naturals{\mathbb{N}}
\def\expstar{\exp_\star}
\def\logstar{\log_\star}

\def\B{\mathcal{B}}
\newcommand{\integration}{\mathcal{I}}
\newcommand{\partition}{\mathcal{P}}

\def\shufset{{\rm{Shuf}}}

\def\otimestilde{\mathop{\tilde{\otimes}}}
\def\otimest{\mathop{\tilde{\otimes}}}

\def\fbar{\bar{f}}
\def\xbar{\bar{X}}
\def\Xbar{\bar{\X}}

\def\Deltabar{\bar{\Delta}}
\def\deltabar{\bar{\delta}}

\def\Xhat{\widehat{\bf X}}
\def\xhat{\widehat{X}}

\def\S{\mathcal{S}}

\newcommand{\sym}{{{\rm sym}}}
\newcommand{\bigytilde}{\tilde{\mathbf{Y}}}

\newcommand{\alphabet}{\mathcal{A}}

\def\Xtilde{\tilde{\X}}
\newcommand{\controlled}{\mathcal{Q}_\X(\reals^e)}

\newcommand{\unit}{{\mathbf{1}}}
\newcommand{\counit}{{\mathbf{1}}^*}

\def\mytreeone#1{\def\firstlabel{\mbox{\tiny $#1$}}\mhpastefig{tree_one}}
\def\mytreeoneone#1#2{\def\firstlabel{\mbox{\tiny $#1$}}\def\secondlabel{\mbox{\tiny $#2$}}\mhpastefig{tree_oneone}}
\def\mytreeoneoneone#1#2#3{\def\firstlabel{\mbox{\tiny $#1$}}\def\secondlabel{\mbox{\tiny $#2$}}\def\thirdlabel{\mbox{\tiny $#3$}}\mhpastefig{tree_oneoneone}}
\def\mytreetwoone#1#2#3{\def\firstlabel{\mbox{\tiny $#1$}}\def\secondlabel{\mbox{\tiny $#2$}}\def\thirdlabel{\mbox{\tiny $#3$}}\mhpastefig{tree_twoone}}

\newcommand{\ntreeoneone}{
\begin{tikzpicture}
\draw (0,0)  
-- (0,1.3ex);
\draw[fill=black] (0,0) circle (0.38ex) ;
\draw[fill=black] (0,1.3ex) circle (0.38ex);
\end{tikzpicture}}

\newcommand{\treeoneone}{
\begin{tikzpicture}
\draw (0,0)  
-- (0,1ex);
\draw[fill=black] (0,0) circle (0.3ex) ;
\draw[fill=black] (0,1ex) circle (0.3ex);
\end{tikzpicture}}

\theoremstyle{plain}
\newtheorem{thm}{Theorem}[section]
\newtheorem{lemma}[thm]{Lemma}
\newtheorem{prop}[thm]{Proposition}
\newtheorem{corr}[thm]{Corollary}

\theoremstyle{definition}
\newtheorem{defn}[thm]{Definition}
\newtheorem{example}[thm]{Example}

\newtheorem{rmk}[thm]{Remark}
\newtheorem{remark}[thm]{Remark}

\numberwithin{equation}{section}

\begin{document}
%
\title{Geometric versus non-geometric rough paths}
\author{Martin Hairer$^1$ and David Kelly$^2$}
\institute{The University of Warwick, \email{M.Hairer@Warwick.ac.uk}
\and The University of Warwick, \email{David.Kelly@Warwick.ac.uk}}
\maketitle\thispagestyle{empty}

\maketitle
\begin{abstract}In this article we consider rough differential equations (RDEs) driven by non-geometric rough paths, using the concept of branched rough paths introduced in \cite{gubinelli10a}. We first show that branched rough paths can equivalently be defined as $\gamma$-H\"older continuous paths in some Lie group, akin to geometric rough paths. We then show that  every branched rough path can be \emph{encoded} in a geometric rough path. More precisely, for every branched rough path $\X$ lying above a path $X$, there exists a geometric rough path $\Xbar$ lying above an extended path $\xbar$, such that $\Xbar$ contains all the information of $\X$. As a corollary of this result, we show that every RDE driven by a non-geometric rough path $\X$ can be rewritten as an extended RDE driven by a geometric rough path $\Xbar$. One could think of this as a generalisation of the It\^o-Stratonovich correction formula. 
\end{abstract}
\section{Introduction}
The so-called \emph{controlled differential equations} have become an important class of dynamical systems throughout the last half century, the most notable example being the It\^o diffusions. Roughly speaking, these systems take the form
\begin{equ}\label{e:intro_SDE}
dY_t = \sum_i f_i(Y_t)dX_t^i\;, \quad Y_0 = \xi \;, 
\end{equ}
where $X$ and $Y$ are paths in vector spaces $V$ and $U$ respectively, with $X = (X^i)$ and $X_0 = 0$, and where the vector fields $f_i : U \to U$ are smooth non-linear functions. For simplicity, we will always assume that $V$ and $U$ are finite dimensional, with $V=\reals^d$ and $U=\reals^e$, so that there
is a canonical identification between these spaces and their duals.
\par
For a path $X$ of bounded variation, the notion of a solution is unambiguously defined using any variant of Riemann-sum style integration. However, for a less regular $X$ this isn't always the case. For example, let $X$ be a sample path of Brownian motion in $\reals^d$, which is (almost surely) $\gamma$-H\"older continuous, for every $\gamma < 1/2$. It is clear that the solution $Y$ depends on how one interprets the integral in \eqref{e:intro_SDE}. In particular, both It\^o and Stratonovich integrals provide two distinct notions of a solution. Another way of looking at it is that there is something \emph{missing} from \eqref{e:intro_SDE}, namely, the blueprint of how to construct integrals against $dX$. The theory of rough paths, first introduced by T.~Lyons in \cite{lyons98}, provides an elegant way of encoding this missing ingredient.      
\par
Instead of viewing \eqref{e:intro_SDE} as an equation controlled by $X$, one should recast it (formally speaking) as 
\begin{equ}\label{e:intro_RDE}
dY_t = \sum_i f_i(Y_t)d\X_t^i\;, \quad Y_0 = \xi \;, 
\end{equ}
an equation controlled by a path $\X$, known as a \emph{rough path}, that is an extension of $X$, taking values in a much bigger (non-linear) space. The equation \eqref{e:intro_RDE} is known as a \emph{rough differential equation} (RDE). The extra components of $\X$ provide the necessary information on how to interpret those integrals encountered in controlled differential equations, hence they provide the information that was missing in \eqref{e:intro_SDE}. This interpretation has proved extremely useful in the framework of It\^o diffusions, most notably in illustrating the continuity properties of the It\^o map.    
\par
However, when the driving path $X$ has H\"older regularity $\gamma \leq 1/3$, one must impose an extra condition to ensure that equations like \eqref{e:intro_SDE} can still be treated in the framework of rough paths. Namely, the integrals in \eqref{e:intro_SDE} must obey ``the usual rules of calculus'' in that, like Stratonovich integrals, they must satisfy the ordinary chain-rule and integration by parts formulae, without any correction terms. This framework has been used, for example, in the analysis of equations driven by fractional Brownian motion with Hurst parameter $H > 1/4$ \cite{coutin02,victoir10,hairer12}.    
\par
In certain situations, the geometric framework is not an appropriate model for a stochastic system. For example, in some financial models, an It\^o type integral is more appropriate than Stratonovich, since the latter scheme requires one to ``look into the future". More generally, it is often the case that natural approximations to stochastic integrals \emph{do not} converge to objects for which the usual change of variables formula holds. Indeed, discrete
approximations to an integral do not in general have any reason to satisfy the integration by parts formula exactly.
While the resulting error term would vanish when integrating smooth functions against each other,
this does not always happen in the stochastic case where integrands and integrators are typically very rough. 
The most famous example of this is of course the It\^o integral, however the phenomenon is also widespread in 
the world of non semi-martingales \cite{burdzy96,russo00,nourdin05,swanson10}. Thus, the limiting objects from discretisation schemes are often \emph{non}-geometric. Recently, M.~Gubinelli introduced the notion of  a \emph{branched rough path}, which is an extension of the original formulation, created to extend the scope of rough path theory to such non-geometric situations \cite{gubinelli10a}. 

As we will see below, this extension does actually not alter the fundamental theory of rough paths at all, but merely requires that some \emph{additional} components be added to the rough path $\X$. Indeed, the main
result of this article, Theorem~\ref{thm:intro1} below, shows that the solution to a differential equation driven by a 
branched rough path can always be recovered as the solution to a (usually different) differential equation driven by
a geometric rough path.
Before introducing branched rough paths, we will first give an overview of how geometric rough paths are used to solve controlled differential equations. 

\subsection{Geometric rough paths}\label{sec:geometric}
The missing ingredients contained in the rough path $\X$ can be interpreted as the \emph{iterated integrals} of $X$. If $X$ takes values in $V$, then $\X$ takes values in $T((V))$, the topological dual of the tensor product algebra $T(V)$, defined by
\begin{equ}
T(V) = V \oplus V^{\otimes 2} \oplus \dots \;.
\end{equ}   
Hence, $T((V))$ can be identified with formal tensor series on $V$. The lowest order components, are simply the components of $X$, in that
\begin{equ}
\inner{\X_t,e_i} = X_t^i \;,
\end{equ}
for $i=1 \dots d$, where $e_i$ is the $i$-th basis vector of $V$. The higher order components are (formally) given by the iterated integrals
\begin{equ}\label{e:intro_iterated}
\inner{\X_t, e_{i_1 \dots i_n}} \defin \int_0^t \dots \int_0^{r_{2}} dX_{r_1}^{i_1} \dots dX_{r_n}^{i_n}\;,
\end{equ}
for $i_1,\dots,i_n = 1 \dots d$, where we use the shorthand $e_{i_1 \dots i_n} = e_{i_1}\otimes \dots \otimes e_{i_n}$. Of course, this is only defined formally since the above integrals cannot be constructed for an arbitrary $X$. Hence, one should think that a given rough path $\X$ \emph{defines} the integral on the right hand side of \eqref{e:intro_iterated}. 
\par
The concept of satisfying the ``usual rules of calculus'' is encapsulated by requiring that 
\begin{equ}\label{e:intro_shuffle}
\inner{\X_t,e_{i_1 \dots i_n}}\inner{\X_t,e_{j_1 \dots j_m} } = \inner{\X_t,e_{i_1 \dots i_n} \shuffle e_{j_1 \dots j_m}}\;,
\end{equ} 
for all tensors $e_{i_1 \dots i_n}$, $e_{i_1 \dots i_n}$ and where $\shuffle$ denotes the shuffle product \cite{reutenauer93}. The shuffle product $w\shuffle v$ of two words $w,v$ is given the sum of all words that are obtained by combining and rearranging the words $w,v$ whilst also preserving their original orderings. For example,
\begin{align*}
e_i \shuffle e_j &= e_{ij} + e_{ji}\\
e_i \shuffle e_{jk} &= e_{ijk} + e_{jik} + e_{jki}\;,
\end{align*}
note that $e_{kji}$ does not appear in the second expression since it does not preserve the ordering $jk$. Hence, we have that, for example
\begin{equ}
\inner{\X_t,e_i}\inner{\X_t,e_j} = \inner{\X_{t},e_{ij}} + \inner{\X_{t},e_{ji}}\;,
\end{equ}
which, by substituting \eqref{e:intro_iterated}, gives the usual integration by parts formula. Hence, one should think of \eqref{e:intro_shuffle} as a generalisation of the integration by parts formula to higher order iterated integrals. 
\begin{remark}
It is well known that when $X$ is smooth and the rough path $\X$ is constructed \emph{canonically} using Riemann integrals, then the identity \eqref{e:intro_shuffle} is always satisfied \cite{chen77}. 
\end{remark}
Of course, for a fixed $t$, the object $\X_t$ cannot be \emph{any} element of the truncated tensor product algebra. Instead, $\X_t$ lives in a special subset, which happens to be a Lie group, denoted by $(G(V),\otimes)$, called the \emph{free nilpotent group}, with the group operation given by the tensor product. This is defined by 
\begin{equ}
G(V) \defin \exp \G(V)\;,
\end{equ} 
where $\G(V) \subset T((V))$ is the space of formal Lie series generated by $V$ and where $\exp$ is the tensor exponential. The group $G(V)$ can equally be defined as the \emph{group-like} elements or \emph{characters} with respect to the shuffle product, which ensures \eqref{e:intro_shuffle}. These algebraic ideas will be made concise in Section \ref{sec:geometric}.  
\par
When solving controlled differential equations, it is often more convenient to work with the increment $\delta X_{st} \defin X_t -X_s$ instead of the path $X_t$. The same is true of rough paths, hence we define 
\begin{equ}
\X_{st} \defin \X_s^{-1}\otimes \X_t\;,
\end{equ}
where $\X_s^{-1}$ denotes the group inverse of $\X_s$. This yields the following definition, which is equivalent
to the one given in \cite{lyons98,friz10}:

\begin{defn}
A \emph{weak geometric rough path of regularity $\gamma$} is a map $\X:[0,T]\times[0,T]\to T((V))$ satisfying the following three conditions
\begin{enumerate}
\item$\innerprod{\X_{st},x \shuffle y} = \innerprod{\X_{st},x}\innerprod{\X_{st},y}$, for every $x,y \in T(V)$,
\item$\X_{st} = \X_{su}\otimes \X_{ut}$,
\item $\sup_{s\neq t} {|\innerprod{\X_{st},w}|}/{|t-s|^{\gamma |w|}} < \infty$,
for every $w \in T(V)$ with $|w| \leq N$, where $|w|$ denotes the number of letters composing the word $w$, which we will refer to as the \emph{length} of the word $w$.
\end{enumerate}
\end{defn}

\begin{rmk}
There is a subtle difference between weak geometric rough paths and geometric rough paths \cite{friz06}. In this article we only refer to the weak kind and will henceforth omit the prefix.
\end{rmk}

\begin{rmk}
By definition of the group $G(V)$, we could equivalently say that a geometric rough path $\X$ is a function $\X : [0,T]\times[0,T] \to G(V)$ that satisfies properties $(2)$ and $(3)$.
\end{rmk}
\par
\begin{rmk}\label{rmk:truncate}
One of the crucial properties of a geometric rough path $\X$ of regularity $\gamma$ is that only finitely many components actually matter. To be precise, let $N$ be the larger integer such that $N\gamma\leq1$, then one can show that all components $\inner{\X_{st},e_{i_1 \dots i_n}}$ for $n>N$ are uniquely determined by those elements with $n \leq N$, see \cite[Theorem 2.2.1]{lyons98}. Intuitively, these larger components are `regular enough' to be defined in a canonical way. Moreover, we will see that when solving a differential equation using $\X$, the components with $n>N$ become negligible in an expression for the solution. For these reasons, one often defines a geometric rough path as taking values in the truncated group $G^{(N)}(V)$, defined by simply discarding those components of elements in $\G(V)$ indexed by more than $N$ letters. These ideas will be made precise in Section \ref{sec:geometric}. The intention of defining the geometric rough path in the above fashion is to draw the connection between itself and the branched rough path, which will be introduced in the following subsection. 
\end{rmk}

One simple example of a rough path is the canonical rough path constructed above a smooth. Since the works of Chen \cite{chen77}, it has been known that if $X$ is a smooth path, then the quantities
given by 
\begin{equ}
\inner{\X_{st}, e_{i_1 \dots i_n}} \defin \int_s^t \dots \int_s^{r_{2}} dX_{r_1}^{i_1} \dots dX_{r_n}^{i_n}\;,
\end{equ}
do indeed satisfy the two algebraic relations given in the above definition.
\par
To solve the RDE \eqref{e:intro_SDE}, we adopt the idea of \emph{controlled rough paths}, introduced in \cite{gubinelli04}; the key observation is that $Y$ is \emph{locally controlled} by the rough path $\X$. We will illustrate this by assuming that $1/4<\gamma \leq 1/3$, so that $N=3$. As usual, we assume that $V=\reals^d$, $U=\reals^e$ and that $f(Y)dX = \sum_{i=1}^d f_i(Y)dX^i$, where the vector fields $f_i : \reals^e \to \reals^e$ are smooth. We will denote by $f_{i}^\alpha$ the $\alpha$-th coordinate of the vector field $f_i$. Then \eqref{e:intro_SDE} can be written in the integral form
\begin{equ}\label{e:integral_SDE}
\delta Y_{st} = \int_s^t f_i(Y_v) dX^i_v\;,
\end{equ}
where we omit the sum notation. If we perform a Taylor expansion of $f_i$ around $Y_s$ and repeatedly substitute \eqref{e:integral_SDE} back in to itself, then we formally obtain
\begin{align}
\delta Y_{st} &\approx  f_i(Y_s)\int_s^t dX_{v_1}^i + f^{\alpha_1}_j(Y_s) \del^{\alpha_1}f_i(Y_s) \int_s^t \int_s^{v_1} dX_{v_2}^{j}dX_{v_1}^i\label{e:controlled_expansion} \\
&+ f_k^{\alpha_1}(Y_s) \del^{\alpha_1} f_j^{\alpha_2}(Y_s) \del^{\alpha_2} f_i(Y_s) \int_s^t \int_s^{v_2} \int_s^{v_1} dX_{v_3}^{k}dX_{v_2}^j dX_{v_1}^i\notag \\
&+ \frac{1}{2}f_k^{\alpha_1}(Y_s) f_j^{\alpha_2}(Y_s) \del^{\alpha_1 \alpha_2}f_i(Y_s) \int_s^t \left(\int_s^{v_3} dX_{v_1}^k\right)\left(\int_s^{v_3} dX_{v_2}^j\right)  dX_{v_3}^i\;,  \notag 
\end{align}
where the error is of order $|t-s|^{4\gamma}$ and hence $\liloh(|t-s|)$ for $|t-s| \ll 1$. Now, all of the above integrals are components of $\X_{st}$. For instance, 
\begin{align*}
\int_s^t dX^i_{v_1} &= \inner{\X_{st},e_i}\quad , \quad \int_s^t \int_s^{v_1} dX_{v_2}^{j}dX_{v_1}^i = \inner{\X_{st},e_{ji}}\\ 
&\int_s^t \int_s^{v_2} \int_s^{v_1} dX_{v_3}^{k}dX_{v_2}^j dX_{v_1}^i = \innerprod{\X_{st},e_{kji}}\;.
\end{align*}
The non-trivial term must be understood using the shuffle product. Indeed, the identity \eqref{e:intro_shuffle} guarantees that 
\begin{align*}
\left(\int_s^{v_3} dX_{v_1}^k\right)\left(\int_s^{v_3} dX_{v_2}^j\right) &= \innerprod{\X_{st},e_k}\innerprod{\X_{st},e_j} = \innerprod{\X_{st},e_i\shuffle e_j}\\  &\:= \innerprod{\X_{st},e_{kj}} + \innerprod{\X_{st},e_{jk}}\;,
\end{align*}
and hence we define
\begin{equ}\label{e:vtree}
\int_s^t \left(\int_s^{v_3} dX_{v_1}^k\right)\left(\int_s^{v_3} dX_{v_2}^j\right)  dX_{v_3}^i \defin \innerprod{\X_{st},e_{kji}} + \innerprod{\X_{st},e_{jki}}\;.
\end{equ}
It should then be clear that $Y$ looks locally like $\X$, in the sense that 
\begin{equ}
\delta Y_{st} \approx \sum_{e_{i_1 \dots i_n}} F_{e_{i_1 \dots i_n}} (Y_s) \innerprod{\X_{st},e_{i_1 \dots i_n}} \;,
\end{equ}
where we sum over all basis elements $e_{i_1 \dots i_n} \in T^{(N)}(V)$ and where $F_{e_{i_1 \dots i_n}}: \reals^e \to \reals^e$ are the coefficients from \eqref{e:controlled_expansion}. One then constructs $Y$ over all of $[0,T]$ by \emph{sewing together} the increments $Y_t - Y_s$ over small intervals. The $\liloh(|t-s|)$ terms disappear as we sum over smaller and smaller intervals.

\subsection{Non-geometric rough paths}

Whereas a geometric rough path lives in a tensor product algebra generated by $V=\reals^d$, a branched rough path lives in the Hopf algebra generated by the set of rooted, labelled trees $\trees$ with vertex decorations from the set $\{1,\dots,d \}$. This space is known as the \emph{Connes-Kreimer Hopf algebra} and was famously used in \cite{connes98} in the context of renormalization theory. In general, a Hopf algebra consists of a vector space $\hopf$, equipped with a product $\cdot : \hopf \otimest \hopf \to \hopf$ and a coproduct $\Delta : \hopf \to \hopf \otimest \hopf$, see the standard textbook \cite{sweedler69}. 
As an algebra, $\hopf$ will
simply be the set of abstract polynomials, where we consider the elements of $\trees$ as commuting indeterminates.
The product $\cdot$ is then the usual (commutative) product between polynomials and the basis elements for the vector space $\hopf$ are simply all monomials in the indeterminates from $\trees$. We will frequently omit the product $\cdot$ from the notation, for instance writing $\tau_1 \tau_2$ for the product of $\tau_1$ and $\tau_2$. The coproduct $\Delta$ is the dual of a more interesting product $\star$, also known as the \emph{convolution product}. Much like the deconcatenation coproduct describes all ways of cutting apart a tensor, the coproduct $\Delta$ describes all ways of cutting apart a tree. For an introduction to Hopf algebras aimed towards the Connes-Kreimer algebra, see the monograph \cite{manchon04}. 

The following is a slight rewriting of the definition given in \cite{gubinelli10a}:

\begin{defn}
A \emph{branched rough path of regularity $\gamma$} is a map $\X:[0,T]\times[0,T]\to\hopf^*$ satisfying the following three conditions
\begin{enumerate}
\item$\innerprod{\X_{st},h_1 h_2} = \innerprod{\X_{st},h_1}\innerprod{\X_{st},h_2}\;,$ \quad \text{for every $h_1,h_2 \in \hopf$\;. }
\item$\X_{st} = \X_{su}\star \X_{ut} \quad \text{or equivalently}\quad \innerprod{\X_{st},h} = \sum_{(h)} \innerprod{\X_{su},h^{(1)}}\innerprod{\X_{ut},h^{(2)}}$\;, \\
where $\Delta h = \sum_{(h)} h^{(1)} \otimestilde h^{(2)}$ and $h\in\hopf$.
\item $\sup_{s\neq t} {|\innerprod{\X_{st},\tau}|}/{|t-s|^{\gamma |\tau|}} < \infty$\;,\\
for every $\tau \in \hopf$, where $|\tau|$ counts the number of vertices in $\tau$.
\end{enumerate}
\end{defn}

\begin{rmk}
As was the case with geometric rough paths, for branched rough paths only finitely many components $\inner{\X_{st},\tau}$ actually matter. As always, let $N$ be the largest integer such that $N\gamma \leq1$, then the components $\inner{\X_{st},\tau}$ with $|\tau| > N$ are determined by those with $|\tau| \leq N$ \cite{gubinelli10a} and moreover the components with $|\tau| >N$ never show up in expressions for solutions of differential equations.
\end{rmk}

\begin{remark}
Here, we used the notation $\otimestilde$ for elements in the tensor product of $\hopf$ with itself.
The reason for not using the standard notation $\otimes$ is because the latter will be reserved for the 
tensor product within the tensor algebra built over some vector space, as in Section~\ref{sec:geometric}.
\end{remark}

Condition $1$ confirms that the polynomial product plays the role of the shuffle product in $\hopf$. That is, it picks out some object $h_1 h_2$ so that $\innerprod{\X,h_1 h_2} = \innerprod{\X,h_1}\innerprod{\X,h_2}$. The fact that this product is commutative in both theories is a reflection of the fact that the usual product between smooth functions
is commutative.
Condition $2$ is a natural requirement of any iterated integral. Indeed, no matter how one defines an integral, it should always be linear with respect to the integrand, and satisfy $\int_s^t = \int_s^u + \int_u^t$. Condition $2$
encapsulates this identity in our context, if we interpret the components of $\X$ in the way described below. 
Condition $3$ reflects the fact that the integral $\innerprod{\X_{st},\tau}$ should be $|\tau|$ times as regular as the underlying path $X$; it is a purely analytic condition, as opposed to the first two purely algebraic conditions. 
\par 
We will now illustrate the definition with the example of $\gamma \in (1/4, 1/3]$. Here, we would have
\begin{align*}
&\int_s^t dX^i_{v_1} = \innerprod{\X_{st},\mytreeone{i}\;\;\;}\;,\;\int_s^t \int_s^{v_1} dX^j_{v_2} dX^i_{v_1} = \innerprod{\X_{st},\mytreeoneone{j}{i}\;\;\;}\;\;\\ \text{and} \;\;\;\;&\int_s^t\int_s^{v_1} \int_s^{v_2} dX^k_{v_3}dX^j_{v_2} dX^i_{v_1} = \innerprod{\X_{st},\mytreeoneoneone{k}{j}{i}\;\;\;}\;,
\end{align*}
as well as the \emph{branched} object
\begin{equ}
\int_s^t \left(\int_s^{v_3} dX_{v_1}^k\right)\left(\int_s^{v_3} dX_{v_2}^j\right)  dX_{v_3}^i = \innerprod{\X_{st},\mytreetwoone{k}{j}{i}\;\;}\;.
\end{equ}
In general, components of $\X$ should be interpreted as in Remark~\ref{rem:canonical} below.
Essentially, every node corresponds to one integration, with each incoming branch denoting
a factor of the integrand. 

In the above example, the only additional objects in our non-geometric rough path are the components corresponding to $\mytreetwoone{}{}{}$. Contrary to the case of geometric rough paths, we cannot use the integration by parts formula to simplify these further.
As $N$ increases (or $\gamma$ decreases), a branched rough path becomes much larger than a geometric rough path. For $\tau = \mytreeoneone{j}{i}\;\;$, Condition $(2)$ becomes the familiar identity for the L\'evy area
\begin{equ}
\innerprod{\X_{st},\mytreeoneone{j}{i}\;\;\;} = \innerprod{\X_{su},\mytreeoneone{j}{i}\;\;\;} + \innerprod{\X_{ut},\mytreeoneone{j}{i}\;\;\;} + \innerprod{\X_{su},\mytreeone{j}\;\;\;}\innerprod{\X_{ut},\mytreeone{i}\;\;\;}\;,
\end{equ}
or in the language of the coproduct
\begin{equ}
\Delta \mytreeoneone{j}{i}\;\;\; = \mytreeoneone{j}{i}\;\;\otimestilde1 + 1\otimestilde\;\mytreeoneone{j}{i}\;\;\; + \mytreeone{j}\;\;\otimestilde\;\mytreeone{i}\;\;\;.
\end{equ}
Let us again consider the solution to \eqref{e:intro_SDE}, now driven by a branched rough path $\X$ with $1/4 < \gamma \leq 1/3$. From \eqref{e:controlled_expansion}, we would have
\begin{equ}\label{e:intro_controlled}
Y_t - Y_s \approx \sum_{\tau} f_{\tau}(Y_s)\innerprod{\X_{st},\tau}
\end{equ}
where we sum over all $\tau \in \trees_3$, or in the case of arbitrary $\gamma$, all $\tau \in \trees_N$, the set of $\tau \in \trees$ with $|\tau|\leq N$. Hence, the idea of viewing the solution to \eqref{e:intro_SDE} as an object that locally ``looks like" $\X$ carries through nicely to the framework of non-geometric rough paths. The coefficients $f_\tau$ are known as the \emph{Butcher coefficients}, in honour of J. Butcher who was the first to represent solutions to ODEs as a series indexed by trees, which turned out to be a very fruitful approach to the development of 
numerical methods for the
solutions to ODEs \cite{butcher72,hairer74,MR2657947}.

\subsection{Converting non-geometric to geometric}

The main objective of the article is to provide a translation between branched rough paths and geometric rough paths. The first step is to rephrase branched rough paths in the language of geometric rough paths. For a geometric rough path, Chen's property is not a definition, but is a corollary from the definition $\X_{st} = \X_{s}^{-1}\otimes \X_t$. However, for a branched rough path, this is considered part of the definition. We will show that a branched rough path can equivalently be defined as a a path $\X: [0,T] \to G_N$, where $(G_N,\star)$ is the (truncated) Lie group of \emph{characters} in the Connes-Kreimer Hopf algebra, satisfying
\begin{equ}
\innerprod{g,xy} = \innerprod{g,x}\innerprod{g,y} \;, 
\end{equ}
for all $x,y \in \hopf$. This allows us to define $\X_{st} = \X_{s}^{-1}\star \X_t$ and hence guarantee Chen's property from the definition. The Lie group $(G_N,\star)$ bears great similarity to the step $N$ free nilpotent group, since it is the truncated set of characters in $\hopf$, and the step $N$ free nilpotent group is the truncated set of characters in the tensor product algebra $T(V)$. Moreover, one obtains $G_N$ as the $\expstar$ of the Lie algebra of so-called primitive elements, where $\expstar$ is simply the tensor exponential, with tensor products replaced with $\star$ products.  
\par
Unsurprisingly, it is easy to show that a geometric rough path is a type of branched rough path. The main result of the article provides a surprising converse statement, namely that every branched rough path over a path can be \emph{encoded} in a geometric rough path. More precisely, for any branched rough path $\X$ above $X$ there exists a geometric rough path $\Xbar$ above $\xbar$, where $\xbar$ is an extension of $X$ and $\Xbar$ contains all the information held in $\X$. 
\par
The path $\xbar$ will take values in $\B_N$, where we define $\B_n$ as the real vector space spanned by the set $\trees_n$. Clearly, one can think of $X$ as taking values in 
\begin{equ}
\B_1 \defin \spn \{\bullet_i : i=1\dots d \} \cong \reals^d\;.
\end{equ}
Under this interpretation, $\xbar$ is an extension of $X$ in the sense that $\pi_{\B_1}(\xbar)=X$, where $\pi_V$ denotes projection onto $V$. The geometric rough path $\Xbar$ lives in the truncated tensor product space 
\begin{equ}
T^{(N)}(\B_N)  = \spn \{\tau_1 \otimes \dots \otimes \tau_n : \text{$\tau_i \in \trees_N$ and $1\leq n \leq N$} \} \;.
\end{equ}
Thus, since $\tau$ is a basis vector of the underlying vector space $\B_N$, the object $\inner{\Xbar_{st},\tau}$ will actually denote a \emph{path} component of $\Xbar$, in that \begin{equ}
\inner{\Xbar_{st},\tau} =\delta \xbar^{\tau}_{st}\;,
\end{equ}
for all $\tau\in\trees_N$, as opposed to the original $\inner{\X_{st},\tau}$ which must be interpreted as a integral component, indexed by the tree $\tau$. Moreover, the tensor components must be interpreted as the iterated integrals
\begin{equ}
\inner{\Xbar_{st},\tau_1 \otimes \dots \otimes \tau_n} \defin \int_s^t \dots \int_s^{v_2} d\xbar^{\tau_1}_{v_1} \dots d\xbar^{\tau_n}_{v_n}\;.
\end{equ}
We will prove the following result. As always, $\gamma \in (0,1)$ and $N$ is the largest integer such that $N\gamma \leq 1$. 

\begin{thm}\label{thm:intro1}
Let $X=(X^i)_{i=1\dots d}$ be a path in $\reals^d$ and $\X$ a $\gamma$-H\"older branched rough path in $\hopf$ such that $\innerprod{\X_{st},\bullet_{i}} = \delta X^i_{st}$. Then there exists 
\begin{enumerate}
\item a path $\xbar = (\xbar^\tau)_{\tau \in \trees_N}$ taking values in $\B_N$, with $\pi_{\B_1}(\xbar) = X$\;,
\item a $\gamma$-H\"older geometric rough path $\Xbar$ in $T^{(N)}(\B_N)$ satisfying
$\innerprod{\Xbar_{st},\tau} = \delta \xbar^\tau_{st}$ for each $\tau \in \trees_N$
and
\item a graded morphism of Hopf algebras $\psi:\hopf \to T(\B_N)$\;,
\end{enumerate}
such that
\begin{equ}\label{e:intro_psi}
\innerprod{\X_{st},h} = \innerprod{\Xbar_{st},\psi(h)}\;,
\end{equ}
for every $h \in \hopf$.    
 \end{thm} 
 Before adding a few remarks, we will illustrate the result with the first non-trivial example.
 \begin{example}
 Consider the case where $X \in \reals^d$ with H\"older exponent $1/3 < \gamma \leq 1/2$, so that $N=2$. The important components of the branched rough path $\X$ above $X$ are $\inner{\X,\bullet_i}$ and $\inner{\X,\ntreeoneone_j^k}$, for all $i,j,k=1\dots d$. The theorem tells us that there exists a path
 \begin{equ}
 \xbar = (\xbar^{\bullet_i}, \xbar^{{\treeoneone_j^k}})_{i,j,k=1\dots d} 
 \end{equ} 
 where $\xbar^{\bullet_i} = X^i$ for all $i=1\dots d$ and moreover there exists a geometric rough path $\Xbar$ above $\xbar$. Since 
 \begin{equ}
 \B_2 \defin \spn\{\bullet_i, \ntreeoneone_j^k : i,j,k=1\dots d\}\;,
 \end{equ}
we can see that $\Xbar$ is defined on the (truncated) tenor product space $\B_2 \oplus \B_2^{\otimes 2}$. The map $\psi$ tells us how to write $\X$ in terms of $\Xbar$, for instance we have $\psi(\bullet_i) = \bullet_i$ and $\psi(\ntreeoneone_i^j) = \bullet_j\otimes \bullet_i + \ntreeoneone_i^j$ and therefore
 \begin{equ}
 \inner{\X_{st},\bullet_i} = \inner{\Xbar_{st},\bullet_i} \quad\text{and}\quad \inner{\X_{st},\ntreeoneone_i^j} = \inner{\Xbar_{st},\bullet_j\otimes \bullet_i + \ntreeoneone_i^j}\;.
 \end{equ}
Or in the more formal language
\begin{equ}\label{e:intro_eg_integrals}
\delta X^i_{st} = \delta \xbar^{\bullet_i}_{st} \quad\text{and}\quad \int_s^t \int_s^{v_1} dX^j_{v_2} dX^i_{v_1} =  \int_s^t \int_s^{v_1}  d\xbar^{\bullet_j}_{v_2}  d\xbar^{\bullet_i}_{v_1} + \delta \xbar^{\treeoneone^j_i}_{st}\;,
\end{equ}
for all $i,j,k=1\dots d$. Note that even though $X^i = \xbar^{\bullet_i}$, the integrals defined on the left hand and right hand side of the second equality in \eqref{e:intro_eg_integrals} are \emph{different}, since  the one on the left is defined by $\X$ and the one on the right is defined by $\Xbar$.    
\end{example}
 
This result relies on the Lyons-Victoir extension theorem of \cite{lyons07}, which shows that every $\gamma$-H\"older path in a quotient of the free nilpotent group $G^{(N)}(V)$ can be extended to a $\gamma$-H\"older path in $G^{(N)}(V)$. Since the extension theorem of \cite{lyons08} is non-unique, the path $\Xbar$ is also non-unique. Moreover, there is a great deal of redundancy in $\Xbar$, since it has many more components than $\X$, however, this is the most \emph{convenient} way to build a geometric rough path containing all the information of $\X$. The map $\psi$ describes how the components of $\X$ should be split up amongst the components of the tensor product algebra $T^{(N)}(\B_N)$. As we shall see, the fact that $\psi$ is a Hopf algebra morphism is crucial not only when obtaining $\Xbar$, but also when applying \eqref{e:intro_psi} further down the line. 

\begin{remark}
In \cite{lejay06}, the authors consider non-geometric rough path to be geometric rough paths without the assumption of satisfying the shuffle product relation. They show that these non-geometric rough paths are in fact isomorphic to a special class of geometric rough paths, known as $(p,q)$-rough paths, living above a path in an extended space. Hence, our result is an extension of this result, in the sense that the more general (and more useful) branched rough paths can also be encoded in a geometric rough path living above a path in an extended space. Note however that our result does not yield an isomorphism.     
\end{remark}

The main motivation behind Theorem \ref{thm:intro1} is that it allows us to rewrite an expression controlled by a branched rough path as an expression controlled by a geometric rough path. In particular, we can use this to show that every RDE driven by a branched rough path can be rewritten as another RDE driven by a geometric rough path. 
\begin{thm}[Generalised It\^o-Stratonovich correction]\label{thm:intro2}
Let $Y$ solve \eqref{e:intro_SDE}, driven by a branched rough path $\X$. Let $\xbar$ and $\Xbar$ be as defined in Theorem \ref{thm:intro1}. Then $Y$ is also a solution to
\begin{equ}\label{e:intro_conversion}
dY_t = \sum_{\tau \in \trees_N} f_\tau (Y_t)d\xbar^{\tau}_t \;,
\end{equ}
driven by the geometric rough path $\Xbar$, where the vector fields $f_{\tau}$ are defined by \eqref{e:recurrence} with $f_{\bullet_i} = f_i$ (and can be seen, for example in \eqref{e:intro_controlled}). 
\end{thm} 

\begin{example}
Returning back to the $1/4 < \gamma \leq 1/3$ example, if $Y$ solves \eqref{e:intro_SDE} driven by some $\X$ then we also have
\begin{align*}
dY_t &= f_i(Y_t)d\xbar^{\bullet_i}_t + (f_i^\alpha \del^\alpha f_j)(Y_t) d\xbar^{\mytreeoneone{j}{i}}\;\; + (f_k^\alpha \del^\alpha f_j^\beta \del^\beta f_i)(Y_t)d\xbar^{\mytreeoneoneone{k}{j}{i}}\\
&+ \frac{1}{2}(f_{k}^\alpha f_j^\beta \del^{\alpha} \del^{\beta} f_i)(Y_t) d\xbar^{\mytreetwoone{k}{j}{i}}\;,
\end{align*}
driven by the geometric rough path $\Xbar$ found in Theorem \ref{thm:intro1}, where we sum over all $i,j,k=1 \dots d$ and $\alpha,\beta = 1 \dots e$, noting that $\xbar^{\mytreetwoone{k}{j}{i}} = \xbar^{\mytreetwoone{j}{k}{i}}$. Even though $\xbar^{\bullet_i} = X^i$, one must distinguish between $f_i(Y_t) dX^i_t$  and $f_i(Y_t) d\xbar_t^{\bullet_i}$, since the former is driven by $\X$ and the latter is driven by $\Xbar$.   
\end{example}
\begin{rmk}
Although we call this a generalised It\^o-Stratonovich correction, it is really more like a ``Any non-geometric integral''-``Particular class of geometric integral'' correction. However, we are quite justified in giving it this name. Suppose $X$ was a non semi-martingale path for which there exists a branched rough path $\X$ above it and also some kind of ``Stratonovich'' rough path $\Xbar^{(1)}$ above it, fractional Brownian motion with Hurst parameter $H> 1/4$ being a good example \cite{coutin02}. As will be clear in the proof of Theorem \ref{thm:intro1}, we can actually choose $\Xbar$ such that the components above $X$ are given by $\Xbar^{(1)}$ (or indeed any geometric rough path above $X$). Hence, the formula can tell us what correction we get if we take an RDE driven by $\X$ and rewrite it using ``Stratonovich'' integrals, just as in the usual It\^o-Stratonovich correction formula.    
\end{rmk}

The outline of the article is as follows. In Section \ref{sec:hopf} we define the algebraic concepts underlying branched rough paths, including the Connes-Kreimer Hopf algebra. We then provide a definition of branched rough paths, equivalent to that given in \cite{gubinelli10a}, that is more in line with the concept of a geometric rough path. In Section \ref{sec:controlled}, we define solutions to RDEs driven by branched rough paths, via the idea of controlled rough paths.  In Section \ref{sec:geometric}, we first recall the definition of a geometric rough path. We then show that geometric rough paths fit easily in to the framework of branched rough paths, before providing a proof of Theorem \ref{thm:intro1}. In Section \ref{sec:RDEnewbasis}, we discuss the special case of RDEs driven by geometric rough paths, before proving the generalised It\^o-Stratonovich correction formula.

\subsection*{Acknowledgements}

{\small
We would like to thank M. Gubinelli, R. Hudson and J. Jones for their contributions to several fruitful discussions. Financial support for MH was kindly provided by EPSRC grant EP/D071593/1, by the 
Royal Society through a Wolfson Research Merit Award, and by the Leverhulme Trust through a Philip Leverhulme Prize.
DK was supported by a Warwick Postgraduate Research Scholarship.
}

\section{Hopf algebras and branched rough paths}\label{sec:hopf}
\subsection{Hopf algebras for probabilists}\label{subsec:hopf_prob}
In this subsection we will give a non-specialist outline of what a Hopf algebra is and why it is a useful concept. For a more detailed introduction, we recommend the notes \cite{manchon04,brouder04} as well as the standard texts \cite{sweedler69,abe80}.
\par
A Hopf algebra is a special kind of bialgebra, so we will first define the latter. A bialgebra arises naturally when one algebra is in some sense \emph{acting} on another. To this end, let $H$ be a vector space and let $H^*$ be another vector space, acting linearly on $h$ via the pairing $\inner{\cdot,\cdot} : H^* \otimest H \to \reals$. Suppose moreover that $H$ is actually an algebra, with some product $\cdot : H \otimest H \to H$ and unit element $\unit$. In many natural situations, the space $H^*$ is also an algebra, with some other product $\star : H^*\otimest H^* \to H^*$ and a \emph{counit} $\counit$, which acts as the dual element of $\unit$. 
\par
It is often advantageous to superimpose the structure from $H^*$ onto $H$, so that we simply have a vector space $H^*$ acting on a more structured space $H$. To be precise, the product $\star$ can be encoded into $H$ by a map $\Delta : H \to H\otimest H$ called a \emph{coproduct}.  The coproduct is the \emph{dual} of $\star$ in the sense that
\begin{equ}\label{e:hopf_coproduct}
\inner{f\star g , h} = \inner{f\otimest g , \Delta h}\;, 
\end{equ}
for every $f,g \in H^*$ and $h\in H$. In other words, the action of $f\star g$ on $h$ is determined by the action of $f\otimest g$ on the coproduct of $h$. We will often use the notation
\begin{equ}
\Delta h = \sum_{(h)} h^{(1)}\otimest h^{(2)}\;,
\end{equ}
and in the sequel we will occasionally omit the summation notation. In this notation \eqref{e:hopf_coproduct} can be written
\begin{equ}
\inner{f\star g ,h } = \sum_{(h)} \inner{f,h^{(1)}}\inner{g,h^{(2)}}\;.
\end{equ}
The triple $(H,\cdot,\Delta)$ is then called a \emph{bialgebra}, provided certain consistency relations between the product and coproduct are satisfied.   
\begin{rmk}
Recall that, although both $\otimes$ and $\otimestilde$ are tensor products, we reserve the former for the product in the tensor product algebra $T(V)$ and the latter simply to discriminate between the left and the right part of a coproduct. If $x,y$ are two elements in some algebra and $f,g$ are two maps on that algebra, then we use the convention $(f\otimestilde g) (x\otimestilde y) = f(x)\otimestilde g(y)$.   
\end{rmk}
Suppose that some $f\in H^*$ has an \emph{inverse} $f^{-1} \in H^*$, satisfying $f \star f^{-1} = f^{-1} \star f = \counit$, of course there is always at least one element in $H^*$ with an inverse. Since we want all the structure of $H^*$ to be contained in $H$, we must encode an inverse map into $H$. In fact, we introduce a map $\S : H \to H$ such that $\S^* : H^* \to H^*$ is the inverse map, satisfying $\S^* f \star f = f\star \S^* f = \counit$. The map $\S$ is called the \emph{antipode}. But since we only want to work on $H$ and not $H^*$, the dual requirement for $\S$ is that 
\begin{equ}
(\Id \otimest \S) \Delta h = (\S \otimest \Id) \Delta h = \inner{\counit,h}\unit\;,
\end{equ}  
for all $h \in H$, where $\Id : H \to H$ is the identity map. The quadruple $(H,\cdot,\Delta,\S)$ is called a \emph{Hopf algebra}. Thus, a Hopf algebra is nothing more than a bialgebra with an antipode.  
\par
A bialgebra is called \emph{graded} if it can be decomposed into a direct sum of vector spaces
\begin{equ}
H = \bigoplus_{n \in \naturals} H_{(n)}\;,
\end{equ}
satisfying the natural multiplication and comultiplication rules
\begin{equ}
H_{(n)} \cdot H_{(m)} \subset H_{(n+m)} \quad \text{and} \quad \Delta H_{(n)} \subset \bigoplus_{p+q=n} H_{(p)} \otimest H_{(q)}\;,
\end{equ}
for any $n\in\naturals$. A graded Hopf algebra must satisfy the additional property
\begin{equ}
\S H_{(n)} \subset H_{(n)}\;,
\end{equ}
for any $n\in\naturals$. For any graded bialgebra, one can define a map $|\cdot|$ whose domain is given
by some ``natural'' basis elements of $H$, and which simply reads off the index $n$ of the space $H_{(n)}$ in which the basis element lives.  
\par
A standard result in Hopf algebra theory states that every graded bialgebra $H$ satisfying $H_0 = \reals$ is in fact a Hopf algebra. That is, one can find an antipode for $H$. Moreover, every Hopf algebra has a \emph{unique} antipode. See \cite{abe80,dascalescu01} for details. To round off this subsection, we will give a simple example of a Hopf algebra. A more detailed exposition of this example can be found in \cite{brouder04}. 
\begin{example}[The algebra of differential operators] Consider the differential operator $\del_i = \del/\del x_i$ for $i=1\dots d$. The set $\{\del_i \}_{i=1}^d$ generates an algebra $H$, where multiplication is given by composition of the operators and the unit $\unit$ is given by the identity operator. To turn $H$ into a Hopf algebra, we must find a coproduct and an antipode. As stated above, coproducts arise naturally when an algebra $H^*$ is acting linearly on $H$. To this end, let $H^*$ be the space of smooth function $f : \reals^d \to \reals$ and define the pairing 
\begin{equ}
\inner{f,D} = (D f)(0)\;,
\end{equ}  
for any $D \in H$. The space of smooth functions $H^*$ can be turned into an algebra by introducing pointwise multiplication $\star$, and the counit $\counit$ is simply the constant function $f=1$. The coproduct $\Delta$ arises when we consider the action of the product $f\star g$ on a differential operator $D \in H$. For instance, Leibniz rule tells us that 
\begin{align*}
\inner{f\star g , \del_{i} \del_{j} } = (\del_{i} \del_{j} (f g) )(0) &=  \del_{i} \del_{j} f(0)g(0) +\del_{i}f(0) \del_{j}g(0)  + \del_{j}f(0) \del_{i}g(0) + f(0)\del_{i} \del_{j} g(0) \\
&= \inner{f \otimest g , \del_{i} \del_{j} \otimest 1 + \del_{i} \otimest \del_{j} + \del_{j} \otimest \del_{i} + 1 \otimest \del_{i} \del_{j} }\;.
\end{align*}
Hence, we can encode the action of $f\star g$ on $\del_i \del_j$ using the coproduct
\begin{equ}
\Delta (\del_i \del_j) = \del_{i} \del_{j} \otimest 1 + \del_{i} \otimest \del_{j} + \del_{j} \otimest \del_{i} + 1 \otimest \del_{i} \del_{j}\;.
\end{equ}
Of course, one can use this same technique to decide how to define $\Delta (\del_{i_1} \dots \del_{i_n})$. Moreover, it is an easy exercise to check that $\S 1 = 1$,  $\S \del_i = - \del_i$ and more generally $\S (\del_{i_1} \dots \del_{i_n}) = (-1)^n \del_{i_1} \dots \del_{i_n}$ defines an antipode on $H$.  
\end{example}

\subsection{The Connes-Kreimer Hopf algebra}
In this subsection we will define another important example of a Hopf algebra, called the \emph{Connes-Kreimer Hopf algebra}, which is a critical object in the theory of branched rough paths.
\par 
Let $\trees$ be the set of all rooted trees with finitely many vertices, whose vertices are decorated by labels from the alphabet $\{1,\dots,d \}$. Every element in $\trees$ can be constructed recursively by attaching a collection of trees (of lower order) to a new root. For example, the 
set of (undecorated) trees with three vertices or less is given by 
\begin{equ}
\trees_3 = \{\mytreeone{}, \mytreeoneone{}{} , \mytreeoneoneone{}{}{}, \mytreetwoone{}{}{}\;\}\;.
\end{equ}
We can then construct all single vertex trees by attaching the empty tree $1$ to a new root. We denote this by
\[
[1]_a = \mytreeone{a}\notag\;\;\;,
\]
for any $a$ from the alphabet. All trees of two vertices can be constructed by attaching these trees to a new root
\[
[\mytreeone{a}\;\;]_b = \mytreeoneone{a}{b}\;\;\;.
\]
For the trees of three vertices, we similarly have
\[
[\;\;\mytreeoneone{a}{b}\;\;\; ]_c\;\;= \;\; \mytreeoneoneone{a}{b}{c}\;\;\;.
\]
The remaining tree in $\trees_3$ is obtained by attaching a pair of single vertex trees to a root
\[
[\;\mytreeone{a}\;\;\;\mytreeone{b}\;\;\;]_c = \mytreetwoone{a}{b}{c}\;\;\;.
\]
Indeed, every element in $\trees$ can be written recursively as 
\begin{equ}\label{e:tree_prod}
[\tau_1 \tau_2\dots\tau_m]_a\;,
\end{equ}
for some smaller trees $\tau_1,\dots,\tau_m \in \trees \cup \{ 1\}$ and some $a$ from the alphabet. We will always assume that the order of the branches in each tree does not matter, in the sense that $[\tau_1 \dots \tau_n]_i = [\tau_{\sigma(1)} \dots \tau_{\sigma(n)}]_i$ for all permutations $\sigma$ of $\{1,\dots,n\}$. For each $[\tau_1 \dots \tau_n]_i$, only one such representation appears in the set $\trees$. 
\begin{remark}
In the rough path setting, rearranging branches in a tree corresponds to rearranging real-valued factors in an integrand. Hence, this is quite a natural assumption to make. 
\end{remark}
The \emph{Connes-Kreimer Hopf algebra} $(\hopf,\cdot,\Delta,\S)$ is the commutative polynomial algebra generated by the variables $\trees$, equipped with a coproduct $\Delta: \hopf \to \hopf \otimestilde \hopf$ and an antipode $\S:\hopf \to \hopf$. Alternatively, we can view the set $\hopf$ as a real vector space whose basis is the commutative monoid $\forest\cup \{1\}$ where $\forest$ is given by
\begin{equ}
\forest = \{ \tau_1 \dots \tau_n : \tau_i \in \trees, \;n\in \naturals^+ \}\;.
\end{equ}
Each monomial $\tau_1 \dots \tau_n$ can be thought of as an \emph{unordered forest}, since the polynomial product is commutative. Hence, a typical element of $\hopf$ is for example
\begin{equ}
\mytreeoneoneone{1}{2}{3}\;\;\; + 6\;\mytreeone{3}\;\;\;\mytreeoneone{1}{2}\;\; - \sqrt{2}\;\mytreeone{3}\;\;\;\mytreetwoone{3}{2}{1}\;.
\end{equ}

\begin{rmk}
We could equally construct the Connes-Kreimer Hopf algebra $\hopf(\alphabet)$, using any countable alphabet $\alphabet$ in place of $\{1,\dots,d\}$. However, since $\{1,\dots,d\}$ is the most commonly used choice, we reserve the notation $\hopf$ for this particular alphabet. 
\end{rmk}

The coproduct $\Delta$ is defined recursively. We first set $\Delta1 =  1\otimest 1 $, then for any $[\tau_1 \dots \tau_m]_a \in \trees$ we set
\begin{equ}\label{e:recursive_coproduct}
\Delta [\tau_1 \dots \tau_m]_a =  [\tau_1 \dots \tau_m]_a\otimestilde 1 + \sum_{(\tau_1) \dots (\tau_m)}(\tau_1^{(1)}  \dots \tau_m^{(1)})\otimestilde [\tau_1^{(2)}  \dots \tau_m^{(2)}]_a\;,
\end{equ}
where we use the Sweedler notation $\Delta x = \sum_{(x)} x^{(1)}\otimestilde x^{(2)}$. In the sequel, we will often omit the summation sign and simply write $\Delta x = x^{(1)}\otimestilde x^{(2)}$. In Remark \ref{rmk:cuts}, we will see that the coproduct $\Delta$ has a nice combinatorial interpretation when restricted to trees. We then extend $\Delta$ to all polynomials by requiring that it be linear and also a morphism with respect  to polynomial multiplication, that is
\begin{equ}
\Delta (\tau_1 \dots \tau_n) = \Delta \tau_1 \dots \Delta \tau_n\;,
\end{equ} 
for every $\tau_i\in\trees$. It is often useful to consider the \emph{reduced coproduct} $\Delta'$ defined by $\Delta ' x = \Delta x - 1\otimestilde x - x\otimestilde 1$. In any coalgebra, the coproduct is required to be \emph{coassociative}, which means that
\begin{equ}
(\Delta \otimestilde \Id) \Delta = (\Id \otimestilde \Delta) \Delta\;.
\end{equ}
One can check that this is true for both the coproduct and the reduced coproduct described above.
\par
In any Hopf algebra, the \emph{antipode} $\S : \hopf \to \hopf$ is a morphism of bialgebras satisfying 
\begin{equ}
M(\Id \otimestilde \S) \Delta x = M(\S \otimestilde \Id)\Delta x = x  \;,
\end{equ}
for any $x\in\hopf$, where $M$ is the multiplication map $M(x\otimestilde y) = xy$. The existence of an antipode for $\hopf$ follows from the fact that $\hopf$ is actually a graded bialgebra, we will define this grading below. For the Connes-Kreimer Hopf algebra the antipode has been explicitly constructed in \cite{connes98}. 
\par
The Hopf algebra $(\hopf,\cdot,\Delta,\S)$ gives rise to a dual Hopf algebra $(\hopf^*,\star,\delta,\S^*)$. Since $\hopf$ is a countable vector space, the elements in the topological dual $\hopf^*$ can be identified with formal series of elements in $\hopf$. In particular, we identify elements in the basis $\forest$ with elements in $\hopf^*$ by the natural pairing $\innerprod{h_1,h_2} = \delta_{h_1,h_2}$ for $h_1, h_2 \in \forest$. The co-unit $\counit\in \hopf^*$ is the map satisfying $\inner{\counit,1}=1$ and $\inner{\counit,\tau_1 \dots \tau_n} = 0$ for all $\tau_1 \dots \tau_n \in \forest$.  
\begin{rmk}
In the sequel, our notation will not distinguish between the unit and the co-unit, nor the basis $\forest$ and its dual elements $\forest^*$ (and likewise $\trees$ and $\trees^*$). However, it will always be clear from the context which we are referring to.
\end{rmk}
The product $\star: \hopf^* \otimestilde \hopf^* \to \hopf^*$, often referred to as \emph{convolution}, is the dual of $\Delta$, that is    
\begin{equ}
\innerprod{f\star g,  h} \defin \innerprod{f \otimestilde g , \Delta h} = \sum_{(h)} \innerprod{f,h^{(1)}}\innerprod{g,h^{(2)}}  \;,
\end{equ}
for any $f,g \in \hopf^*$ and $h\in\hopf$. It follows from the properties of the coproduct $\Delta$ (namely, coassociativity) that $\star$ provides $\hopf^*$ with an associative algebra structure. Let $\trees^*$ denote those elements in $\hopf^*$ that correspond to dual elements of $\trees$. Then for $\tau_1,\tau_2 \in \trees^*$, the product $\tau_1\star\tau_2$ can be interpreted as attaching $\tau_1$ to $\tau_2$. In particular, we have that
\begin{equ}\label{e:GL}
\tau_1 \star \tau_2 = \tau_1 \tau_2 + \tau_1 \star_t \tau_2\;,
\end{equ}
where $\tau_1 \star_t \tau_2$ is the sum of all trees in $\trees^*$ obtained by growing $\tau_1$ from a vertex of $\tau_2$. For example,
\begin{equ}
\mytreeone{a}\;\; \star_t \mytreeoneone{b}{c}\;\; = \mytreeoneoneone{a}{b}{c}\;\; + \mytreetwoone{a}{b}{c}\;.
\end{equ}
This is often referred to as the \emph{Grossman-Larson product}, and was first discussed in \cite{larson89}. The antipode $\S$ plays the role of an inverse with respect to $\star$ in the space $\hopf^*$, precisely as stated in Subsection \ref{subsec:hopf_prob}. The dual coproduct $\delta : \hopf^* \to \hopf^* \otimestilde \hopf^*$ is likewise the dual of polynomial multiplication
\begin{equ}
\innerprod{\delta \tau , h_1\otimestilde h_2} = \innerprod{\tau, h_1 h_2}\;.
\end{equ}
Just as above, this endows $\hopf^*$ with a coassociative coalgebra structure and it is a nice exercise to check that $\delta$ is a morphism with respect to $\star$, as every coproduct should be.  
\par
The trees $\trees$ give rise to a natural \emph{grading} on $\hopf$. For each $\tau \in \trees$, we define $|\tau|$ to be the number of vertices in $\tau$. We extend $|\cdot|$ to all of $\forest$ by 
\begin{equ}
|\tau_1 \dots \tau_n| = |\tau_1| + \dots +  |\tau_n|\;,
\end{equ}
for any $\tau_i \in \trees$. If we let $\forest_{(k)}$ denote the set of $\tau_1 \dots \tau_m \in\forest$ with $|\tau_1 \dots \tau_m|=k$ and $\hopf_{(k)}$ denote the real vector space spanned by $\forest_{(k)}$, with $\hopf_{(0)} = \reals$, then we clearly have
\begin{equ}
\hopf = \bigoplus_{k=0}^\infty \hopf_{(k)}\;.
\end{equ}
One can easily check that this satisfies the right consistency conditions to ensure $\hopf$ is a graded Hopf algebra. We will also make use of the truncated algebra 
\begin{equ}
\hopf_n = \bigoplus_{k=0}^n \hopf_{(k)}
\end{equ}
and its basis elements $\forest_n$, containing all $m\in\forest$ with $|m|\leq n$. Keeping in line with this notation, we also define $\trees_{(n)}$ as the set of $\tau \in \trees$ with $|\tau|=n$ and $\trees_n$ as the set of $\tau\in\trees$ with $|\tau| \leq n$. Likewise, we denote by $\B$,  $\B_n$ and $\B_{(n)}$ the real vector spaces spanned by $\trees$, $\trees_n$ and $\trees_{(n)}$, respectively. 

\begin{rmk}
It is natural to ask why one needs to consider polynomials of $\trees$ rather than just the set of trees. Indeed, for non-geometric rough paths, the trees are the important ingredients when solving an RDE. The reason we require polynomials is that we would like to define a rough path as a functional on some algebra, and this algebra must be big enough to include an element $h_1 h_2$ such that 
\begin{equ}
\innerprod{\X,h_1}\innerprod{\X,h_2} = \innerprod{\X,h_1 h_2}\;.
\end{equ}
This, in particular, allows us to write Chen's property as a fundamental operation on the algebra $\hopf$, described by the coproduct $\Delta$, rather than just an identity on the tree indexed components of $\X$.
\end{rmk}

\begin{remark}\label{rmk:tensor_identify}
Let $\B$ be the vector space spanned by the set of trees $\trees$. This is clearly a subspace of $\hopf$. The tensor product algebra $T(\reals^d)$ can easily be identified with the subspace of $\B$, and hence $\hopf$, spanned by the linear trees. This is achieved by identifying
\begin{equ}
e_{a}\otimes e_b \cong \mytreeoneone{a}{b}\;\; \quad \text{,} \quad e_a \otimes e_b \otimes e_c \cong \mytreeoneoneone{a}{b}{c}\;\;\;,  
\end{equ} 
and so forth, for any $a,b,c = 1\dots d$. In the sequel, we will refer to this identification via the inclusion map $\iota: T(\reals^d) \to \hopf$. In light of this, we should think of the Hopf algebra $\hopf$ as being an extension of the tensor product algebra over the same index set. As discussed in the introduction, the extra \emph{branched} objects are required to encode a non-trivial product that cannot be described by objects in the tensor product algebra alone.    
\end{remark}

\begin{rmk}\label{rem:canonical}
The definition \eqref{e:recursive_coproduct} is indeed quite a natural one. If $X$ were a smooth path in $\reals^d$ then we could build the branched rough path $\X$ canonically, by setting 
\begin{equ}
\innerprod{\X_{st} ,\mytreeone{i}\;\;} = \delta X^i_{st}\quad \text{and}\quad \innerprod{\X_{st},[\tau_1 \dots \tau_n]_i} = \int_s^t \innerprod{\X_{sr},\tau_1}\dots\innerprod{\X_{sr},\tau_n}dX^i_{r}\;.
\end{equ}
Using the properties of a path integral, namely, linearity with respect to the integrand and the adjacent interval property $\int_s^t = \int_s^u + \int_u^t$, one can recursively show that
\begin{align*}
\innerprod{\X_{st},[\tau_1 \dots \tau_n]_i} &= \innerprod{\X_{su},[\tau_1 \dots \tau_n]_i }\\ &+  \sum_{(\tau_1) \dots (\tau_n)}\big(\innerprod{\X_{su},\tau_1^{(1)}} \dots \innerprod{\X_{su}, \tau_n^{(1)}}\big) \int_u^t \innerprod{\X_{ur},\tau_1^{(2)}} \dots \innerprod{\X_{ur}, \tau_n^{(2)}}dX_{r}^i
\end{align*}
or in other words,
\begin{equ}\label{e:rmk_chen}
\innerprod{\X_{st},[\tau_1 \dots \tau_n]_i} = \innerprod{\X_{su}\otimestilde \X_{ut},\Delta [\tau_1 \dots \tau_n]_i}\;,
\end{equ}
with $\Delta$ satisfying \eqref{e:recursive_coproduct}. Hence, \eqref{e:rmk_chen} is an extension of Chen's property to more complicated looking integrals. 
\end{rmk}
\begin{rmk}\label{rmk:cuts}
When restricted to linear trees (or tensor products), the coproduct $\Delta$ is known as \emph{deconcatenation}, since it decomposes tensors into subtensors that can be concatenated into the original expression. There is a similar interpretation for $\Delta$ on all of $\hopf$, which is described by \emph{cuts} of a tree. We will say that the pair $(\tau_1 \dots \tau_m)\otimestilde \tau_0 $ is an admissible cut of $\tau \in \trees$, if one can obtain $\tau$ by \emph{attaching} the trees $\tau_1, \dots, \tau_m$ to the nodes of $\tau_0$. We then have the interpretation
\begin{equ}
\Delta \tau = \sum_{(\tau)} \tau^{(1)}\otimestilde \tau^{(2)}\;,
\end{equ}
where we sum over all admissible cuts $\tau^{(1)}\otimestilde \tau^{(2)}$, with $\tau^{(1)}$ and $\tau^{(2)}$ playing the roles of $(\tau_1 \dots \tau_m)$ and $\tau_0$ respectively. For example, we have that
\begin{equ}
\Delta \mytreetwoone{a}{b}{c}\;\;\; = 1\otimestilde \mytreetwoone{a}{b}{c}\;\;\; +    (\mytreeone{a}\;\;\;\mytreeone{b}\;\;\;)\otimestilde \mytreeone{c}\;\;\;+ \mytreeone{a}\;\;\;\otimestilde \mytreeoneone{b}{c}\;\;\;+\mytreeone{b}\;\;\;\otimestilde \mytreeoneone{a}{c}\;\;\; + \mytreetwoone{a}{b}{c}\otimestilde 1\;.
\end{equ}
In particular, we always have that $\Delta \tau = 1 \otimest \tau + \tau\otimest 1 + \tau^1 \otimest \tau^2$, where $\tau^1 \otimest \tau^2$ is shorthand for the sum over all non-trivial admissible cuts of $\tau$. In the sequel, we will frequently omit the sum in the fashion. Each term $\tau^1 \otimest \tau^2 \in \forest\otimest \trees$ and we have that $|\tau^1| + |\tau^2| = |\tau|$, recalling that $|\cdot|$ simply counts the number of vertices in a forest or tree. This observation will be crucial in the sequel.
\end{rmk}

\subsection{Group-like and primitive elements}\label{subsec:grouplike}
We will denote by $\hom(\hopf,\reals)$ those elements in $\hopf^*$ that are also homomorphisms with respect to polynomial multiplication $\cdot$, that is, $f \in \hom(\hopf,\reals)$ if and only if 
\begin{equ}\label{e:homo}
\innerprod{f,h_1 h_2 } = \innerprod{f,h_1}\innerprod{f,h_2}\;.
\end{equ}
These are also known as the \emph{characters} of $\hopf$. It is easy to check that $\hom(\hopf,\reals)$ can be identified with the \emph{group-like elements}, defined by 
\begin{equ}
G(\hopf) = \{ g \in \hopf^* : \delta g = g \otimestilde g \}\;.
\end{equ} 
In particular, the equality \eqref{e:homo} holds if and only if 
\begin{equ}
\innerprod{\delta g, h_1\otimestilde h_2 } = \innerprod{g,h_1 h_2 } = \innerprod{g,h_1}\innerprod{g,h_2} = \innerprod{g\otimestilde g , h_1\otimestilde h_2}\;,
\end{equ}
for all $h_1,h_2\in \hopf$. The reason $G(\hopf)$ is called the set of group-like elements is because it is indeed a group. For the Connes-Kreimer Hopf algebra, this is often referred to as the \emph{Butcher group} \cite{hairer74}. 
\begin{prop}\label{prop:group}
The pair $(G(\hopf),\star)$ is a group with inverses given by $g^{-1}\defin \S^*g$, where $\S^*$ is the adjoint of the antipode.  
\end{prop}
\begin{proof}
Standard result for Hopf algebras and an easy exercise.
\end{proof}
The group property of $\hom(\hopf,\reals)$ is one of the main motivations behind Hopf algebras and in particular explains the role of the antipode. Indeed, the concept of a Hopf algebra is often introduced as the linearisation of a group. 
\par
If we were to replace $\hopf$ with the tensor product space $T(V)$ over the vector space $V=\reals^d$,  then we could equivalently characterise each group-like elements as the exponential of a Lie polynomial \cite[Theorem 1.4]{reutenauer93}. Remarkably, the same construction works in this setting too. We define the bracket $[\cdot,\cdot]_\star : \hopf^*\times \hopf^* \to \hopf^*$ by  
\begin{equ}\label{e:liebracket}
[h_1,h_2]_\star = h_1\star h_2 - h_2 \star h_1\;,
\end{equ}
which one can easily check is a Lie bracket. We define the set of $\delta$-primitives as
\begin{equ}\label{e:liealgebra}
P(\hopf) = \{ h \in \hopf^* : \delta h = 1\otimestilde h + h \otimestilde 1 \}\;,
\end{equ}
where $1$ is the co-unit in $\hopf^*$. In the context of Lie algebras, this condition is often stated as \\$\innerprod{h,xy} = \innerprod{1,x}\innerprod{h,y} + \innerprod{h,x}\innerprod{1,y}$ and the elements are known as \emph{derivations}. As suggested by the notation, $P(\hopf)$ is a Lie algebra with respect to $[\cdot,\cdot]_\star$ and has a very natural basis in $\hopf^*$. 
\begin{prop}\label{prop:algebra} The set $P(\hopf)$ is a Lie algebra with bracket $[\cdot,\cdot]_\star$ and moreover
\begin{equ}
P(\hopf) = \B\;,
\end{equ}
where $\B$ is the real vector space spanned by the dual trees $\trees^*$. 
\end{prop}
\begin{proof}
It is easy to check that $P(\hopf)$ is a Lie algebra. To check that $P(\hopf) = \B$, first let $\tau\in\trees^*$, then by definition
\begin{equ}
\delta \tau = \sum_{x_1, x_2 \in \forest^0} \inner{\tau, x_1 x_2 } x_1 \otimest x_2\;,
\end{equ}
where $\forest^0 = \forest \cup \{ 1\}$ and where we identify $x_1 \otimest x_2$ with the corresponding element in $\hopf^*\otimest \hopf^*$. But clearly, $\inner{\tau, x_1 x_2}=0$ unless $x_1=1$ or $x_2 = 1$ (but not both). It follows that 
\begin{equ}
\delta \tau = \sum_{x_1 \in \forest } \inner{\tau,x_1} \big(  x_1 \otimest 1 + 1 \otimest x_1\big) =  \tau\otimest 1 + 1 \otimest \tau\;, 
\end{equ}
and hence $\B \subseteq P(\hopf)$. To prove the reverse statement, suppose $h\in P(\hopf)$ and that
\begin{equ}
h = u + v\;,
\end{equ}
where $u \in \B$ and $v \in \B^\perp$, which is the vector space spanned by $1$ and all non-trivial products $\tau_1 \dots \tau_n \in \forest^*$ with $n\geq 2$. Since $u \in P(\hopf)$, it follows that $v=h-u \in P(\hopf)$. Thus,
\begin{align*}
1\otimest v + v\otimest 1 = \delta v = \inner{v,1}1\otimest 1 +  \sum_{\tau_1 \dots \tau_n} \inner{v,\tau_1 \dots \tau_n} \delta(\tau_1 \dots \tau_n)\;, 
\end{align*}
where we only sum over those $\tau_1 \dots \tau_n \in \forest^*$ with $n\geq 2$. By definition of $\delta(\tau_1 \dots \tau_n)$, this equals
\begin{equ}\label{e:Bstar_proof1}
\inner{v,1}1\otimest 1 + \sum_{\tau_1 \dots \tau_n} \inner{v,\tau_1 \dots \tau_n} \sum_{(i,j)} \tau_{i_1} \dots \tau_{i_p}\otimest \tau_{j_1}\dots \tau_{j_q}\;,
\end{equ}
where we sum over all subsets $\{i_1,\dots,i_p\}$, $\{j_1,\dots,j_q\}$ of $\{1,\dots,n\}$. However, each term $ \tau_{i_1} \dots \tau_{i_p}\otimest \tau_{j_1}\dots \tau_{j_q}$ (with $p,q \neq 0$) can only appear once in the expression \eqref{e:Bstar_proof1}, hence there can be no cancellations. Since these terms (as well as $1\otimest 1$) are basis elements of $\hopf^* \otimest \hopf^*$, we must have that $\inner{v,1}=0$ and $\inner{v,\tau_1 \dots \tau_n} = 0$ for all $\tau_1 \dots \tau_n \in \forest$ with $n\geq 2$. It follows that $P(\hopf) \subseteq \B$.    
\end{proof} 
Let $\mathfrak{h}$ be the space of all $h \in \hopf^*$ with $\innerprod{h,1}=0$ and let $H = 1+\mathfrak{h}$. Just as in the tensor product algebra case, the spaces $\mathfrak{h}$ and $H$ are diffeomorphic via the exponential map $\exp_\star : \mathfrak{h} \to H$ given by  
\begin{equ}
\exp_\star h = \sum_{k\geq0} \frac{h^{\star k }}{k!}\;,
\end{equ}  
where $h^{\star k } = h\star h^{\star (k-1)}$. Likewise we can define its inverse, the logarithmic map  by
\begin{equ}
\logstar (1+h) = \sum_{k\geq1} (-1)^{k-1}\frac{h^{\star k }}{k}\;,
\end{equ}  
for any $1+h \in H$. See \cite{manchon04} for further details. This allows us to classify the group-like elements as being the exponential of a Lie element.
\begin{prop}\label{prop:explie}
For any $g \in \mathfrak{h}$, we have that $g \in G(\hopf)$ if and only if $g=\expstar h$ for some $h \in P(\hopf)$. 
\end{prop}
\begin{proof}
The proof is identical to the tensor product algebra case \cite[Theorem 3.2]{reutenauer93}.
\end{proof}
\subsection{Branched rough paths}\label{subsec:brp}
We define the truncated group-like elements $G_N(\hopf)$, obtained from $G(\hopf)$ by quotienting out the ideal
\begin{equ}
\bigoplus_{k=N+1}^\infty \hopf_{(k)}^*\;,
\end{equ}
hence we identify all elements $\tau_1 \dots \tau_n \in \forest^*$ such that $|\tau_1 \dots \tau_n| \geq N+1$, with zero. From Proposition \ref{prop:explie}, it follows that $G_N(\hopf)$ is diffeomorphic to the real vector space $\trees_N$ and is therefore a Lie group. This Lie group plays precisely the same role as the step $N$ free nilpotent group in the geometric theory of rough paths. Indeed, the definition for branched rough paths follows naturally from that of geometric rough paths. 
\par
Let $X = (X^i)$ be a path in $\reals^d$ with H\"older regularity $\gamma \in (0,1)$. As always, we reserve the symbol $N$ for the largest integer such that $N\gamma \leq 1$.   
\begin{defn}\label{defn:branched}
A map $\X : [0,T]\to G_N(\hopf)$ is called a $\gamma$-H\"older \emph{branched rough path} if it satisfies
\begin{equ}\label{e:Gnorm}
\sup_{s\neq t}\; \frac{|\innerprod{\X_{st},\tau}|}{|t-s|^{\gamma |\tau|}} < \infty\;,
\end{equ}
for every $\tau \in \hopf_N$ and where $\X_{st} \defin \X_s^{-1} \star \X_t$. If $\inner{\X_{st},\bullet_i} = \delta X_{st}^i$ for each $i=1\dots d$, then we call $\X$ a branched rough path \emph{above} $X$.
\end{defn}
We see that the generalised version of Chen's property, or Condition (2) of the introduction, is immediate from the definition, since we have
\begin{equ}\label{e:chenstar}
\X_{su}\star\X_{ut} = (\X_s^{-1}\star \X_{u})\star(\X_u^{-1}\star\X_t) = \X_{st}\;.
\end{equ} 
Moreover, Definition \ref{defn:branched} is clearly equivalent to the original definition in \cite{gubinelli10a} and also stated in the introduction. In particular, Condition 1 from the original definition can be reformulated as $\X_{st} \in G_N(\hopf)$ for each $s,t\in[0,T]$. 
\begin{rmk}
As we shall see, the solution to an RDE only depends on the increment $\X_{st}$ rather than the path $\X_t$, hence there is no need to specify the initial value of the path $\X_0$. 
\end{rmk}
\begin{rmk}
In Definition \ref{defn:branched}, to justify calling $\X$ a $\gamma$-H\"older path, it should satisfy $d(\X_s,\X_t) \leq C|t-s|^\gamma $ for some metric $d$. This can be achieved using \emph{homogeneous norms}. For the step $N$ free nilpotent group, as with any Carnot group, one can show that all ``norms'' that are sub-additive and homogeneous with respect to the natural dilation of the group are equivalent \cite{lyons07}. This does not quite work with $G_N(\hopf)$, since it is not a Carnot group with respect to the right dilation. To be precise, we see that
\begin{equ}
\G_N(\hopf) = \expstar \left( \bigoplus_{k=0}^N \B_{(k)} \right)\;,
\end{equ}  
where $\B_{(k)}$ is the vector space spanned by $\trees_{(k)}$ . If $y_k \in \B_{(k)}$, then the natural dilation on $G_N(\hopf)$ is given by 
\begin{equ}
\delta_t \expstar \left(y_1 +  \dots +  y_N \right) = \expstar \left( t y_1 + t^2 y_2 + \dots + t^N y_N \right)\;.
\end{equ} 
In particular, in the case of a smooth path $X$ whose branched rough path $\X$ is given by the corresponding iterated integrals, if we multiplied $X$ by $t$, then we would obtain a factor of $t^{|\tau|}$ infront of $\innerprod{\X,\tau}$, so it is clear that this is the right choice of dilation. On the other hand, the only way $G_N(\hopf)$ could be a Carnot group is if we let all of $\B_N$ have the same grade, which would lead to useless norms. The correct notion is to view $G_N(\hopf)$ as a \emph{homogeneous group}, as defined in \cite{folland}. A homogeneous group $G$ is a Lie group whose Lie algebra is graded, and hence comes with a natural dilation. A (non-smooth) homogeneous norm on $G$ is then a map $\norm{\cdot} : G \to [0,\infty)$ that is continuous with respect to the manifold topology of $G$ and satisfies the homogeneity property $\norm{\delta_t g} = |t| \norm{g}$, where $\delta_t$ is the natural dilation of $G$ (along with other standard conditions). It is easy to show that all homogeneous norms on $G$ are equivalent. In the case of $G_N(\hopf)$, one can show that all homogeneous norms are equivalent to the natural homogeneous norm          
\begin{equ}
\norm{g}_{G_N(\hopf)} = \sum_{\tau \in \trees_N} |\innerprod{\logstar g,\tau}|^{1/|\tau|}\;. 
\end{equ}   
Moreover, it is easy to verify that the map $\X : [0,T] \to \big(G_N(\hopf),\norm{\cdot}_{G_N(\hopf)}\big)$ is $\gamma$-H\"older continuous if and only if condition \eqref{e:Gnorm} is satisfied. We can therefore equivalently define a branched rough path as a $\gamma$-H\"older path taking values in $\big(G_N(\hopf),\norm{\cdot}_{G_N(\hopf)}\big)$. 
\end{rmk}
As with classical rough paths, one can show that every branched rough path $\X$ can be canonically extended to a $\gamma$-H\"older continuous path taking values in $G(\hopf)$, courtesy of the \emph{sewing map} \cite{gubinelli04,gubinelli10a}. In more generality, branched rough paths also extend the idea of an almost multiplicative functional, in the following way. One calls $\Xtilde$ an \emph{almost branched rough path} if
\begin{equ}
|\innerprod{\Xtilde_{st} - \Xtilde_{su}\star\Xtilde_{ut}}|= \liloh(|t-s|)\;,
\end{equ}
for all $s,t \in [0,T]$, with $|t-s|\ll 1$. And moreover, we have the following
\begin{prop}
For every almost branched rough path $\Xtilde$, there exists a unique branched rough path $\X$ of regularity $\gamma$ such that
\begin{equ}
|\innerprod{\X_{st}-\Xtilde_{st},\tau}| = \liloh(|t-s|)\;,
\end{equ}
for all $s,t \in [0,T]$ and $\tau \in \trees_N$. 
\end{prop}
\begin{proof}
See \cite[Theorem 7.7]{gubinelli10a}. 
\end{proof}
Although we will not explicitly use the notion of an almost branched rough path, we include the definition to illustrate that all of the important tools for multiplicative functionals are still present in the setting of branched rough paths. 

\section{Controlled rough paths and solving RDEs}\label{sec:controlled}
In this section we recall the definition of a \emph{controlled rough path}, first defined in  \cite{gubinelli04} and later extended to branched rough paths in \cite{gubinelli10a}. We show how one can define rough integrals and moreover solutions to RDEs using this simple concept. 
\subsection{Controlled rough paths}
A crucial step in the theory of geometric rough paths is defining the integral of a one-form along a geometric rough path \cite{lyons98}. For $\alpha: \reals^d \to L(\reals^d,\reals)$ and a geometric rough path $\X$ above $X \in \reals^d$, in order to define $\int \alpha(X) dX$ one needs to impose a ${\rm{Lip}}(\beta)$ condition on $\alpha$, which states that for $j=1\dots N$, there exists $\alpha^j : \reals^d \to L((\reals^d)^{\otimes j}, \reals)$ such that $\alpha^1 = \alpha$ and
\begin{equ}\label{e:alpha_higher}
\alpha^j(X_t) = \sum_{i=0}^{N-j} \alpha^{i+j}(X_s)(\X_{st}^i) + R^j(X_s,X_t)\;,
\end{equ}
where $\X_{st}^i$ is the component of $\X_{st}$ in $(\reals^d)^{\otimes i}$ and the remainders $R^j$ satisfy $|R^j(\xi,\eta)| \leq M |\xi - \eta|^{\beta - j}$. In particular, from the $j=1$ case we see that
\begin{equ}\label{e:alpha_controlled}
\alpha(X_t) - \alpha(X_s)  = \sum_{i=1}^{N-1} \alpha^{i+1}(X_s)(\X_{st}^i) + R^1(X_s,X_t)\;,
\end{equ}
and hence the increment of $\alpha(X)$ is (locally) controlled by $\X$. The expression \eqref{e:alpha_controlled} leads directly to a definition of an almost multiplicative functional $\bigytilde$ which is subsequently extended to define $\int \alpha(X) dX$. The conditions on the higher order $\alpha^j$ given in \eqref{e:alpha_higher} are required to ensure that $\bigytilde$ actually is an almost multiplicative functional and thus prove that the map $\X \mapsto \int \alpha(X) dX$ is continuous in the $p$-var topology.   
\par
In the theory of controlled rough paths, the construction of integrals is more-or-less the same, except for that fact that one-forms are replaced with any object that satisfies a condition like \eqref{e:alpha_higher}. In particular let $\X$ be a branched rough path above $X$ and suppose $Z : [0,T] \to \reals$ satisfies
\begin{equ}\label{e:deltaZ}
\delta Z_{st} = \sum_{h\in \forest_{N-1}} Z^h_{s} \inner{\X_{st},h} + R^Z_{st}\;,
\end{equ}
where $|R^Z_{st}| \leq C |t-s|^{N\gamma}$ and the coefficients $Z^h : [0,T] \to \reals$. 
It is clear that the integral $\int_s^t Z_r dX^i_r$ should be approximated by the expression
\begin{equ}\label{e:ztilde_def}
\Ztilde_{st} =   Z_s \inner{\X_{st},\bullet_i} +  \sum_{h \in \forest_{N-1}}  Z^{h}_s \inner{\X_{st},[h]_i} = \sum_{h \in \forest_{N-1}^0} Z^{h}_s \inner{\X_{st},[h]_i} \;,
\end{equ}
for $|t-s| \ll 1$ where we denote $Z_s^1 = Z_s$ and $\forest_{N-1}^0 = \forest_{N-1}\cup \{1\}$, since one would expect
\begin{equ}
\int_s^t dX_r^i = \inner{\X_{st},\bullet_i} \quad \text{and} \quad \int_s^t \inner{\X_{sr},h}dX_r^i = \inner{\X_{st},[h]_i}\;.
\end{equ}
This idea is formalised by the \emph{sewing map}. The sewing map is essentially the same as the map which extends an almost multiplicative functional to an (approximately equal) multiplicative functional. 

\begin{lemma}[Sewing Map]
For any $\Ztilde : [0,T]\times[0,T] \to \reals$, if 
\begin{equ}\label{e:sewing_condition}
|\Ztilde_{st} - \Ztilde_{su} - \Ztilde_{ut}| \leq C|t-u|^{p}|u-s|^{q}\;,
\end{equ}
for some $p+q > 1$, then there exists a unique remainder terms $r: [0,T]\times [0,T] \to \reals$ such that $\Ztilde_{st} + r_{st}$ is the increment of a path and $|r_{st}| = \liloh(|t-s|)$. That is, there is some $Y: [0,T]\to \reals$ such that 
\begin{equ}
\delta Y_{st} = \Ztilde_{st} + r_{st}\;.
\end{equ}
for all $s,t \in [0,T]$. Moreover, it follows immediately that
\begin{equ}
\delta Y_{st} = \lim_{\partition \to 0} \sum_{[u,v]\in \partition} \Ztilde_{uv}\;,
\end{equ}
for any sequence of partitions $\partition$ with mesh-size tending to zero.  

\end{lemma}
Just as in \eqref{e:alpha_higher}, one needs conditions on the coefficients $Z_s^h$ to ensure that $\Ztilde$ defined in \eqref{e:ztilde_def} satisfies \eqref{e:sewing_condition}. The most convenient way of defining these controlled objects $Z$ along with their coefficients $Z^h$ is to consider them as one object $\bigz : [0,T] \to \hopf_{N-1}$, by setting
\begin{equ}
\inner{1,\bigz_t} = Z_t \quad \text{and} \quad \inner{h,\bigz_t} = Z^h_t\;, 
\end{equ}
for all $h \in \forest_{N-1}$. In the sequel we use the notation $\forest_{n}^0 = \forest_n \cup\{1\}$ and similarly for $\trees_n^0$. 
\begin{defn}
Let $\X$ be a $\gamma$-H\"older branched rough path. An $\X$-controlled rough path is a path $\bigz : [0,T] \to \hopf_{N-1}$ satisfying
\begin{equ}\label{e:consistency}
\inner{h,\bigz_t} =  \inner{\X_{st} \star h,\bigz_s} + R^h_{st}\;,
\end{equ}
for each $h \in \forest_{N-1}^0$, where $|R^{h}_{st}| \leq C |t-s|^{(N-|h|)\gamma}$. When $\inner{1,\bigz_t} = Z $, we say that $\bigz$ is a controlled rough path above $Z$. 
\end{defn}

Note that when $h=1$ and $\inner{1,\bigz}=Z$, the expression \eqref{e:consistency} can be written
\begin{equ}
\delta Z_{st} = \sum_{h\in\forest_{N-1}} \inner{h,\bigz_s} \inner{\X_{st},h} + R_{st}^1 \;,
\end{equ}
just as suggested in \eqref{e:deltaZ}. It is clear that \eqref{e:consistency} is simply the $\hopf$ counterpart of the ${\rm{Lip} (\beta)}$ condition \eqref{e:alpha_higher}.

\begin{rmk}
We can easily adapt this to the situation in which the coefficients of the `controlled object' take values in $\reals^e$ rather than $\reals$. In this case we have $\bigz: [0,T] \to (\hopf_{N-1})^e$ where $(\hopf_{N-1})^e$ denotes the $e$-th cartesian power of $\hopf_{N-1}$. Hence, the coefficients $\inner{h,\bigz}$ take values in $\reals^e$ and we denote the $i$-th component by $\inner{h,\bigz}_i$.
\end{rmk}
\par
Let $\bigz$ be an $\X$-controlled rough path above $Z$, then we can use $\bigz$ to define the integral $\int Z dX^i$, for any $1\leq i \leq d$. It is an easy exercise to check that the condition \eqref{e:consistency} ensures that 
\begin{equ}
\Ztilde_{st} = \sum_{h \in \forest_{N-1}^0} \inner{h,\bigz_s}\inner{\X_{st}, [h]_i}
\end{equ}
satisfies \eqref{e:sewing_condition}. From the sewing lemma, it follows that there exists a unique remainder $r$ with $|r_{st}| = \liloh(|t-s|)$ such that $\Ztilde_{st} + r_{st}$ is the increment of a path. Naturally, this increment is chosen as a definition of the integral
\begin{equ}\label{e:path_def}
\int_s^t Z_r dX_r^i  \defin \Ztilde_{st} + r_{st} = \lim_{\partition \to 0} \sum_{[u,v]\in \partition} \Ztilde_{uv}\;,
\end{equ}
for any partition $\partition$ of $[s,t]$. Hence, we have defined a map which sends a controlled rough path $\bigz$ to a path $\int Z dX^i$. This map can be extended to $\integration : \bigz \mapsto \int \bigz dX^i$, where $\int \bigz dX^i$ is a controlled rough path above $\int Z dX^i$. To define $\int \bigz dX^i$, we simply specify $\inner{h,\int \bigz dX^i}$ for all dual basis elements $h\in \forest_{N-1}^* \cup \{1\}$. Firstly, we let $\inner{1,\int_0^t \bigz_r dX_r^i}$ be the unique (up to an additive constant) path satisfying 
\begin{equ}
\innerbig{1,\int_0^t \bigz_r dX_r^i } - \innerbig{1,\int_0^s \bigz_r dX_r^i } = \int_s^t Z_r dX_r^i\;,
\end{equ}
and then define the coefficients by
\begin{equ}
\innerbig{[\tau_1 \dots \tau_n]_i,\int_0^t \bigz_r dX_r^i} = \innerbig{\tau_1 \dots \tau_n,\bigz_t}\;, 
\end{equ}
for all $[\tau_1 \dots \tau_n]_i \in \trees_{N-1}$ with $i$ fixed and
\begin{equ}
\quad \innerbig{\tau_1 \dots \tau_n,\int_0^t \bigz_t dX_r^i} = 0 \quad \text{otherwise}\;. 
\end{equ}
More generally, if $\bigz = (\bigz^1,\dots,\bigz^d)$ where each $\bigz^i$ is an $\X$-controlled rough path above $Z^i$, then we can define an $\X$-controlled rough path $\int \bigz \cdot dX$ above $\int Z\cdot dX$, where $Z= (Z^1,\dots,Z^d)$. To do this, we set
\begin{equ}
\innerbig{1,\int_s^t \bigz_r \cdot dX_r} = \sum_{i=1}^d \innerbig{1,\int_s^t \bigz^i_r dX^i_r} \;,
\end{equ}
with coefficients
\begin{equ}
\innerbig{[\tau_1 \dots \tau_n]_i,\int_0^t \bigz_r \cdot dX_r} = \inner{h,\bigz^i_t}\;
\end{equ}
for all $[\tau_1 \dots \tau_n]_i \in \trees_{N-1}$ and each $1\leq i \leq d$ and
\begin{equ}
\innerbig{\tau_1 \dots \tau_n,\int_0^t \bigz_r \cdot dX_r}  = 0 \quad \text{otherwise}\;. 
\end{equ}
For verification that $\int \bigz dX^i$ satisfies \eqref{e:consistency} and hence actually is a controlled rough path, see \cite[Theorem 8.5]{gubinelli10a}. 
\begin{rmk}
Since the definition of $\int Z dX^i$ depends on how we define a controlled rough path above $Z$, it makes more sense to use the controlled rough path notation $\int \bigz dX^i$. 
\end{rmk} 
\par
Not only are controlled rough paths stable under the integration map, but they are also stable under composition by smooth functions. We will demonstrate this for a controlled rough path $\bigz: [0,T] \to (\hopf_{N-1})^e$ and a smooth function $\phi : \reals^e \to \reals^e$. We first introduce the notation
\begin{equ}
D^n \phi(u) :  (v_1 ,\dots ,v_n) = \sum_{\alpha_1,\dots,\alpha_n=1}^e \del^{\alpha_1} \dots \del^{\alpha_n} \phi(u) v_1^{\alpha_1} \dots v_n^{\alpha_n}\;,
\end{equ}
where $u, v_i \in \reals^e$, $v_i^j$ denotes the $j$-th component of $v_i$. We define a controlled rough path $\phi(\bigz) : [0,T] \to (\hopf_{N-1})^e$  above $\phi(Z)$ using a Taylor expansion. In particular, we have that 
\begin{equ}\label{e:control_taylor}
\phi(Z_t) - \phi(Z_s) = \sum_{n=1}^{N-1} \frac{1}{n!}D^{n}\phi(Z_s): (\inner{h_1,\bigz_s},\dots,\inner{h_n,\bigz_s}) \innerprod{\X_{st},h_1 \dots h_n} + R^\phi_{st}\;,
\end{equ}
where we sum over all $h_i \in \forest$ with $|h_1| + \dots + |h_n| \leq N-1$ and $|R_{st}^\phi| \leq C |t-s|^{N\gamma}$. It is clear that the controlled rough path $\phi(\bigz)$ should have $\inner{1,\phi(\bigz_s)} = \phi(Z_s)$ and coefficients
\begin{equ}
\inner{h,\phi(\bigz)_s} = \sum_{n=1}^{N-1} \sum_{h_1 \dots h_n = h} \frac{1}{n!}D^{n}\phi(Z_s): (\inner{h_1,\bigz_s},\dots,\inner{h_n,\bigz_s})\;, 
\end{equ}
where we sum over all $h_1, \dots , h_n$ appearing in \eqref{e:control_taylor} such that $h_1 \dots h_n = h$. For verification that $\phi(\bigz)$ satisfies \eqref{e:consistency}, see \cite{gubinelli10a}. 
\begin{example}\label{eg:integral}
As an exercise, we will calculate $\int_s^t F(X)dX^i$, where $F:\reals^d \to \reals^d$ is a smooth function and $X$ has a branched rough path $\X$ above it. Firstly, since $\X$ is clearly an $\X$-controlled rough path, we can define $F(\X)$. We set $\inner{1,F(\X_t)} = F(X_t)$ and 
\begin{equ}
\inner{\bullet_{\beta_1} \dots \bullet_{\beta_m},F(\X_t)} = \del^{\beta_1} \dots \del^{\beta_m} F(X_t)\;,
\end{equ}
for all $\bullet_{\beta_1} \dots \bullet_{\beta_m} \in \forest_{N-1}$ and $\inner{h,F(\X_t)}=0$ otherwise. We then have 
\begin{align*}
\int_s^tF(X_r) &dX^i_r = F(X_s)\inner{\X_{st},\bullet_{i}} \\ &+ \sum_{m=1}^{N-1} \sum_{\bullet_{\beta_1}\dots\bullet_{\beta_m} \in \forest_{N-1}} \inner{\bullet_{\beta_1} \dots \bullet_{\beta_m}, F(\X_s)}\innerbig{\X_{st},[\bullet_{\beta_1} \dots \bullet_{\beta_m}]_i} + \liloh(|t-s|) \\
&= \sum_{m=0}^N \sum_{\beta_1,\dots,\beta_m=1}^d \frac{\del^{\beta_1}\dots\del^{\beta_m}F(X_s)}{m!}\innerbig{\X_{st},[\bullet_{\beta_1} \dots \bullet_{\beta_m}]_{i}} +\liloh(|t-s|)\;, 
\end{align*}
where in the last line we have used the symmetry of the expression to replace $\sum_{\bullet_{\beta_1}\dots\bullet_{\beta_m} \in \forest_N} $, the unordered sum, with $\sum_{\beta_1,\dots,\beta_m=1}^d 1/m!$.
\end{example}

The set of $\X$-controlled rough paths is easily seen to be a vector space. One can turn it into a Banach space, denoted $\controlled$, by introducing the norm
\begin{equ}
\norm{\bigz}_{\controlled} = |\bigz_0| + \sum_{h \in \forest_{N-1}} \norm{R^h}_{(N-|h|)\gamma}\;,
\end{equ}
where $\norm{f}_{(N-|h|)\gamma} =  \sup_{s\neq t} \frac{|f_{st}|}{|t-s|^{(N-|h|)\gamma}}$. The Banach space $\controlled$ turns out to be the right environment in which to solves RDEs. 

\subsection{Solving Rough DEs}
The foremost example of a controlled rough path is the solution to an RDE. We will consider the equation
\begin{equ}\label{e:solveRDE}
\delta Y_{st} = \int_s^t f(Y_r)\cdot dX_r \quad \text{with} \quad {Y_0 = \xi}\;,
\end{equ}
for every $s,t\in[0,T]$, where $X=(X^i) \in \reals^d$ and $f(Y)\cdot dX = \sum_{i=1}^d f_i(Y)dX^i$. The vector fields $f_i: \reals ^e \to \reals^e$ are assumed to be as smooth as required. To solve this RDE, we must specify a branched rough path $\X$ above $X$. In \cite{gubinelli10a}, solutions to \eqref{e:solveRDE} are defined by lifting the problem to the space of $\X$-controlled rough paths.
\begin{defn}
A path $Y: [0,T] \to \reals^e$ with $Y_0 = \xi$ is a solution to \eqref{e:solveRDE} if and only if there exists an $\X$-controlled rough path $\bigy$ above $Y$ satisfying
\begin{equ}\label{e:control_fixpoint}
\bigy_t - \bigy_s = \int_s^t f(\bigy_r)\cdot dX_r
\end{equ}
for every $s,t \in [0,T]$. 
\end{defn}
One can define the fixed point map $\fixmap : \controlled \to \controlled$ by $(\fixmap \bigy)_t = \int_0^t f(\bigy_r)\cdot dX_r$. Since $\controlled$ is a Banach space, we can apply standard fixed point arguments on $\fixmap$ to obtain existence and uniqueness results for \eqref{e:control_fixpoint}. In particular, if the vector fields $f_i$ have $N$ continuous and bounded derivatives, then global solutions exist for any initial condition. Moreover, if the vector fields have $N+1$ continuous and bounded derivatives then the solution is unique \cite[Theorem 8.8]{gubinelli10a}. Throughout the sequel, we will always assume the vector fields are smooth enough to guarantee existence and uniqueness of solutions. 
\par
In this article we are more concerned with the structure of RDEs, and would like an explicit representation of the controlled rough path solution to \eqref{e:control_fixpoint}. In particular, it is easy to see that a controlled rough path $\bigy$ is a solution if and only if 
\begin{equ}
\delta Y_{st}= \innerbig{1,\int_s^t f(\bigy_r)\cdot dX_r} 
\end{equ}
and 
\begin{equ}
\inner{[\tau_1 \dots \tau_n]_i, \bigy_s} = \innerbig{[\tau_1 \dots \tau_n]_i, \int_0^s f(\bigy_r) \cdot dX_r}\;, \end{equ}
for all $[\tau_1\dots\tau_n]_i \in \trees_{N-1}^*$ with $i=1\dots d$ and
\begin{equ}
 \inner{\tau_1 \dots \tau_m,\bigy_s} = 0\;,
\end{equ}
for all non-trivial products $\tau_1 \dots \tau_m \in \forest_N^\star \setminus \trees_N^\star$. Using the definition of $\int f(\bigy) \cdot dX$, we can refine the condition on the coefficients to $\inner{\bullet_i,\bigy_s} = f_i(Y_s)$ and
\begin{align}
\inner{[\tau_1 \dots \tau_n]_i,\bigy_s} &= \inner{\tau_1 \dots \tau_n,f_i(\bigy_s)} \notag\\
&= \sum_{\sigma \in \sym(n)}  \frac{D^n f_i(Y_s)}{n!} : \big(\inner{\tau_{\sigma(1)},\bigy_s},\dots,\inner{\tau_{\sigma(n)},\bigy_s}\big)\notag \\ 
&= D^n f_i(Y_s) : \big(\inner{\tau_1,\bigy_s},\dots,\inner{\tau_n,\bigy_s}\big)\label{e:fix_condition} \;,
\end{align}
where we sum over all permutations $\sigma$ of $\{1,\dots,n\}$. It follows that we can always write the coefficients as $\inner{\tau,\bigy_t} = f_\tau(Y_t)$ where $f_\tau : \reals^e \to \reals^e$ is some smooth function determined by $f$ and its derivatives. For instance, 
\begin{equ}
\inner{\mytreetwoone{j}{k}{i},\bigy_s} = D^2 f_i(Y_s):\big(\inner{\bullet_j,\bigy_s},\inner{\bullet_k,\bigy_s}\big) = D^2f_i:\big(f_j,f_k\big)(Y_s)\;.
\end{equ}
In the sequel, we will always reserve $\{f_\tau\}_{\tau \in \trees^*}$ for the family of functions satisfying the recurrence 
\begin{equ}\label{e:recurrence}
f_{[\tau_1 \dots \tau_n]_i} = D^n f_i : (f_{\tau_1},\dots,f_{\tau_n})\;, 
\end{equ}
for some specified $\{f_{\bullet_i}\}_{i=1\dots d}$. We also extend the family to any $h\in\hopf^*$ by 
\begin{equ}\label{e:fh_def}
f_{h} \defin \inner{h,1}\Id +  \sum_{\tau\in\trees} \inner{h,\tau} f_{\tau}\;, 
\end{equ}
where $\Id : \reals^e \to \reals^e$ is the identity map. It also follows that 
\begin{equ}
f_{\tau_1 \dots \tau_n} = 0 \;,
\end{equ}
for all non-trivial products $\tau_1 \dots\tau_n \in \forest_N^* \setminus \trees_N^*$. Moreover, if $f_{\bullet_i} = f_i$, then we have that $f_h(Y_t) = \inner{h,\bigy_t}$ for all $h\in\hopf_N^*$. 
\begin{rmk}

There is a conflict of notation here, if $h=1$ is the co-unit then from \eqref{e:fh_def} we have $f_1 = \Id$, so that $f_1(Y_t) = \inner{1,\bigy_t}$. This is not to be confused with the vector field $f_1$ in the original RDE. Since never actually refer to $f_1$ for the co-unit $1$, the reader should not be concerned. Indeed, it is simply included to make the definition \eqref{e:fh_def} consistent.
\end{rmk}

By exploiting some algebraic properties of the coefficients $f_\tau$, we can obtain an explicit formula for  $\inner{1,\bigy}=Y$. In the following proposition we define a controlled rough path $\bigy : [0,T] \to (\hopf_{N})^e$, with an extra layer of components $\inner{\tau_1\dots\tau_n,\bigy}$ for $|\tau_1 \dots \tau_n|=N$, these extra components serve no purpose other than to facilitate the definition of $\inner{1,\bigy}$. It is not hard to see that these extra components become important when considering the fixed point equation $\bigy = \fixmap \bigy$. 
\begin{prop}\label{prop:euler}
$\bigy : [0,T] \to (\hopf_N)^e$ with $\inner{1,\bigy}=Y$ is the unique controlled rough path solution to \eqref{e:control_fixpoint} if and only if
\begin{equ}\label{e:deltaY}
\delta Y_{st} = \sum_{\tau \in \trees_N} f_\tau(Y_s)\inner{\X_{st},\tau} + r_{st}
\end{equ}
where $|r_{st}| = \liloh(|t-s|)$ and the coefficients of $\bigy$ are given by $\inner{\tau_1 \dots \tau_n,\bigy_t} = f_{\tau_1 \dots \tau_n}(Y_t)$ for all $\tau_1 \dots \tau_n\in\forest_N^*$, with $f_{\bullet_i}=f_i$  for $i=1\dots d$. 
\end{prop}

\begin{rmk}\label{rmk:redundancy}
This proposition is particularly useful when one wants to understand how the solution depends explicitly on the vector fields and the rough path. In particular, it can be used to easily show how introducing redundancies in the vector fields $\{f_i\}$ leads to the solution only depending on certain components of the rough path $\X$. These types of results have been discussed in \cite{lyons98}, in the context of geometric rough paths.  
\end{rmk}

In order to show that $\bigy$ constructed by \eqref{e:deltaY} with $\inner{\tau_1 \dots \tau_n,\bigy_t}=f_{\tau_1 \dots \tau_n}(Y_t)$ is a solution, we must first show that it is a controlled rough path.

\begin{lemma}\label{lem:Y_control}
Suppose $\bigy : [0,T] \to (\hopf_N)^{e}$ with $\inner{1,\bigy}=Y$ satisfies \eqref{e:deltaY} and $\inner{\tau_1 \dots \tau_n,\bigy_t} = f_{\tau_1 \dots \tau_n}(Y_t)$ for all $\tau_1 \dots \tau_n \in \forest_N^*$ with $f_{\bullet_i} = f_i$. Then $\bigy$ is an $\X$-controlled rough path.
\end{lemma}

\begin{proof}[Proof of Lemma \ref{lem:Y_control}]
We must check the consistency condition \eqref{e:consistency} to ensure that $\bigy$ is a controlled rough path. The assumption \eqref{e:deltaY} ensures the condition holds for $h=1$, so it is sufficient to prove the condition for all $\tau \in \trees_{N}^*$, since the coefficients vanish on non-trivial products. We will assume the consistency condition holds for all of $\trees_k^*$ and prove the condition for $\tau = [\tau_1 \dots \tau_n]_i$ where $n\geq 0$ and $\tau_i \in \trees_k^*$. We have that 
\begin{align*}
\inner{[\tau_1 \dots \tau_n]_i,\bigy_t} &- \inner{\X_{st}\star [\tau_1 \dots \tau_n]_i, \bigy_s }\\
&= D^n f_i(Y_t):\big(\inner{\tau_1,\bigy_t},\dots,\inner{\tau_n,\bigy_t} \big) - \sum_{\rho \in \forest^0_{N}} \inner{\Y_{s},\rho} \inner{\X_{st}\otimest [\tau_1 \dots \tau_n]_i , \Delta \rho} \\
&=D^n f_i : \big(f_{\tau_1},\dots,f_{\tau_n}\big)(Y_t) - \sum_{\rho \in \trees_N^0} f_{\rho}(Y_s) \inner{\X_{st}\otimest [\tau_1 \dots \tau_n]_i,\Delta \rho}\;.
\end{align*}
Now, by a Taylor expansion on $D^n f_i$, we obtain
\begin{align}
\label{e:prove_control1}D^n f_i(Y_t)&:\big(f_{\tau_1},\dots, f_{\tau_n} \big) (Y_t)\\
 &= \sum_{m=n}^{N}\frac{1}{(m-n)!} D^m f_i(Y_s): \big(f_{\tau_1}(Y_t),\dots,f_{\tau_n}(Y_t),\delta Y_{st},\dots,\delta Y_{st}  \big) + R^f_{st}\notag\;,
\end{align}
where the term $\delta Y_{st}$ appears $m-n$ times and 
\begin{equ}
|R^f_{st}| \leq C|\delta Y_{st}|^{N-m} \leq C |t-s|^{(N-n)\gamma}\;.
\end{equ}
Now, by the inductive hypothesis we have that
\begin{equ}
f_{\tau_j}(Y_t) = \sum_{\sigma_j \in \trees^0_{N}} f_{\sigma_j}(Y_s)\inner{\X_{st}\star \tau_j,\sigma_j} + R_{st}^{\tau_j} \;,
\end{equ}
where $|R_{st}^{\tau_j}| \leq C|t-s|^{(N-|\tau_j|)\gamma}$ and by assumption we have that
\begin{equ}
 \delta Y_{st} = \sum_{\lambda_j \in \trees_N} f_{\lambda_j}(Y_s)\inner{\X_{st},\lambda_j} + r^j_{st}\;, 
\end{equ}
where $|r^j_{st}| = \liloh(|t-s|)$. If we substitute these into \eqref{e:prove_control1}, we obtain
\begin{align}
\sum_{m=n}^{N}&\frac{1}{(m-n)!} D^m f_i(Y_s): \big(f_{\sigma_1}(Y_s),\dots,f_{\sigma_n}(Y_s),f_{\lambda_1}(Y_s),\dots,f_{\lambda_{m-n}}(Y_s) \big)\notag \\ 
&\times \inner{\X_{st}\star \tau_1,\sigma_1} \dots \inner{\X_{st}\star \tau_n,\sigma_n} \inner{\X_{st},\lambda_1\dots\lambda_{m-n}} + R^\tau_{st} \label{e:prove_control2}\;,
\end{align}
where we sum over all $\sigma_j \in \trees_N$ (since $\sigma_j=1$ vanishes) and $\lambda_j \in \trees_N$ and where $R^\tau_{st}$ is the sum of all terms that contain at least one factor from the set $\{R_{st}^{\tau_1},\dots,R_{st}^{\tau_n},R_{st}^{f},r^1_{st},\dots,r^{m-n}_{st} \}$. Hence, 
\begin{equ}\label{e:R_estimate}
|R^\tau_{st}| \leq C \max_{1\leq j\leq n}\left( |t-s|^{(N-|\tau_j|)\gamma}\right) + C|t-s|^{(N-n)\gamma} \leq C|t-s|^{(N-|[\tau_1 \dots \tau_n]_i|)\gamma} \;,
\end{equ} 
where the bound on the second term follows from the fact that $n \leq |\tau_1| + \dots +|\tau_n| \leq |[\tau_1 \dots \tau_n]_i|$. On the other hand, we have that
\begin{align*}
\sum_{\rho \in \trees_{N}^0}& f_{\rho}(Y_s) \inner{\X_{st}\otimest [\tau_1 \dots \tau_n]_i , \Delta \rho}\\ 
&=\sum_{m=n}^{N-1} \sum_{\rho_1,\dots, \rho_m} \frac{1}{m!}f_{[\rho_1 \dots \rho_m]_i}(Y_s)\inner{\X_{st}\otimest [\tau_1 \dots \tau_n]_i, \Delta [\rho_1 \dots \rho_m]_i} 
\end{align*}
where we sum over $\rho_i \in \trees_N$ with $|\rho_1|+\dots+|\rho_m| \leq N-1$, since only those $\rho\in \trees$ with $\rho = [\rho_1 \dots \rho_m]_i$ for $m\geq n$ will not vanish. Note that the factor of $1/m!$ appears since all rearrangements of the $\rho_i$ in $[\rho_1 \dots \rho_m]$ produce the same $\rho$. Using the recurrence \eqref{e:recurrence}, this expands to
\begin{align*}
\sum_{m=n}^{N-1} \sum_{\rho_1,\dots, \rho_m}  \frac{1}{m!} D^m f_i(Y_s) : \big(f_{\rho_1}(Y_s),\dots,f_{\rho_m}(Y_s)\big) \inner{\X_{st}\otimest [\tau_1 \dots \tau_n]_i, \Delta [\rho_1 \dots \rho_m]_i}\;. 
\end{align*}
But we also have that
\begin{align*}
\inner{\X_{st}\otimest [\tau_1 \dots \tau_n]_i, \Delta [\rho_1 \dots \rho_m]_i} = \inner{\X_{st},\rho_1^{(1)}\dots \rho_m^{(1)}}\inner{\tau_1 \dots \tau_n,\rho_1^{(2)}\dots \rho_m^{(2)}}\;.
\end{align*}
It follows that
\begin{align}
&\sum_{\rho_1,\dots, \rho_m}  \frac{1}{m!} D^m f_i(Y_s) : \big(f_{\rho_1}(Y_s),\dots,f_{\rho_m}(Y_s)\big) \inner{\X_{st}\otimest [\tau_1 \dots \tau_n]_i, \Delta [\rho_1 \dots \rho_m]_i}\notag\\
&= \sum_{\rho_1,\dots, \rho_m}\frac{1}{m!} D^m f_i(Y_s) : \big(f_{\rho_1}(Y_s),\dots,f_{\rho_m}(Y_s)\big)\notag\\
&\qquad\qquad\qquad\qquad\times\inner{\X_{st},\rho_1^{(1)}\dots \rho_m^{(1)}}\inner{\tau_1 \dots \tau_n,\rho_1^{(2)}\dots \rho_m^{(2)}}\notag\\
&=\sum_{\rho_1,\dots, \rho_m}\frac{1}{m!} D^m f_i(Y_s) : \big(f_{\rho_1}(Y_s),\dots,f_{\rho_m}(Y_s)\big)\notag\\
&\qquad\qquad\qquad\qquad\times \binom{m}{n} \inner{\X_{st},\rho_1^{(1)}\dots \rho_n^{(1)} \rho_{n+1} \dots \rho_m}\inner{\tau_1 \dots \tau_n,\rho_1^{(2)}\dots \rho_n^{(2)}}\;.\label{e:prove_control3}
\end{align}
In the last equality we have used the fact that each term in $\inner{\tau_1 \dots \tau_n, \rho_1^{(2)}\dots \rho_m^{(2)}}$ will vanish unless $\rho_j^{(2)}=1$ for some choice of $m-n$ terms from the product $\rho_1^{(2)}\dots \rho_m^{(2)}$. But since the function $D^m f_i(Y_s) : \big(f_{\rho_1},\dots,f_{\rho_m}\big)$ is symmetric in $\rho_1,\dots,\rho_m$ and we are summing over all $\rho_1,\dots,\rho_m$, we can assume without loss of generality that $\rho_j^{(2)}=1$ for $n+1 \leq j \leq m$, provided we include the combinatorial factor $\binom{m}{n}$. Of course, it follows that for each $n+1 \leq j \leq m$, the only term remaining from the sum $\rho_{j}^{(1)}\otimest\rho_j^{(2)}$ is $\rho_j\otimest1$. 
\par
Now, since $\inner{\tau_1 \dots \tau_n,\rho_1^{(2)}\dots \rho_n^{(2)}} =1$ if and only if $\inner{\tau_j,\rho_{i_j}^{(2)}}=1$ for any permutation $(i_1,\dots,i_n)$ of $\{1,\dots,n\}$ and every $j=1\dots n$, we can write \eqref{e:prove_control3} as 
\begin{align}
&\sum_{\rho_1,\dots, \rho_m}\frac{1}{m!} D^m f_i  : \big(f_{\rho_1},\dots,f_{\rho_m}\big) (Y_s)\notag\\
&\qquad\qquad\qquad\qquad\times \binom{m}{n} \inner{\X_{st},\rho_1^{(1)}\dots \rho_n^{(1)} \rho_{n+1} \dots \rho_m} \sum_{(i_1,\dots,i_n)} \inner{\tau_1,\rho_{i_1}^{(2)}}\dots \inner{\tau_n,\rho_{i_n}^{(2)}}\notag\\
&=\sum_{\rho_1,\dots, \rho_m}\frac{n!}{m!} D^m f_i : \big(f_{\rho_1},\dots,f_{\rho_m}(Y_s)\big)\notag\\
&\qquad\qquad\qquad\qquad\times \binom{m}{n} \inner{\X_{st},\rho_1^{(1)}\dots \rho_n^{(1)} \rho_{n+1} \dots \rho_m}\inner{\tau_1,\rho_{1}^{(2)}}\dots \inner{\tau_n,\rho_{n}^{(2)}}\label{e:prove_control4}\;,
\end{align}
where in the first line we sum over all permutations $(i_1,\dots,i_n)$. If we set 
\begin{equ}
(\rho_1,\dots,\rho_m) = (\sigma_1,\dots,\sigma_n,\lambda_1,\dots,\lambda_{m-n})\;,
\end{equ}
then by comparing \eqref{e:prove_control2} with \eqref{e:prove_control4} and using the fact that $\frac{n!}{m!}\binom{m}{n} = \frac{1}{(m-n)!}$ we see that 
\begin{equ}
\inner{[\tau_1 \dots \tau_n]_i,\bigy_t} - \inner{\X_{st}\star [\tau_1 \dots \tau_n]_i, \bigy_s } = R^\tau_{st}\;,
\end{equ} 
and the estimate \eqref{e:R_estimate} proves the result. 
\end{proof}

\begin{proof}[Proof of Proposition \ref{prop:euler}]
We will first prove the `if' statement. Define $\bigy$ as in Lemma \ref{lem:Y_control}, then $\bigy$ is indeed an $\X$-controlled rough path. Now, from \eqref{e:fix_condition}, we have that $\inner{\tau,\bigy_t} = \inner{\tau,(\fixmap\bigy)_t}$ for all $\tau \in \trees_N^*$, so to show $\bigy$ is the unique solution, it suffices to show that 
\begin{equ}
\delta Y_{st} = \innerbig{1,\int_s^t f(\bigy_r) \cdot dX_r}\;.
\end{equ}
By the definition, we have that
\begin{align}\label{e:pathproof1}
&\innerbig{1,\int_s^t f(\bigy_r)\cdot dX_r}\\ 
\notag &\qquad= \sum_{i=1}^d \sum_{\tau_1 \dots \tau_n \in\forest_{N-1}} D^n f_i(Y_s):(\inner{\tau_1,\bigy_s},\dots,\inner{\tau_n,\bigy_s}) \innerprod{\X_{st},[\tau_1 \dots \tau_n]_i} + \rtilde_{st}\\
&\qquad= \sum_{i=1}^d \sum_{\tau_1 \dots \tau_n \in\forest_{N-1}}D^n f_i:(f_{\tau_1},\dots,f_{\tau_n})(Y_s) \innerprod{\X_{st},[\tau_1 \dots \tau_n]_i} + \rtilde_{st}\notag\\
&\qquad=\sum_{i=1}^d \sum_{\tau_1 \dots \tau_n \in\forest_{N-1}}f_{[\tau_1 \dots \tau_n]_i}(Y_s)\inner{\X_{st},[\tau_1 \dots \tau_n]_i} + \rtilde_{st}\notag\;,
\end{align}
where $|\rtilde_{st}| = \liloh(|t-s|)$. However, by setting $\tau = [\tau_1 \dots \tau_n]_i$, we can rewrite the sum 
\begin{equ}
\sum_{\tau\in \trees_N}  =  \sum_{i=1}^d \sum_{\tau_1 \dots \tau_n \in\forest_{N-1}}\;.
\end{equ}
We obtain
\begin{equ}\label{e:pathproof2}
\innerbig{1,\int_s^t f(\bigy_r)\cdot dX_r} = \sum_{\tau \in \trees_N} f_{\tau}(Y_s)\innerprod{\X_{st},\tau} + \rtilde_{st}\;. 
\end{equ}
But from \eqref{e:deltaY}, it follows that 
\begin{equ}
\innerbig{1,\int_s^t f(\bigy_r)\cdot dX_r}  - \delta Y_{st} =  \rtilde_{st}-r_{st}\;,
\end{equ}
and since the left hand side is an increment and the right hand side is $\liloh(|t-s|)$, the left hand side must be identically zero. 
\par
For the `only if' statement, suppose $\bigy$ is the unique solution to \eqref{e:control_fixpoint} with $\inner{1,\bigy}=Y$. Then since $\inner{[\tau_1 \dots \tau_n]_i,\bigy_t} = \inner{[\tau_1 \dots \tau_n]_i,(\fixmap\bigy)_t}$, it follows from \eqref{e:fix_condition} (and the preceding argument) that $\inner{\tau_1 \dots \tau_n,\bigy_t} = f_{\tau_1 \dots \tau_n}(Y_t)$ for all $\tau_1 \dots \tau_n \in \forest_N^*$ with $f_{\bullet_i} = f_i$ for $i=1\dots d$. Note that \eqref{e:pathproof2} also holds for the solution $\bigy$, since the identity only relies on the coefficients $\inner{\tau,\bigy}$ satisfying \eqref{e:recurrence}. It follows that
\begin{align*}
\delta Y_{st} &= \innerbig{1,\int_s^t f(\bigy_r)\cdot dX_r} \\
&= \sum_{\tau \in \trees_N} f_{\tau}(Y_s)\inner{\X_{st},\tau} + \rtilde_{st}\;,
\end{align*}   
This proves \eqref{e:deltaY} and hence completes the proof. 
\end{proof}
 \begin{example}\label{eg:linear}
Let us consider the RDE with linear vector fields,
\begin{equ}\label{e:linearRDE}
\delta Y_{st} = \sum_{i=1}^d \int_s^t V_i Y_r dX^i_r \;,
\end{equ}
where $V_i \in L(\reals^e,\reals^e)$. Since the vector fields are smooth, the solution $Y$ must take the form \eqref{e:deltaY}, where $Y =\inner{1,\bigy}$ and the coefficients satisfy $\inner{[\tau]_i,\bigy_s} = V_i \inner{\tau,\bigy_s}$ for any $[\tau]_i \in \trees_N$ and $\inner{\tau_1 \dots \tau_m,\bigy_s} = 0$ for any non-trivial product of $\tau_i \in \trees$. Hence, we have that
\begin{equ}
\delta Y_{st} = \sum_{n=1}^N \sum_{a} (V_{a_n} \dots V_{a_1}Y_s) \inner{\X_{st},\tau^{a_n \dots a_1 }} + r_{st}\;,
\end{equ}
where we sum over all vectors $a= (a_1,\dots,a_n) \in \{1,\dots,d \}^n$ and we use the shorthand for linear trees $\tau^{a_n \dots a_1} = [\tau^{a_n\dots a_{2}}]_{a_1}$. If we use the injection map $\iota : T(\reals^d) \to \hopf$ defined in Remark \ref{rmk:tensor_identify}, then we have
\begin{equ}
\delta Y_{st} = \sum_{n=1}^N \sum_{a} (V_{a_n} \dots V_{a_1}Y_s) \inner{\X_{st}, \iota \big(e_{a_n} \otimes \dots \otimes e_{ a_1 }\big) } + r_{st}\;,
\end{equ}
which coincides with the standard Davie solution \cite{davie07,friz10}, defined in the case of a geometric rough path. In our case, the ``branched'' components only influence the solution through terms involving second order derivatives of the vector field, which always vanish. This is a good example of how to use Proposition \ref{prop:euler} to show which components of the rough paths actually count towards the solution, as mentioned in Remark \ref{rmk:redundancy}.   
\end{example}

\section{Geometric rough paths}\label{sec:geometric}
Let $V$ be some real Banach space. Let $T(V) = \bigoplus_{i=0}^\infty V^{\otimes i}$ be the tensor product algebra of $V$, with the convention $V^{\otimes 0} = \reals$. We will call $T^{(n)}(V) = \bigoplus_{i=0}^n V^{\otimes i}$ the \emph{step-$n$ truncated tensor algebra}. The vector space $T(V)$ can be viewed as a Hopf algebra, by adding the \emph{shuffle} product $\shuffle$ and the \emph{deconcatenation} coproduct $\Deltabar$. The existence of an antipode for this bialgebra is guaranteed by the fact that it is graded, with zeroth grade equal to $\reals$.  The shuffle product is defined in the following way, let $e_a = e_{a_1}\otimes \dots \otimes e_{a_n}$ and $e_b = e_{b_1}\otimes\dots\otimes e_{b_m}$ then 
\begin{equ}
e_a \shuffle e_b = \sum_{c \in \shufset(a,b)} e_c\;,
\end{equ}
where $c \in  \shufset(a,b)$ if and only if $c$ is a permutation of the index sequence $(a,b) = (a_1,\dots,a_n,b_1,\dots,b_m)$ which preserves the original ordering of the index sequences $a$ and $b$ respectively. The coproduct $\Deltabar$ is defined by
\begin{equ}    
\Deltabar e_c = \sum_{(a,b)=c} e_a\otimestilde e_b\;.
\end{equ}
The dual Hopf algebra $T((V))$ is the space of formal series of tensors, equipped with the concatenation product $\otimes$ and the coproduct $\deltabar$, that are dual to $\Deltabar$ and $\shuffle$ respectively. We likewise have $T((V)) = \bigoplus_{i=0}^\infty (V^*)^{\otimes i}$ (allowing for formal series) and the truncation, $T^{(n)}((V))$ which can clearly be identified with $T^{(n)}(V)$. More details on the above construction can be found, for instance, in \cite{reutenauer93}. 
\par
We define a Lie bracket on $T((V))$ using the commutator 
\begin{equ}[]
[x,y]_\otimes = x\otimes y - y \otimes x\;,
\end{equ}
for any $x,y \in T((V))$. Define the sequence of vector spaces $W_i(V)$ by $W_1(V)=V$ and $W_{i+1}(V) = [V,W_i(V)]_\otimes$ for $i\geq 1$. We call $S$ a \emph{formal Lie series} if $S$ can be written as a formal sum
\begin{equ}
S = \sum_{i\geq 1} S_i\;,
\end{equ}
where each $S_i \in W_i(V)$ is some Lie polynomial. We will denote the vector space of formal Lie series by $\G(V)$, clearly $\G(V)$ is a subspace of $T((V))$. We similarly denote the \emph{step-$n$ free Lie algebra} by 
\begin{equ}
\G^{(n)}(V) = \bigoplus_{i=1}^n W_i(V) \;,
\end{equ}
which is the level-$n$ truncation of $\G(V)$. We define the group 
\begin{equ}
G(V) = \exp (\G(V))\;,
\end{equ}
as the image of $\G(V)$ under the exponential map, defined by
\begin{equ}
\exp(x) = \sum_{k\geq0} \frac{x^{\otimes k}}{k!}\;.
\end{equ}
We similarly define the \emph{step-$n$ free nilpotent group} as
\begin{equ}
G^{(n)}(V)= \exp(\G^{(n)}(V))\;.
\end{equ}
Clearly, we have that $G^{(n)}(V) \subset T^{(n)}(V)$. Moreover, it is well known that $G(V)$ coincides with the group-like objects, in that $g \in G(V)$ if and only if $\deltabar g = g\otimestilde g$, the proof of this statement can be found in \cite{reutenauer93}. It follows that $G^{(n)}(V)$ is the step-$n$ truncation of the group-like objects. Since $\deltabar$ is dual to $\shuffle$, this group-like property can be equivalently stated as 
\begin{equ}\label{e:geoshuffle}
 \innerprod{g,x}\innerprod{g,y} = \innerprod{g,x\shuffle y}\;,
\end{equ} 
for every $g\in G(V)$ and $x,y\in T(V)$. 
\par
The group $G^{(n)}(V)$ can be equipped with a \emph{subadditive homogeneous norm} $\norm{\cdot}_{G^{(n)}(V)}: G^{(n)}(V) \to [0,\infty)$ as defined in \cite{lyons07}. One can show that all such norms are equivalent, and in particular all such norms are equivalent to 
\begin{equ}
\rho(x) = \sum_{k=1}^n\norm{\ell_k}^{1/k}\;, 
\end{equ}
where $x = \exp (\ell_1 + \dots + \ell_n)$ with $\ell_i \in W_i(V)$ and where $\norm{\cdot}$ denotes the Euclidean norm \cite{lyons07}. A path $\X : [0,T] \to G^{(n)}(V)$ is called $\gamma$-H\"older continuous if and only if
\begin{equ}\label{e:georegularity}
\sup_{s\neq t} \frac{\norm{\X_{st}}_{G^{(n)}(V)}}{|t-s|^\gamma} < \infty\;,
\end{equ}
where $\X_{st} = \X_s^{-1}\otimes \X_t$. These ingredients allow us to define rigorously a geometric rough path.
\begin{defn} Fix $\gamma \in (0,1)$ and let $N$ be the largest integer such that $N\gamma \leq 1$. A path $\X : [0,T] \to G^{(N)}(V)$ of H\"older exponent $\gamma$ is called a \emph{geometric rough path}.  
\end{defn}
In light of Remark \ref{rmk:truncate}, one can easily check that this definition of geometric rough paths is equivalent to the definition given in the introduction. In particular, by the equivalence of norms, the regularity condition \eqref{e:georegularity} is synonymous with the statement 
\begin{equ}\label{e:georegularity2}
\sup_{s\neq t}\frac{|\innerprod{\X_{st},e_{i_1 \dots i_k} }|}{|t-s|^{k\gamma}} < \infty\;,
\end{equ}
for any word $e_{i_1 \dots i_k}$. Hence, we also see that the regularity condition for a geometric rough path is identical to that of a branched rough path.
\par
It turns out that every $\gamma$-H\"older path in a Banach space $V$ can be extended to a path $\Xbar$ taking values in $G^{(N)}(V)$. That is, every path has a geometric rough path lying above it. This is a particular case of the following theorem proved in \cite{lyons07}. If $K$ is a normal subgroup of $G^{(N)}(V)$, we define the quotient homogeneous norm on the quotient group $G^{(N)}(V)/K$ by 
\begin{equ}
\norm{g\otimes K}_{G^{(N)}(V)/K} = \inf_{k\in K} \norm{g\otimes k }_{G^{(N)}(V)}\;.
\end{equ}
\begin{thm}[Lyons-Victoir extension]\label{thm:extension}
Let $\gamma \in (0,1)$ such that $\gamma^{-1} \notin \naturals \setminus \{0,1\}$. Suppose $K$ is a normal subgroup of $G^{(N)}(V)$. If $\X$ is a $\gamma$-H\"older continuous path in the quotient $G^{(N)}(V)/K$, then there exists a $\gamma$-H\"older continuous path $\Xbar$ taking values in $G^{(N)}(V)$ and satisfying
\begin{equ}
\pi_{G^{(N)}(V)/ K} \left(\Xbar \right) = \X \;, 
\end{equ}  
where $\pi$ denotes the projection map. 
\end{thm}

\begin{rmk}
The restriction $\gamma^{-1} \notin \naturals \setminus \{0,1\}$ is a necessary one and a counter example can be found in \cite{victoir04}. Hence, all our results in this chapter actually assume $\gamma \in (0,1)$ with $\gamma^{-1} \notin \naturals$. 
\end{rmk}

\begin{example}
To give an idea of the type of situation in which this theorem applies, let $\X$ be a geometric rough path in $\Ntensor(\reals^d)$, lying above a path $X \in \reals^d$ and suppose we would like to add a new path component $X^{d+1}$ to $X$, by setting $\xbar = (X,X^{d+1})$. The extension theorem tells us that there exists a geometric rough path $\Xbar$ above $\xbar$ that agrees with $\X$ on the subspace $\Ntensor(\reals^d) \subset \Ntensor(\reals^{d+1})$. To be precise, we set 
\begin{equ}
\Xhat_t = \exp\left( \log\X_t + x^{d+1}_{t} e_{d+1}  \right)\;.
\end{equ}
This is an element in $\Ntensor(\reals^{d+1})$ and one can easily check that it is $\gamma$-H\"older in the quotient space $G^{(N)}(\reals^{d+1})/K$, where $K=\exp L$ and $L$ is the Lie ideal generated by 
\begin{equ}[]
[e_{d+1},\reals^d]_{\otimes} = \spn \{e_{d+1}\otimes e_j - e_j \otimes e_{d+1} : j=1\dots d\}. 
\end{equ}
In particular, under the quotient norm we can effectively ignore all bracket terms involving $e_{d+1}$, and the $\gamma$-H\"older property then follows from the fact that $\X$ is $\gamma$-H\"older in $G^{(N)}(\reals^d)$. Theorem \ref{thm:extension} tells us that we can \emph{add} the missing $e_{d+1}$ components to obtain a geometric rough path $\Xbar$ on $\Ntensor(\reals^{d+1})$.

\end{example}

\begin{rmk}
Although the proof of Theorem \ref{thm:extension}, as stated in \cite{lyons98}, appears to require the axiom of choice, the map $\X \mapsto \Xbar$ can actually be defined \emph{constructively} (and hence the map is measurable). In particular, this implies that the components of $\Xbar$ can be built explicitly from the components of $\X$. 
\end{rmk}

\subsection{Geometric rough paths are branched rough paths}
It should be no surprise that a geometric rough path is a special kind of branched rough path. As mentioned in Remark \ref{rmk:tensor_identify} the tensor algebra $T(\reals^d)$ can be identified with the subspace of $\hopf$ spanned by the linear trees. Given a geometric rough path $\Xbar$, the idea is to extend $\Xbar$ from this subspace of linear trees to a branched rough path $\X$ defined on the whole of $\hopf$. To perform this extension, we simply replace $\hopf$ products with $\shuffle$ products. That is, we set 
\begin{equ}\label{e:geocondition}
\innerprod{\X_{st},h} = \innerprod{\Xbar_{st},\phi_g(h)}\;,
\end{equ}  
for every $h \in \hopf_N$, where the map $\phi_g : \hopf \to T(\reals^d)$ is defined by the rules $\phi_g(1) = 1$\;,
\begin{equ}
\phi_g([h]_i) = \phi_g(h)\otimes e_{i} \quad \text{and} \quad \phi_g(h_1 h_2) = \phi_g(h_1) \shuffle \phi_g(h_2)\;,
\end{equ}
for $h,h_1,h_2 \in \hopf$ and $e_i$ is the $i$-th canonical basis vector in $\reals^d$. For example, we have
\begin{equ}
\phi_g(\bullet_{a}\bullet_{b}) =  e_{ab}+ e_{ba} \quad \text{and} \quad \phi_g (\mytreetwoone{a}{b}{c}\;\;) = e_{abc} + e_{bac} \;,
\end{equ}
where $e_{ab} = e_a\otimes e_b$ and so forth.
\begin{prop}\label{prop:geobranched}
If $\Xbar$ is a $\gamma$-H\"older geometric rough path defined on $T^{(N)}(\reals^d)$ and $\X$ is defined by \eqref{e:geocondition}, then $\X$ is a $\gamma$-H\"older branched rough path on $\hopf$. 
\end{prop}
This also provides a way to test the geometricity of a branched rough path. In particular, a branched rough path is geometric if and only if the identity
\begin{equ}
\innerprod{\X_{st},h} = \innerprod{\X_{st},\iota \phi_g(h)}\;,
\end{equ}
holds for every $h \in \hopf_N$, where $\iota: T(\reals^d) \to \hopf$ is the inclusion map that identifies each tensor in $T(\reals^d)$ with its corresponding linear tree in $\hopf$. Before proving the proposition, we need an important lemma. The map $\phi_g$ is clearly a morphism from $\cdot$ to $\shuffle$. What is less clear is that it is also a morphism of coproducts $\Delta$ and $\Deltabar$ and hence a Hopf algebra morphism. This is crucial in guaranteeing that $\X$ constructed above satisfies the right algebraic conditions. 
\begin{rmk}
There is a well known `universality' result \cite{foissy13}, which states the following. Let $\CK$ be any Hopf algebra and let $\{L_i : \CK \to \CK \}_{i=1\dots d}$ be any collection of $1$-cocycles with respect to the coproduct, then there exists a Hopf algebra morphism $\zeta : \hopf \to \CK$ satisfying 
\begin{equ}
\zeta([\tau_1 \dots \tau_n]_i) = L_i \circ \zeta(\tau_1 \dots \tau_n)\;,
\end{equ}
for any $[\tau_1 \dots \tau_n]_i \in \trees$. By taking $\CK = T(V)$ and  $L_i (e_{j_1 \dots j_n}) = e_{j_1 \dots j_n i}$, we see that the map $\zeta$ is precisely $\phi_g$. However, in our case it is not too difficult to simply check that the map $\phi_g$ is indeed a Hopf algebra morphism and we include this in Lemma \ref{lem:deltaphig}.
\end{rmk}

For the following, recall that $\forest_{(n)}$ is all $\tau_1 \dots \tau_k \in \forest$ with $|\tau_1| +  \dots +  |\tau_k| =n$ and that $\hopf_{(n)}$ is the vector space spanned by $\forest_{(n)}$ and also that $\forest_n$ is all $\tau_1 \dots \tau_k \in \forest$ with $|\tau_1| + \dots + |\tau_k| \leq n$.
\begin{lemma}\label{lem:deltaphig}
We have that
\begin{equ}\label{e:deltaphig}
\Deltabar \phi_g(h) = (\phi_g \otimestilde \phi_g) \Delta h\;,
\end{equ}
for every $h \in \hopf$.
\end{lemma}
\begin{proof}
When applied to any $h\in \hopf_{(1)}$, the identity \eqref{e:deltaphig} is clear, so assume the claim holds on all $h\in \hopf_{(n)}$, we will prove that the claim holds for $\forest_{(n+1)}$ and hence $\hopf_{(n+1)}$. If $h \in \forest_{(n+1)}$, then $h = [h_1]_i$ for some $h_1 \in \forest_{(n)}$ or $h = h_1 h_2$ for $h_1 \in \forest_{(p)}, h_2 \in \forest_{(q)}$ for $p+q=n$. In the first case, 
\begin{equ}
\Deltabar \phi_g([h_1]_i) = \Deltabar \big( \phi_g(h_1)\otimes e_{i} \big) = \big( \phi_g(h_1)\otimes e_{i} \big)\otimestilde 1 + (\Deltabar \phi_g(h_1)) \otimes (1\otimestilde e_{i} )\;.
\end{equ}
By the inductive assumption, $\Deltabar \phi_g (h_1) = (\phi_g\otimestilde \phi_g) \Delta h_1$. If $\Delta h_1 = h_1^{(1)}\otimestilde h_1^{(2)}$, then we obtain  
\begin{align*}
\phi_g &([h_1]_i)\otimestilde 1 + (\phi_g \otimestilde \phi_g)(h_1^{(1)}\otimestilde h_1^{(2)})\otimes(1\otimestilde e_{i})\\
&= (\phi_g \otimestilde \phi_g) \big([h_1]_i\otimestilde 1 + h_1^{(1)}\otimestilde [h_1^{(2)}]_i  \big) = (\phi_g \otimestilde \phi_g)\Delta [h_1]\;.
\end{align*}
In the second case, 
\begin{equ}
\Deltabar\phi_g (h_1 h_2) = \Deltabar \big(\phi_g(h_1)\shuffle \phi_g(h_2)\big) = (\Deltabar \phi_g(h_1))\shuffle(\Deltabar \phi_g(h_2))\;,
\end{equ}
where we have used the fact that $\Deltabar$ is a morphism with respect to $\shuffle$. By the inductive assumption, we obtain
\begin{equ}
(\phi_g \otimestilde \phi_g)(\Delta h_1) \shuffle (\phi_g \otimestilde \phi_g)(\Delta h_2) = (\phi_g \otimestilde \phi_g) (\Delta h_1 \cdot \Delta h_2) = (\phi_g \otimestilde \phi_g) \Delta (h_1 h_2)\;,
\end{equ}
where, in the first equality we have used the fact that $\phi_g$ is a $\shuffle$ morphism and $\Delta$ is a $\cdot$ morphism. This proves \eqref{e:deltaphig}.
\end{proof}
\begin{proof}[Proof of Proposition \ref{prop:geobranched}]
From \eqref{e:geocondition}, the path $\X$ is only defined through the incremental object $\X_{st}$. Hence, we must first check that $\X_{tt} = 1$ and that
\begin{equ}\label{e:geomult}
\X_{st} = \X_{su}\star \X_{ut},
\end{equ}
for every $s,u,t \in [0,T]$. The first claim follows from the fact that $\Xbar_{tt}=1$ and that $\phi_g^* 1 =1 $, where $\phi_g^*$ is the adjoint of $\phi_g$ and $1$ is the counit. To check \eqref{e:geomult}, notice that 
\begin{equ}
\innerprod{\X_{su}\star \X_{ut},h} = \innerprod{\X_{su}\otimestilde \X_{ut},\Delta h} = \innerprod{\Xbar_{su}\otimestilde \Xbar_{ut},(\phi_g \otimestilde \phi_g)\Delta h}\;.
\end{equ}
Applying Lemma \ref{lem:deltaphig}, the above equals
\begin{equ}
\innerprod{\Xbar_{su}\otimestilde \Xbar_{ut},\Deltabar \phi_g (h)} = \innerprod{\Xbar_{su}\otimes \Xbar_{ut},\phi_g (h)} = \innerprod{\Xbar_{st},\phi_g (h)}  =\innerprod{\X_{st},h}\;.
\end{equ}
The regularity condition \eqref{e:Gnorm} for a branched rough path follows easily from the fact that $\phi_g(\tau)$ is, for every $\tau \in \trees$, a linear combination in $(\reals^d)^{\otimes |\tau|}$. Hence, the regularity of $\innerprod{\X_{st},\tau}$ will follow from \eqref{e:georegularity2}. We finally check that $\X_t \defin \X_{0t}$ takes values in the truncated group-like elements $G_N(\hopf)$. Since $\phi_g$ is a morphism with respect to $\cdot$ and $\shuffle$, we have that 
\begin{equ}
\innerprod{\X_t,h_1 h_2} = \innerprod{\Xbar_t,\phi_g(h_1 h_2)} = \innerprod{\Xbar_t,\phi_g(h_1)\shuffle \phi_g(h_2)}\;,
\end{equ}
for any $h_1, h_2 \in \hopf$ with $|h_1|+|h_2| \leq N$. Since $\Xbar$ is geometric and hence group-like, \eqref{e:geoshuffle} yields 
\begin{equ}
\innerprod{\Xbar_t,\phi_g(h_1)\shuffle \phi_g(h_2)}=\innerprod{\Xbar_t, \phi_g(h_1)}\innerprod{\Xbar_t,\phi_g(h_2)} = \innerprod{\X_t,h_1}\innerprod{\X_t,h_2}\;.
\end{equ}
Hence, $\X$ takes values in $G_N(\hopf)$. 
\end{proof}

\subsection{Branched rough paths are geometric rough paths}
The main result of this subsection provides a converse to Proposition \ref{prop:geobranched}, namely, for a given branched rough path $\X$ lying above a path $X$, we can construct a geometric rough path $\Xbar$ lying above a higher dimensional path $\xbar$, in such a way that $\Xbar$ contains all the information of $\X$. Hence, every branched rough path can be viewed as a geometric rough path, living in an extended space.
\par
Before stating the main result, we first need some notation. As above, let $\B, \B_n$ be the real vector spaces spanned by $\trees, \trees_n$ respectively, then we can then define the tensor product algebras $T(\B), T(\B_n)$ exactly as above. In $T(\B)$ (and $T(\B_n)$), the elements of $\trees$ ($\trees_n$) are indivisible objects with respect to the coproduct $\Deltabar$, that is
\begin{equ}
\Deltabar \tau = 1\otimestilde \tau + \tau \otimestilde 1\;,
\end{equ}
for any $\tau \in \trees$. Moreover, the basis elements of $T(\B)$ are tensors of the form $\tau_1\otimes \dots \otimes \tau_k$, for $\tau_i \in \trees$ and similarly for $T(\B_n)$. As usual, we denote the truncated tensor algebra by
\begin{equ}\label{e:TBN_truncated}
T^{(N)}(\B) = \bigoplus_{k=0}^N \B^{\otimes k} \quad \text{and} \quad T^{(N)}(\B_n) = \bigoplus_{k=0}^N \B_n^{\otimes k}\;.
\end{equ}
Every tensor product space can be equipped with the usual grading which counts the number of non-trivial factors in each tensor product. However, we equip $T(\B)$ (and $T(\B_n)$) with a grading that does not ignore the individual grading of the trees. That is, we have
\begin{equ}\label{e:tensor_grading}
|\tau_1 \otimes \dots \otimes \tau_n| = |\tau_1| + \dots + |\tau_n|\;,
\end{equ} 
where $|\tau_i|$ is the $\hopf$ grading that counts the number of vertices in $\tau_i$. Hence, we have the decomposition
\begin{equ}
T(\B) = \bigoplus_{m=0}^\infty T(\B)_{(m)}\;,
\end{equ}
where $T(\B)_{(m)}$ is the vector space spanned by the tensors $\tau_1 \otimes \dots \otimes \tau_k$ for $\tau_i \in \trees$ with $|\tau_1| +  \dots + |\tau_k| = m$, with the convention $T(\B)_{(0)}=\reals$. 
\par
We will construct a path $\xbar$ taking values in the vector space $\B_N$. Since $\B_1 \cong \reals^d$, to say that $\xbar$ is an extension of $X$ means that $\pi_{\B_1}(\xbar)=X$. The geometric rough path $\Xbar$ will be built in the space $T^{(N)}(\B_N)$ defined by \eqref{e:TBN_truncated}, satisfying $\innerprod{\Xbar_{st},\tau} = \delta \xbar^{\tau}_{st}$ for each $\tau \in \trees_N$. Moreover, the tensor components must be interpreted (formally) as candidates for the iterated integrals
\begin{equ}
\inner{\Xbar_{st},\tau_1 \otimes \dots \otimes \tau_n} \defin \int_s^t \dots \int_s^{v_2} d\xbar^{\tau_1}_{v_1} \dots d\xbar^{\tau_n}_{v_n}\;,
\end{equ}
which cannot be defined in the usual Riemann sense. 
\par
Recall that $\hopf$ has the decomposition
\begin{equ}
\hopf = \bigoplus_{m=0}^\infty \hopf_{(m)}\;,
\end{equ}
where $\hopf_{(m)}$ is the vector space spanned by $\forest_{(m)}$, the set of all $\tau_1 \dots \tau_k \in \forest$ with $|\tau_1| + \dots + |\tau_k|=m$.  The construction of $\Xbar$ relies on the following graded morphism of Hopf algebras, that is, a linear map $\psi : \hopf_{(m)} \to T(\B)_{(m)}$ for each $m\in \naturals$, which is a morphism with respect to products and coproducts.   
\begin{lemma}\label{lem:psiexist}
There exists a graded morphism of Hopf algebras $\psi: (\hopf,\cdot,\Delta) \to (T(\B),\shuffle,\Deltabar)$ satisfying
\begin{equ}\label{e:psi_condition}
\psi(\tau) = \tau + \psi_{n-1}(\tau)\;,
\end{equ}
for any $\tau \in \trees_n$, where $\psi_{n-1}$ denotes the projection of $\psi$ onto $T(\B_{n-1})$.
 \end{lemma}
To illustrate the property \eqref{e:psi_condition}, consider the following example. In the unlabelled case $d=1$, we will see that
\begin{equ}
\psi(\mytreetwoone{}{}{}) = \mytreetwoone{}{}{} + 2 \bullet \otimes \bullet \otimes \bullet + 2 \bullet \otimes\; \ntreeoneone \;.
\end{equ} 
Thus, we have 
\begin{equ}
 \psi(\mytreetwoone{}{}{}) = \mytreetwoone{}{}{} + \psi_2(\mytreetwoone{}{}{})\;,
\end{equ}
where 
\begin{align*}
\psi_2(\mytreetwoone{}{}{}) = \pi_{T(\B_2)} \psi(\mytreetwoone{}{}{}) &= \pi_{T(\B_2)} \big( \mytreetwoone{}{}{} + 2 \bullet \otimes \bullet \otimes \bullet + 2 \bullet \otimes\; \ntreeoneone \big)\\
&= 2 \bullet \otimes \bullet \otimes \bullet + 2 \bullet \otimes\; \ntreeoneone\;.
\end{align*}
Notice that $\psi_2$ describes all the ways of \emph{cutting apart} the tree $\mytreetwoone{}{}{}$, this is essentially how $\psi$ is defined in general. 
\begin{proof}[Proof of Lemma \ref{lem:psiexist}]We will construct $\psi$ on each $\hopf_{(n)} $. For $n=1$, the condition \eqref{e:psi_condition} forces 
\begin{equ}
\psi(\bullet_{a}) = \bullet_{a},
\end{equ}
for each $a=1\dots d$. Hence, $\psi : \hopf_{(1)} \to \B_1 = T(\B)_{(1)}$, and it is trivial to check that $\psi$ is a morphism of Hopf algebras. Suppose that we have constructed such a map on $\hopf_{(k)}$, for all $1\leq k\leq n-1$. We will now construct an extension of $\psi$ to $\forest_{(n)}$ and hence $\hopf_{(n)}$. Elements in $\forest_{(n)}$ are either $\tau \in \trees_n$ or products of elements in $\forest_{(p)}$ and $\forest_{(q)}$ for $p+q=n$. We will firstly extend $\psi$ to $\trees_n$. 
\par
 Let $\tau \in \trees_n$ with $\Delta \tau = \tau^1 \otimestilde \tau^2 + 1\otimestilde \tau + \tau \otimestilde 1$, for some $\tau \in \trees_n$, where we sum over the non-trivial parts $\tau^1, \tau^2$. We define 
\begin{equ}\label{e:psi_definition}
\psi_{n-1}(\tau) = \psi(\tau^1)\otimes \tau^2\;.
\end{equ}
We then set $\psi(\tau) = \psi_{n-1}(\tau) + \tau$. To complete the extension we set
\begin{equ}
\psi(h_1 h_2 ) = \psi(h_1)\shuffle \psi(h_2)\;,
\end{equ}
for $h_1 h_2 \in \forest_{(n)}$ with $h_1 \in \forest_{(p)}$ and $h_2 \in \forest_{(q)}$. By construction, $\psi$ satisfies \eqref{e:psi_condition} and is a graded morphism of algebras on $\forest_{(n)}$, hence we only need that
\begin{equ}\label{e:coproductcondition}
(\psi\otimest\psi)\Delta h  = \Deltabar \psi(h)\;, 
\end{equ}
for all $h \in \forest_{(n)}$. For $\tau \in \trees_n$, we have that
\begin{equ}\label{e:psi1}
\Deltabar \psi(\tau) = \Deltabar(\psi(\tau^1)\otimes \tau^2 + \tau) = \Deltabar(\psi(\tau_1)\otimes \tau^2) + \tau\otimestilde 1 + 1\otimestilde \tau\;.
\end{equ}
It is easy to see that 
\begin{equ}
\Deltabar(\psi(\tau^1)\otimes \tau^2) = (\psi(\tau^1)\otimes \tau^2)\otimestilde 1+ (\Deltabar\psi(\tau^1))\otimes (1\otimestilde\tau^2) \;.
\end{equ}
Since $\tau^1 \in \forest_{(n-1)}$, the inductive hypothesis implies that \eqref{e:psi1} equals
\begin{equ}
(\psi(\tau^1)\otimes \tau^2 + \tau)\otimestilde 1 + (\psi\otimestilde \psi)(\Delta \tau^1)\otimes(1\otimestilde \tau^2) + 1\otimestilde \tau\;.
\end{equ}
Using the notation, $(\Delta' \otimestilde \Id)\Delta' \tau  = \tau^{11} \otimestilde \tau^{12} \otimestilde \tau^{2}$, the above equals
\begin{equ}\label{e:psi2}
\psi(\tau)\otimestilde 1 + \psi(\tau^{11})\otimestilde (\psi(\tau^{12})\otimes \tau^2) + 1\otimestilde (\psi(\tau^1)\otimes \tau^2) + \psi(\tau^1)\otimestilde \tau^2 + 1\otimestilde \tau\;.
\end{equ}
On the other hand, using the notation $(\Id \otimestilde\Delta' )\Delta' \tau = \tau^{1}\otimestilde \tau^{21}\otimestilde \tau^{22}$, we have that
\begin{equ}
(\psi\otimestilde \psi) \Delta \tau = \psi(\tau)\otimestilde 1 + 1\otimestilde \psi(\tau) + \psi(\tau^1)\otimestilde (\psi(\tau^{21})\otimes \tau^{22}+\tau^2)\;.
\end{equ}
Hence, it is sufficient to check that
\begin{equ}\label{e:psi3}
\psi(\tau^{11})\otimestilde (\psi(\tau^{12})\otimes \tau^2) = \psi(\tau^1)\otimestilde \psi(\tau^{21})\otimes \tau^{22}\;.
\end{equ}
But, from the coassociativity of the coproduct (and hence the reduced coproduct), we have that 
\begin{equ}
\tau^{11}\otimestilde \tau^{12}\otimestilde \tau^2 = (\Delta' \otimestilde \Id)\Delta' \tau = (\Id \otimestilde\Delta' )\Delta' \tau = \tau^{1}\otimestilde \tau^{21}\otimestilde \tau^{22}
\end{equ}
and \eqref{e:psi3} clearly follows. The fact that \eqref{e:coproductcondition} holds for the product $h_1 h_2$ follows easily from the inductive hypothesis, and the fact that $\Deltabar$ and $\Delta$ are morphisms with respect to $\shuffle$ and $\cdot$ respectively.  
\end{proof}

We can now state the main result of this section. 
\begin{thm}\label{thm:branchedgeometric}
Let $X=(X^i)_{i=1\dots d}$ be a path in $\reals^d$ and  $\X$ a $\gamma$-H\"older continuous branched rough path satisfying $\innerprod{\X_{t},\bullet_{i}}=X^i_t$. Then there exists 
\begin{enumerate}
\item a path $\xbar = (\xbar^\tau)_{\tau\in\trees_N}$ taking values in the vector space $\B_N$ and satisfying $\pi_{\B_1}(\xbar)=X$\;,
\item a $\gamma$-H\"older geometric rough path $\Xbar$ in $T^{(N)}(\B_N)$ satisfying $\innerprod{\Xbar_{st},\tau} = \delta\xbar^\tau_{st}$ for each $\tau \in \trees_N$\;,
\end{enumerate}
such that
\begin{equ}\label{e:psiX}
\innerprod{\X_{st},h} = \innerprod{\Xbar_{st},\psi(h)}\;,
\end{equ}
for every $h \in \hopf_N$ and where $\psi$ is the map constructed in Lemma \ref{lem:psiexist}. 
\end{thm}
The idea behind the proof is to construct $\Xbar$ iteratively, using the extension theorem \ref{thm:extension}. The first part of the iteration is to extend the path $X$. To start the iteration, we define the intermediate extension $\Xhat^{(1)} : [0,T] \to  G^{(N)}(\B_1)$ by  
\begin{equ}\label{e:xhat1_def}
\Xhat^{(1)}_{t} = \exp \left( \sum_{i=1}^d \innerprod{\X_{t},\bullet_i}\;\bullet_i \right).
\end{equ}
Hence, we have that $\innerprod{\Xhat_{t}^{(1)},\bullet_i} = \innerprod{\X_{t},\bullet_i}$ and 
\begin{equ}
\innerprod{\Xhat_{t}^{(1)},\bullet_{a_1}\otimes \dots \otimes \bullet_{a_k}} = \frac{1}{k!} \innerprod{\X_{t},\bullet_{a_1} \dots \bullet_{a_k}} \;,
\end{equ} 
for $a_i=1\dots d$ and $k\leq N$. All we have done is extend $X$ by adding the purely symmetric tensor components. Let $K_1$ be the normal subgroup of $G^{(N)}(\B_1)$ defined by 
\begin{equ}
K_1 = \exp \left(W_2(\B_1)\oplus  \dots \oplus W_N(\B_1) \right)\;,
\end{equ}
or equivalently, let $K_1 = \exp L_1$, where $L_1$ is the Lie ideal generated by 
\begin{equ}[]
[\B_1,\B_1]_{\otimes} \defin \spn\{ \bullet_i \otimes \bullet_j - \bullet_j \otimes \bullet_i : i,j=1 \dots d\}. 
\end{equ}
In general, the path $\Xhat^{(1)}_t$ is not a $\gamma$-H\"older continuous path in the group $G^{(N)}(\B_1)$, but it is in the quotient group $G^{(N)}(\B_1)/K_1$. Indeed, we have that 
\begin{equ}\label{e:xhat1_norm}
\norm{(\Xhat^{(1)}_{s})^{-1}\otimes \Xhat^{(1)}_{t}}_{G^{(N)}(\B_1)/K_1} = \inf_{ k \in K_1} \norm{(\Xhat^{(1)}_{s})^{-1}\otimes \Xhat^{(1)}_{t}\otimes k}_{G^{(N)}(\B_1)} \;,
\end{equ}
and by the Baker-Campbell-Hausdorff formula, 
\begin{align*}
(\Xhat^{(1)}_{s})^{-1}\otimes \Xhat^{(1)}_{t} &= \exp \left(\sum_{i=1}^d\big(\inner{\X_{t},\bullet_{i}} - \inner{\X_{s},\bullet_{i}}\big)\bullet_i  \right)\otimes \exp (\ell)\\ 
&= \exp \left( \sum_{i=1}^d\inner{\X_{st},\bullet_{i}}\bullet_i \right)\otimes \exp (\ell)\;,
\end{align*}
where $\ell \in L_1$. Hence, taking $k=\exp(-\ell)$, we can bound \eqref{e:xhat1_norm} by
\begin{equ}
\norm{\exp \left( \sum_{i=1}^d\inner{\X_{st},\bullet_{i}}\bullet_i \right)}_{G^{(N)}(\B_1)} \leq C \sum_{i=1}^d |\innerprod{\X_{st},\bullet_i}| \leq C|t-s|^\gamma\;,
\end{equ}
which proves the claim for $\Xhat^{(1)}$. We can therefore apply the extension theorem to $\Xhat^{(1)}$, in particular it follows that there exists a $\gamma$-H\"older continuous path $\Xbar^{(1)} \in G^{(N)}(\B_1)$ such that 
\begin{equ}
\pi_{G^{(N)}(\B_1)/K_1} ( \Xbar^{(1)} ) = \Xhat^{(1)}\;,
\end{equ}
which simply means that $\inner{\Xbar^{(1)}_{st},\bullet_i} = \inner{\Xhat^{(1)}_{st},\bullet_i} = \delta X_{st}$ for all $i=1\dots d$.
\begin{rmk}
We should mention that one can actually choose \emph{any} geometric rough path $\Xbar^{(1)}$ above $X$. We only use the choice of $\Xbar^{(1)}$ provided by the extension theorem as it will work for every $X$. 
\end{rmk}

The second part of the iteration relies on a generalisation of the following well-known (and easily verified) fact. Namely, that the difference between two area processes over a common path is equal to the increment of another path. In our case, for each $a,b=1\dots d$ there exists a path 
\begin{equ}\label{e:xbar_condition}
\xbar^{\mytreeoneone{a}{b}}\;\; : [0,T]\to \reals \quad \text{such that} \quad\delta \xbar^{\mytreeoneone{a}{b}}_{st} =  \innerprod{\X_{st} ,\mytreeoneone{a}{b}\;\;} - \innerprod{\Xbar^{(1)}_{st},\mytreeone{a}\;\;\otimes \mytreeone{b}\;\;}, 
\end{equ}
where $\Xbar^{(1)}_{st} = (\Xbar^{(1)}_{s})^{-1}\otimes\Xbar^{(1)}_t$ and the path is unique up to an additive constant. We add $\xbar$ as another component of a new path $\Xhat^{(2)}: [0,T] \to T^{(N)}(\B_2)$. To be precise, we define
\begin{equ}
\Xhat^{(2)}_{t} = \exp \left( \log \Xbar^{(1)}_{t} + \sum_{a,b=1}^d \xbar^{\mytreeoneone{a}{b}}_{t}\;\; \mytreeoneone{a}{b}\;\; \right)\;.
\end{equ}
Hence, $\Xhat^{(2)}$ satisfies
\begin{equ}\label{e:xhat2}
\innerprod{\Xhat^{(2)}_{t} ,\mytreeoneone{a}{b}\;\;} = \xbar^{\mytreeoneone{a}{b}}_{t}\;,
\end{equ}
for all $a,b=1\dots d$,
\begin{equ}
\innerprod{\Xhat_{t}^{(2)},\tau_1 \otimes \dots \otimes \tau_k} = \frac{1}{k!}\innerprod{\Xhat_{t}^{(2)},\tau_1} \dots \innerprod{\Xhat_{t}^{(2)},\tau_k}\;,
\end{equ}
for all tensors $\tau_1 \otimes \dots \otimes \tau_k \in T^{(N)}(\B_2) \setminus T^{(N)}(\B_1) $, and $\Xhat^{(2)}$ is an extension of $\Xbar_1$, in the sense that 
\begin{equ}
\pi_{T^{(N)}(\B_1)}(\Xhat^{(2)}) = \Xbar^{(1)}\;. 
\end{equ}
We then repeat the first step, by finding the right quotient group and re-applying the extension theorem. To this end, for any integer $1\leq n \leq N$, we define $L_n$ as the Lie ideal generated by the set 
\begin{equ}[]
[\B_{(n)},\B_n]_\otimes \defin \spn \{ \tau_1 \otimes \tau_2 - \tau_2 \otimes \tau_1 : \text{$\tau_1,\tau_2 \in \trees$ with $|\tau_1|=n$ and $|\tau_2|\leq n$ } \}
\end{equ}
in the free Lie algebra $\G^{(N)}(\B_n)$. In particular, $L_n$ contains all brackets in $\G^{(N)}(\B_n)$ with at least one factor from $\B_{(n)}$. In order to construct meaningful quotients, we require the following Lemma.   
\begin{lemma}
For each $1\leq n \leq N$, $K_n = \exp(L_n)$ is a normal subgroup of $G^{(N)}(\B_n)$.
\end{lemma}
\begin{proof}
The statement is an elementary result in the theory of Lie algebras, see \cite[Theorem 3.22]{kirillov}, for instance.  
\end{proof}
%
After this step, we will obtain a geometric rough path $\Xbar^{(2)}$ above the path 
\begin{equ}
\xbar^{(2)} = (X^i, \xbar^{\mytreeoneone{j}{k}} \;\;: i,j,k=1\dots d)\;,
\end{equ}
and from \eqref{e:xbar_condition}, it follows that $\Xbar^{(2)}$ contains the information held in the components $\inner{\X,\tau}$ for all $\tau \in \trees_2$. Hence, if we repeat this procedure, we eventually obtain a geometric rough path $\Xbar\defin\Xbar^{(N)}$, containing  \emph{all} the information of held in the components of $\X$.  
\begin{proof}[Proof of Theorem \ref{thm:branchedgeometric}]
Throughout the proof, we will denote $\Xbar^{(n)}_{st} = (\Xbar^{(n)}_{s})^{-1}\otimes \Xbar^{(n)}_{t}$. Proceeding by induction, we will prove that, for each integer $n\geq 1$, there exists a $\gamma$-H\"older continuous path $\Xbar^{(n)}: [0,T] \to G^{(N)}(\B_n)$ such that   
\begin{equ}\label{e:psiclaim}
\innerprod{\X_{st},h} = \innerprod{\Xbar^{(n)}_{st},\psi(h)}\;,
\end{equ}
for every $h \in \hopf_n$. For $n=1$, we know from the introductory argument that such a construction is possible. Hence, assume the claim holds for some $n\geq 1$. We will now construct $\Xbar^{(n+1)}$ and show that \eqref{e:psiclaim} holds in the $n+1$ case. We will first show that for every $\tau\in \trees_{(n+1)}$ there exists a path $\xbar^{\tau}: [0,T] \to \reals$ such that 
\begin{equ}\label{e:xbarh}
\delta \xbar^{\tau}_{st} =  \innerprod{\X_{st},\tau}-\innerprod{\Xbar^{(n)}_{st},\psi_n(\tau)}\;,
\end{equ}
unique up to an additive constant, this path will allow us to define $\Xbar^{(n+1)}$. 
Without loss of generality, let $\tau = [h]$, for some $h\in \forest_{(n)}$, where we omit the label of the root. We have that 
\begin{equ}\label{e:Xhclaim}
\innerprod{\X_{st},[h]} = \innerprod{\X_{su},[h]} + \innerprod{\X_{ut},[h]} + \innerprod{\X_{su},h^{1}}\innerprod{\X_{ut},[h^{2}]}  \;,
\end{equ}
where $\Delta h = h^{1}\otimestilde h^{2}  + 1\otimestilde h + h \otimestilde 1 $ and we omit the summation. By hypothesis, we have that 
\begin{equ}
\innerprod{\X_{su},h^{1}}\innerprod{\X_{ut},[h^{2}]} = \innerprod{\Xbar^{(n)}_{su},\psi(h^{1})}\innerprod{\Xbar^{(n)}_{ut},\psi([h^{2}])}\;,
\end{equ}
since $h^{1}$ and $[h^{2}]$ are elements of $\hopf_n$. Moreover, by definition of $\psi$, we have that 
\begin{equ}
\psi(h^{1})\otimestilde \psi([h^2]) = (\psi\otimestilde \psi) \Delta ' [h] = \Delta ' \psi([h]) \;,
\end{equ}
where $\Delta'$ is the reduced coproduct. This yields the identity
\begin{equ}
\innerprod{\X_{su},h^{1}}\innerprod{\X_{ut},[h^{2}]} = \innerprod{\Xbar^{(n)}_{st},\psi_n([h])} - \innerprod{\Xbar^{(n)}_{su},\psi_n([h])} - \innerprod{\Xbar^{(n)}_{ut},\psi_n([h])}\;,
\end{equ}
combining this with \eqref{e:Xhclaim}, we obtain
\begin{equ}
\innerprod{\X_{st},[h]} - \innerprod{\Xbar^{(n)}_{st},\psi_n([h])} =   \innerprod{\X_{su},[h]}-\innerprod{\Xbar^{(n)}_{su},\psi_n([h])} + \innerprod{\X_{ut},[h]} -\innerprod{\Xbar^{(n)}_{ut},\psi_n([h])}\;.
\end{equ}
Setting $\tau = [h]$, this implies the existence of $\xbar^{\tau}$ for each $\tau \in \trees_{(n+1)}$, satisfying \eqref{e:xbarh}. We include this path in our construction by defining the intermediate extension $\Xhat^{(n+1)}$ of $\Xbar^{(n)}$, setting
\begin{equ}\label{e:xhat_def}
\Xhat^{(n+1)}_{t} = \exp \left( \log \Xbar_{t}^{(n)} + \sum_{\tau \in \trees_{(n+1)}} \xbar^{\tau}_{t} \tau \right)\;.
\end{equ}
Hence, $\Xhat^{(n+1)}: [0,T] \to G^{(N)}(\B_{n+1})$ and satisfies $\innerprod{\Xhat^{(n+1)}_{t},\tau} = \xbar^\tau_{t}$ for all $\tau \in \trees_{(n+1)}$,
\begin{equ}
\innerprod{\Xhat^{(n+1)}_t,\tau_1\otimes\dots\otimes\tau_m} = \frac{1}{m!} \innerprod{\Xhat^{(n+1)}_{t},\tau_1}\dots \innerprod{\Xhat^{(n+1)}_{t},\tau_m} \;,
\end{equ}
for all $\tau_1\otimes \dots \otimes \tau_m \in T^{(N)}(\B_{n+1})\setminus T^{(N)}(\B_{n})$ and $\Xhat^{(n+1)}$ is an extension of $\Xbar^{(n)}$ is the sense that
\begin{equ}
\pi_{T^{(N)}(\B_{n})}(\Xhat^{(n+1)}) = \Xbar^{(n)}\;.
\end{equ}
We then have the following crucial fact, which we shall verify in the sequel. 
\begin{lemma}\label{lem:quotient_path}
For each $n\leq N-1$, the intermediate extension $\Xhat^{(n+1)}$ is a $\gamma$-H\"older continuous path in the quotient group $G^{(N)}(\B_{n+1})/K_{n+1}$. 
\end{lemma}
Thus, from the extension  theorem \ref{thm:extension}, we know that there exists a $\gamma$-H\"older path $\Xbar^{(n+1)}: [0,T] \to G^{(N)}(\B_{n+1})$ satisfying
\begin{equ}
\pi_{G^{(N)}(\B_{n+1})/K_{n+1}}(\Xbar^{(n+1)}) = \Xhat^{(n+1)}\;.
\end{equ}
We will now check that $\Xbar^{(n+1)}$ satisfies \eqref{e:psiclaim} for the basis elements $\forest_{n+1}$ and hence $\hopf_{n+1}$. Firstly, suppose $h \in \forest_n$, then $\psi(h) \in  T^{(N)}(\B_n)$, which follows from the fact that $\psi$ is graded. Moreover, since $\Xbar^{(n+1)}$ agrees with $\Xbar^{(n)}$ on $T^{(N)}(\B_n)$, we have that 
\begin{equ}
\inner{\X_{st},h} = \inner{\Xbar^{(n)}_{st},\psi(h)} = \inner{\Xbar^{(n+1)}_{st},\psi(h)}\;, 
\end{equ}
which proves the claim for $\forest_n$. It is clear that every element in $\forest_{(n+1)}$ is either a tree $[h]$ for some $h\in\forest_{(n)}$ or a  product $h_1 h_2$ for $h_1, h_2 \in \forest_n$. For the tree case, we have the identity
\begin{equ}\label{e:xbar_proof_1}
\innerprod{\Xbar^{(n+1)}_{st},[h]} = \innerprod{(\Xbar^{(n+1)}_{s})^{-1}\otimes \Xbar^{(n+1)}_{t},[h]} = \innerprod{(\Xhat^{(n+1)}_{s})^{-1}\otimes \Xhat^{(n+1)}_{t},[h]} = \delta \xbar^{[h]}_{st} \;,
\end{equ}
where we have used the facts that $\Xbar^{(n+1)}$ and $\Xhat^{(n+1)}$ coincide on $[h]$ and that $\inner{\Xhat^{(n+1)}_t,[h]} = \xbar^{[h]}_t$ and  $\inner{(\Xhat^{(n+1)}_s)^{-1},[h]} = -\xbar^{[h]}_s$. And by definition, 
\begin{equ}
 \delta \xbar^{[h]}_{st} =  \innerprod{\X_{st},[h]} -\innerprod{\Xbar^{(n)}_{st},\psi_n([h])}= \innerprod{\X_{st},[h]} -\innerprod{\Xbar^{(n+1)}_{st},\psi_n([h])}\;, 
\end{equ}
where the last equality follows from the fact that $\psi_n([h]) \in T^{(N)}(\B_n)$, on which $\Xbar^{(n+1)}$ and $\Xbar^{(n)}$ agree. Combining this with \eqref{e:xbar_proof_1}, the claim follows from the condition $\psi([h]) = [h] +  \psi_n ([h])$. For the product case,   
\begin{equ}
\innerprod{\X_{st},h_1 h_2} = \innerprod{\X_{st},h_1}\innerprod{\X_{st},h_2} = \innerprod{\Xbar^{(n)}_{st},\psi(h_1)}\innerprod{\Xbar^{(n)}_{st},\psi(h_2)}\;.
\end{equ}
Since $\Xbar^{(n)}$ is geometric, the above equals
\begin{equ}
\innerprod{\Xbar_{st}^{(n)},\psi(h_1) \shuffle \psi(h_2)} =  \innerprod{\Xbar^{(n+1)}_{st},\psi(h_1)\shuffle \psi( h_2)}  = \innerprod{\Xbar^{(n+1)}_{st},\psi(h_1 h_2)}\;.
\end{equ} 
where the first equailty follows from the fact that $\psi(h_1)\shuffle \psi(h_2) \in T^{(N)}(\B_n)$, on which $\Xbar^{(n)}$ and $\Xbar^{(n+1)}$ coincide and the second follows from the fact that $\psi$ is a morphism with respect to multiplication. 
\end{proof}
\begin{proof}[Proof of Lemma \ref{lem:quotient_path}] 
By the Baker-Campbell-Hausdorff formula, we have that
\begin{align*}
\Xhat^{(n+1)}_{st} : &= (\Xhat^{(n+1)}_s)^{-1}\otimes \Xhat^{(n+1)}_t\\
 &= \exp\bigg(\sum_{\tau \in \trees_{(n+1)}} -\xbar^{\tau}_{s} \tau  + \ell_1   \bigg) \otimes (\Xbar^{(n)}_s)^{-1} \otimes \Xbar^{(n)}_t \otimes \exp\bigg(\sum_{\tau \in \trees_{(n+1)}} \xbar^{\tau}_{t} \tau  + \ell_2   \bigg) \\
 &=\exp\bigg(\sum_{\tau \in \trees_{(n+1)}} -\xbar^{\tau}_{s} \tau  + \ell_1   \bigg) \otimes  \Xbar^{(n)}_{st} \otimes \exp\bigg(\sum_{\tau \in \trees_{(n+1)}} \xbar^{\tau}_{t} \tau  + \ell_2   \bigg)\\
&= \Xbar^{(n)}_{st}\otimes \exp\bigg(\sum_{\tau \in \trees_{(n+1)}} \delta \xbar^{\tau}_{st} \tau \bigg) \otimes \exp(\ell_3)\;,
\end{align*}
where $\ell_1,\ell_2$ are linear combinations of brackets between $\log \Xbar^{(n)}$ and $\trees_{(n+1)}$ and are therefore in the ideal $L_{n+1}$, and where $\ell_3$ is a linear combination of brackets between $\log\Xbar^{(n)}$, $\trees_{(n+1)}$, $\ell_1$ and $\ell_2$ and is therefore also in $L_{n+1}$. By taking $k=\exp(-\ell_3)$, we therefore have that
\begin{align*}
\norm{\Xhat^{(n+1)}_{st}}_{G^{(N)}(\B_{n+1})/K_{n+1}} &= \inf_{k\in K_{n+1}} \norm{\Xbar^{(n)}_{st}\otimes \exp \bigg(\sum_{\tau \in \trees_{(n+1)}} \delta \xbar^{\tau}_{st} \tau \bigg)\otimes \exp(\ell_3)\otimes k}_{G^{(N)}(\B_{n+1})}\\
& \leq \norm{\Xbar^{(n)}_{st}\otimes \exp \bigg(\sum_{\tau \in \trees_{(n+1)}} \delta \xbar^{\tau}_{st} \tau \bigg)}_{G^{(N)}(\B_{n+1})}\\
&\leq \norm{\Xbar^{(n)}_{st}}_{G^{(N)}(\B_{n+1})} + \norm{ \exp \bigg(\sum_{\tau \in \trees_{(n+1)}} \delta \xbar^{\tau}_{st} \tau \bigg)}_{G^{(N)}(\B_{n+1})}
\end{align*}
where in the last inequality we have used the sub-additivity property of $\norm{\cdot}_{G^{(N)}(\B_{n+1})}$. For the first term, using the equivalence of norms on $G^{(N)}(\B_{n})$, we have that
\begin{align*}
\norm{\Xbar^{(n)}_{st}}_{G^{(N)}(\B_{n+1})} &\leq C \sum_{m=1}^N \sum_{\{\tau_1,\dots,\tau_m\} \subset \trees_{n+1}} |\innerprod{\log\Xbar^{(n)}_{st},\tau_1 \otimes\dots\otimes \tau_m}|^{1/m}\\
 &= C \sum_{m=1}^N \sum_{\{\tau_1,\dots,\tau_m\} \subset \trees_{n}} |\innerprod{\log\Xbar^{(n)}_{st},\tau_1 \otimes\dots\otimes \tau_m}|^{1/m}\\ &\leq C \norm{\Xbar^{(n)}_{st}}_{G^{(N)}(\B_{n})} \leq C|t-s|^\gamma\;,
\end{align*}
since $\Xbar^{(n)}$ is a $\gamma$-H\"older continuous path in ${G^{(N)}(\B_{n})}$. For the second term, we have that 
\begin{equ}
\norm{ \exp \bigg(\sum_{\tau \in \trees_{(n+1)}} \delta \xbar^{\tau}_{st} \tau \bigg)}_{G^{(N)}(\B_{n+1})} \leq C\sum_{\tau \in \trees_{(n+1)}} |\delta \xbar^\tau_{st}|\;. 
\end{equ}
And by definition, 
\begin{align*}
|\delta \xbar^\tau_{st}|=  |\innerprod{\X_{st},\tau} -\innerprod{\Xbar^{(n)}_{st},\psi_n(\tau)}| &\leq |\innerprod{\X_{st},\tau}| + |\innerprod{\Xbar^{(n)}_{st},\psi_n(\tau)}|\\ 
&\leq |\innerprod{\X_{st},\tau}| + C\norm{\Xbar^{(n)}_{st}}_{G^{(N)}(\B_{n})} \leq C |t-s|^\gamma \;.
\end{align*}
This completes the proof.
\end{proof}
\begin{rmk}
Throughout the construction, we have ignored the fact that the path elements $\innerprod{\Xbar,\tau}$ actually have $\gamma |\tau|$-H\"older regularity, rather than just $\gamma$. Hence, for each component $\innerprod{\Xbar,\tau_1\otimes\dots\otimes \tau_n}$ with $|\tau_1| + \dots + |\tau_n|> N$, there will be a \emph{canonical} choice, given by defining the component as a Young integral.  
\end{rmk}

If branched rough paths can be written as geometric rough paths, then we should be able to import some of the tools from geometric rough paths to the world of branched rough paths. The following result tells us that the extension theorem \ref{thm:extension} can also be used on branched rough paths, for a special but very useful class of extension. Namely, if we have a branched rough path $\X^1$ above a path $X = (X^1,\dots,X^d)$ and an \emph{extended} path $\xbar = (X^1,\dots,X^d,\xbar^{d+1},\dots,\xbar^e)$, then there exists a branched rough path $\X^2$ above $\xbar$ which agrees with $\X^1$ on the $X$ components.

\begin{corr}\label{lem:alphabet_extension}
Let $\hopf^{1}, \hopf^2$ be the Connes-Kreimer Hopf algebras generated by the alphabets $\alphabet_1$ and $\alphabet_2$ respectively, where $\alphabet_1 \subset \alphabet_2$, so that $\hopf^1$ is a sub Hopf algebra of $\hopf^2$. Let $X = (X^i)_{i\in\alphabet_1}$ and $\xbar = (\xbar^i)_{i\in\alphabet_2}$ be two $\gamma$-H\"older continuous paths with $\xbar^i = X^i$ when $i \in \alphabet_1$. Let $\X^1$ be a branched rough path on $\hopf^1$ with $\inner{\X^1_{st},\bullet_i} = \delta X^i_{st}$ for each $i \in \alphabet_1$. Then there exists a branched rough path $\X^2$ on $\hopf^2$ with $\inner{\X^2_{st},\bullet_i} = \delta \xbar^i_{st}$ for each $i\in\alphabet_2$ and $\X^2$ is an extension of $\X^1$ in the sense that
\begin{equ}\label{e:add_path1}
\inner{\X^2_{st}, h} = \inner{\X^1_{st},h}\;,
\end{equ}
for every $h \in \hopf^1$. 
\end{corr}
\begin{proof}
Without loss of generality, assume $\alphabet_1 = \{1,\dots,d \}$ and $\alphabet_2 = \{1,\dots,d+1 \}$, so that $X = (X^1,\dots,X^d)$ and $\xbar = (X^1,\dots,X^d,X^{d+1})$. Let $\B^1_N$ and $\B_N^2$ be the vector spaces spanned by the trees $|\tau| \leq N$, with vertex decorations from $\alphabet_1$, $\alphabet_2$ respectively. Let $\psi^1 : \hopf^1 \to T(\B_N^1)$ be the map constructed in Lemma \ref{lem:psiexist} and similarly for $\psi^2 : \hopf^2 \to T(\B_N^2)$. Clearly, we have that $\psi^1(h) = \psi^2(h)$ for $h \in \hopf^1$. From Theorem \ref{thm:branchedgeometric}, we know that there exists a geometric rough path $\Xbar^{1}$ on $\Ntensor(\B^1_N)$ satisfying 
\begin{equ}
\inner{\X^1_{st},h} = \inner{\Xbar^{1}_{st},\psi^1(h)}\;.
\end{equ}
Now define $\Xhat^1 : [0,T] \to \Ntensor(\B_N^2)$ by
\begin{equ}
\Xhat^1_{t} = \exp \left( \log \Xbar^1_t + \xbar^{d+1}_t \bullet_{d+1} \right)\;.
\end{equ}
Using the same techniques employed in Lemma \ref{lem:quotient_path}, one can show that $\Xhat^1$ is a $\gamma$-H\"older continuous path in the quotient group $G^{(N)}(\B^2_N )/K$, where $K=\exp L$ and $L$ is the Lie ideal in $\Ntensor(\B_N^2)$ generated by $[\B_N^2, \bullet_{d+1}]$. From the Lyons-Victoir extension theorem \ref{thm:extension}, there exists a $\gamma$-H\"older path $\Xbar^2: [0,T] \to G^{(N)}(\B_N^2)$ satisfying 
\begin{equ}
\inner{\Xbar^2_{st}, u} = \inner{\Xhat^1_{st},u}\;,
\end{equ}
for all $u \in \Ntensor(\B^1_N)$ and $\inner{\Xbar^2_{st},\bullet_{d+1}} = \delta X^{d+1}_{st}$. We then define $\X^2$ in $\hopf^2$ by 
\begin{equ}
\inner{\X_{st}^2,h} = \inner{\Xbar^2_{st},\psi^2(h)}\;,
\end{equ} 
It follows from the properties and $\psi^2$ that $\X^2$ is indeed a branched rough path. Now, let $h \in \hopf^1$ then we have that 
\begin{equ}
\inner{\X^2_{st},h} = \inner{\Xbar^2_{st},\psi^2(h)} = \inner{\Xbar^2_{st},\psi^1(h)} = \inner{\Xhat^1_{st},\psi^1(h)} = \inner{\Xbar^1_{st},\psi^1(h)} = \inner{\X_{st},h}\;,
\end{equ}
which proves \eqref{e:add_path1}. Moreover, because $\psi^2(\bullet_{d+1}) = \bullet_{d+1}$, we have that
\begin{equ}
\inner{\X^2_{st}, \bullet_{d+1}} = \inner{\Xbar^2_{st},\bullet_{d+1}} = \delta \xbar^{d+1}\;,
\end{equ}
which shows that $\X^2$ is a branched rough path above $\xbar$ and hence completes the proof. 
\end{proof}

\section{Conversion formula}\label{sec:RDEnewbasis}
If $\bigy$ is the solution to the controlled rough path equation \eqref{e:control_fixpoint} with $\inner{1,\bigy}=Y$, then from Proposition \ref{prop:euler} we have that
\begin{equ}\label{e:con_intro}
\delta Y_{st} = \sum_{\tau \in \trees_N}f_\tau(Y_s)\inner{\X_{st},\tau} + r_{st}\;,
\end{equ}
where the coefficients $f_\tau(Y_s) = \inner{\tau,\bigy_s}$ are determined by \eqref{e:recurrence} with $f_{\bullet_i}=f_i$. In Section \ref{sec:geometric}, we saw that for every branched rough $\X$, there exists a geometric rough path $\Xbar$ taking values in $T^{(N)}(\B_N)$ and satisfying
\begin{equ}
\innerprod{\X_{st},\tau} = \innerprod{\Xbar_{st},\psi(\tau)} \;,
\end{equ}
where $\psi$ is the map derived in Lemma \ref{lem:psiexist}. If we apply this transformation to \eqref{e:con_intro}, we see that
\begin{equ}\label{e:ito_strato1}
\sum_{\tau \in \trees_N} f_\tau(Y_s)\innerprod{\X_{st},\tau} = \sum_{\sigma \in U_{N,N}} f_{\psi^*(\sigma)}(Y_s) \innerprod{\Xbar_{st},\sigma} 
\end{equ}  
where $\psi^*: T((\B_N)) \to \hopf^*$ is the adjoint of $\psi$, where $f_{\psi^*(\sigma)} = \sum_\tau \inner{\psi^*(\sigma),\tau}f_\tau$ and where 
\begin{equ}\label{e:UNN}
U_{N,n} = \{\tau_1 \otimes \dots \otimes \tau_k : \text{$\tau_i \in \trees_n$ with $k\leq N$} \}
\end{equ}
is the set of basis tensors for $T^{(N)}(\B_n)$. Since $Y$ appears to be controlled by $\Xbar$, it is natural to ask whether $Y$ solves an RDE driven by the geometric rough path $\Xbar$, providing a generalised It\^o-Stratonovich conversion formula. In Subsection \ref{subsec:geometric} we provide a criterion to determined when expressions controlled by a geometric rough path are solutions to RDEs driven by that geometric rough path. In Subsection \ref{subsec:strato}, namely in Theorem \ref{thm:strato} we derive the It\^o-Stratonovich conversion formula.

\subsection{Geometric RDEs}\label{subsec:geometric}
Let $\Xbar$ be a branched rough path above $\xbar \in \reals^d$ satisfying
\begin{equ}\label{e:geo}
\innerprod{\Xbar_{st},h} = \innerprod{\Xbar_{st},\iota\phi_g(h)}\;,
\end{equ}
for each $h\in\hopf_N$, where $\iota: T(\reals^d) \to \hopf$ is the inclusion map. Hence, $\Xbar$ is a geometric (branched) rough path. Let $Y$ be a controlled rough path solution to the RDE
\begin{equ}\label{e:geometricRDE}
dY_t = f(Y_t)\cdot d\xbar_t \;,
\end{equ}
driven by a geometric rough path $\Xbar$, where $f(Y)\cdot d\xbar = \sum_{i=1}^d f_i(Y)d\xbar^i$ and the vector fields $f_i : \reals^e \to \reals^e$ are smooth. From Proposition \ref{prop:euler}, we have that
\begin{equ}\label{e:geo_intro}
\delta Y_{st} = \sum_{\tau \in \trees_N} f_{\tau}(Y_s) \innerprod{\X_{st},\tau} + r_{st}\;,
\end{equ}
where $|r_{st}| = \liloh(|t-s|)$ and the coefficients $f_\tau$ satisfy the recurrence relation \eqref{e:recurrence} with $f_{\bullet_i} = f_i$. The geometric constraint \eqref{e:geo} allows us to rewrite $\delta Y_{st}$ as an expression controlled by only the linear trees in $\hopf_N$, which we identify with the basis elements of $T^{(N)}(\reals^d)$. To be precise,    
\begin{equ}
\sum_{\tau \in \trees_{N}} f_\tau(Y_s)\innerprod{\Xbar_{st},\tau} = \sum_{\sigma \in U_{N,1} } f_{\phi_g^*(\sigma)}(Y_s) \innerprod{\Xbar_{st},\iota\sigma}\;,
\end{equ}
where $\phi_g^*: T((\reals^d)) \to \hopf^*$ is the adjoint of $\phi_g$, where $f_{\phi_g^*(\sigma)} = \sum_\tau \inner{\phi_g^*(\sigma),\tau}f_\tau$ and 
\begin{equ}\label{e:UN1}
U_{N,1} = \{e_{v_1}\otimes \dots \otimes e_{v_k} : \text{$v_i =1 \dots d$ and $k\leq N$} \}
\end{equ}
denotes the basis tensors of $T^{(N)}(\reals^d)$. Note that only those terms in the subspace $T^{(N)}(\reals^d)$ appear, since all branched trees are in the kernel of $\phi_g^*$. 

\begin{rmk}
From \eqref{e:UNN}, we have
\begin{align*}
U_{N,1} &= \{\bullet_{v_1}\otimes \dots \otimes \bullet_{v_k} : \text{$v_i = 1\dots d$ and $k\leq N$} \} \\&\cong \{e_{v_1}\otimes \dots \otimes e_{v_k} : \text{$v_i =1 \dots d$ and $k\leq N$} \}\;,
\end{align*}
so that \eqref{e:UN1} is not an abuse of notation.
\end{rmk}

We can use this representation to develop another recurrence formula, to characterise those expressions controlled by geometric rough paths that are solutions to a given RDE.
\begin{prop}\label{prop:georecurrence}
Let $\Xbar$ be a geometric (branched) rough path above $\xbar$. Then $\bigybar$ with $\inner{1,\bigybar}=Y$ is the controlled rough path solution to 
\begin{equ}\label{e:LRDE}
dY_t = f(Y_t)\cdot d\xbar_t\;,
\end{equ}
driven by $\Xbar$ if and only if $\inner{\tau,\bigybar_t} = f_\tau(Y_t)$ as defined above and
\begin{equ}\label{e:geo_controlled}
\delta Y_{st} = \sum_{\sigma \in U_{N,1}} F_\sigma(Y_s) \innerprod{\Xbar_{st},\iota\sigma} + r_{st}\;,
\end{equ}
where $|r_{st}| = \liloh(|t-s|)$ and where the coefficients $F_\sigma$ are defined by the recurrence $F_{e_i}= f_i$ and
\begin{equ}\label{e:geometricrecurrence}
F_{e_{v_1}\otimes \dots \otimes e_{v_n}} = F_{e_{v_1}}\cdot DF_{e_{v_2}\otimes \dots \otimes e_{v_n}}\;,
\end{equ}
for any $v_i =1,\dots,d$ and any $n\leq N$. 
\end{prop}

\begin{rmk}
Since each $F_{e_{v_1}\otimes\dots\otimes e_{v_k}} : \reals^e \to \reals^e$, the identity \eqref{e:geometricrecurrence} should be interpreted as
\begin{equ}
F_{e_{v_1}\otimes \dots \otimes e_{v_n}}(Y)_i = F_{e_{v_1}}(Y)_j \del^{j}F_{e_{v_2}\otimes \dots \otimes e_{v_n}}(Y)_i\;,  
\end{equ} 
for each $i=1\dots e$, where $F_{\sigma}(Y)_i$ denotes the $i$-th component.
\end{rmk}

\begin{rmk}
One can also define $\Xbar$-controlled rough paths for a geometric $\Xbar$ on $T^{(N)}(\reals^d)$. These are similarly defined as paths $\bigybar : [0,T]\to T^{(N-1)}(\reals^d)$ satisfying the consistency condition
\begin{equ}
\inner{v,\bigybar_t} = \inner{\Xbar_{st}\otimes v,\bigybar_s} + R_{st}^{v}\;,
\end{equ}
for every tensor $v$ and where $|R_{st}^v| \leq C |t-s|^{(N-|v|)\gamma}$. The new recurrence condition  \eqref{e:geometricrecurrence} is then simply the analogue of the branched rough path recurrence \eqref{e:recurrence} in a geometric controlled rough path setting. In particular, one could also read Proposition \ref{prop:georecurrence} as: The geometric controlled rough path $\bigybar : [0,T]\to T^N(\reals^d)$ is a controlled rough path solution to \eqref{e:LRDE} with $\inner{1,\bigybar}=Y$ if and only if $Y$ satisfies \eqref{e:geo_controlled} where the coefficients $\inner{\sigma,\bigybar_t} = F_\sigma(Y_t)$ are determined by the recurrence \eqref{e:geometricrecurrence} with $F_{e_i} = f_i$. However, since we can already define geometric rough paths as a special class of branched rough paths, we see no need for this extra definition. 
\end{rmk}

\begin{rmk}\label{rmk:different_VS}
Naturally, we can apply Proposition \ref{prop:georecurrence} to any geometric (branched) rough path $\Xbar$ above a path $\xbar$, where $\xbar$ takes values in an arbitrary vector space $V$. For instance, in the next subsection we will have $\xbar$ taking values in $\B_N$, as constructed in Theorem \ref{thm:branchedgeometric}. In this case, the condition \eqref{e:geo_controlled} looks like 
\begin{equ}
\delta Y_{st} = \sum_{\sigma \in U_{N,N}} F_\sigma(Y_s) \inner{\Xbar_{st},\sigma} + r_{st}\;,
\end{equ}
where $U_{N,N}$ is defined by \eqref{e:UNN} and where $F_\sigma$ satisfy
\begin{equ}
F_{\tau_1 \otimes \dots \otimes \tau_n} = F_{\tau_1} \otimes DF_{\tau_2 \otimes \dots \otimes \tau_n}\;,
\end{equ}
for all $\tau_1 \otimes \dots \otimes \tau_n \in U_{N,N}$. 
\end{rmk}

Before proving the proposition, we need the following lemma, which highlights a useful property of the functions $f_\tau$. This lemma will be used in both this subsection and the next. As usual, we will use the notation $f_h = \inner{h,1}\Id +  \sum_{\tau}  \inner{h,\tau}f_\tau$ for any $h\in\hopf^*$. 
\begin{lemma}\label{lem:L_GL}
We have that
\begin{equ}
D^q f_{h}: (f_{\lambda_1},\dots,f_{\lambda_q}) = f_{(\lambda_1 \dots \lambda_q)\star h}\;,
\end{equ}
for any $\lambda_1,\dots,\lambda_q \in \trees^*$ and any $h\in\hopf^*$. 
\end{lemma}
\begin{remark}
In this article we only ever require Lemma \ref{lem:L_GL} in the case $q=1$. However, we include the general statement as it highlights a striking algebraic feature of the coefficients $f_h$. In particular, the algebraic structure of the rough path $\X$ is \emph{twinned} with another algebraic structure on the coefficients of the solution. 
\end{remark}

\begin{proof}[Proof of Proposition \ref{prop:georecurrence}] We will first prove the `only if' statement. From Proposition \eqref{prop:euler}, we know that the controlled rough path solution $\bigybar$ to \eqref{e:LRDE} with $\inner{1,\bigybar}=Y$ satisfies
\begin{equ}
\delta Y_{st} = \sum_{\tau \in \trees_N} f_\tau(Y_s)\innerprod{\Xbar_{st},\tau}+ r_{st} \;,
\end{equ}
where $|r_{st}| = \liloh(|t-s|)$ and has coefficients $\inner{\tau,\bigybar_t} = f_\tau(Y_t)$. Since $\Xbar$ is geometric, we also know that 
\begin{equ}
\delta Y_{st} = \sum_{\sigma \in U_{N,1}} f_{\phi_g^*(\sigma)}(Y_s) \innerprod{\Xbar_{st},\iota\sigma} + r_{st}\;.
\end{equ}
Therefore, since $f_{\phi_g^*(e_i)} = f_{\bullet_i} = f_i$, it suffices to check that $f_{\phi_g^*(\sigma)}(Y_s)$ satisfies \eqref{e:geometricrecurrence} for each tensor $\sigma \in U_{N,1}$. Firstly, from Lemma \ref{lem:deltaphig}, we know that $(\phi_g \otimestilde \phi_g) \Delta = \Deltabar \phi_g$. Using the dual of this expression, we obtain
\begin{equ}
\phi_g^*(\sigma_1 \otimes \sigma_2) = \phi_g^*(\sigma_1)\star \phi_g^*(\sigma_2)\;,
\end{equ}
for any $\sigma_1, \sigma_2 \in T^{(N)}(\reals^d)$. In particular,
\begin{equ}
\phi_g^*(e_{v_1}\otimes \dots \otimes e_{v_n}) =  \phi_g^* (e_{v_1}) \star \phi_g^* (e_{v_2}\otimes \dots \otimes e_{v_n}) = e_{v_1} \star \phi_g^* (e_{v_2}\otimes \dots \otimes e_{v_n})\;.
\end{equ}
Combining this with Lemma \ref{lem:L_GL}, we obtain
\begin{align*}
f_{\phi_g^* ( e_{v_1}\otimes \dots \otimes e_{v_n})} &= f_{e_{v_1} \star \phi_g^* (e_{v_2}\otimes \dots \otimes e_{v_n})}\\ 
&= f_{e_{v_1}} \cdot Df_{\phi_g^*(e_{v_2}\otimes \dots \otimes e_{v_n})}\\
&= f_{\phi_g^*(e_{v_1})}\cdot Df_{\phi_g^*(e_{v_2}\otimes \dots \otimes e_{v_n})}\;. 
\end{align*}
This proves the claimed recurrence. For the `if' statement, suppose $Y$ satisfies \eqref{e:geo_controlled}, with coefficients $F_{\sigma}$ satisfying the recurrence \eqref{e:geometricrecurrence}. Let $f_\tau$ be the coefficients defined by \eqref{e:recurrence} with $f_{\bullet_i} =f_i$. Since both $F_{\sigma}$ and $f_{\phi_g^*(\sigma)}$ satisfy \eqref{e:geometricrecurrence}, with $F_{e_{i}}=f_{\phi_g^*(e_{i})}$ we must have $F_{\sigma}=f_{\phi_g^*(\sigma)}$ for all $\sigma \in U_{N,1}$. Then, using the same calculation as above, we have that
\begin{align*}
\delta Y_{st} - \sum_{\tau \in \trees_N} f_{\tau}(Y_s) \innerprod{\Xbar_{st},\tau} &= \delta Y_{st} - \sum_{\sigma \in U_{N,1}} f_{\phi_g^*(\sigma)}(Y_s) \innerprod{\Xbar_{st},\iota\sigma}\\
 &= \delta Y_{st} - \sum_{\sigma \in U_{N,1}} F_{\sigma}(Y_s) \innerprod{\Xbar_{st},\iota\sigma} = r_{st}\;.
\end{align*}
It follows from Proposition \ref{prop:euler} that $\bigybar$ is the controlled rough path solution to \eqref{e:LRDE}. Hence, upon proving Lemma \ref{lem:L_GL}, this completes the proof.
\end{proof}

\begin{proof}[Proof of Lemma \ref{lem:L_GL}]
First note that if $h$ is a non-trivial product then certainly $(\lambda_1 \dots \lambda_q) \star h$ is a linear combination of non-trivial products. Hence, we can restrict our attention to those $h \in \hopf^*$ that are linear combinations of elements in $\trees^*$, since $f_{\tau_1 \dots \tau_n} = 0$ all non-trivial products $\tau_1 \dots \tau_n \in \forest^* \setminus \trees^*$. Moreover, since $f_{\tau_1 + \dots + \tau_n} = f_{\tau_1} + \dots + f_{\tau_n}$, the claim will follow from 
\begin{equ}
D^q f_{\tau}: (f_{\lambda_1},\dots,f_{\lambda_q}) = f_{(\lambda_1 \dots \lambda_q)\star \tau}\;,
\end{equ}
for all $\lambda_1 , \dots , \lambda_q, \tau \in \trees^*$. 
\par
We will prove the claim by induction. The claim clearly holds for $\tau = \bullet_i$ and any $\lambda_1, \dots,\lambda_q \in \trees^*$, since this is simply the recurrence \eqref{e:recurrence}. Suppose the claim holds for $\tau \in \trees_m^*$ and all $\lambda_1,\dots,\lambda_q \in \trees^*$ we will prove the claim for $\tau \in \trees_{m+1}^*$ and all $\lambda_1,\dots,\lambda_q \in \trees^*$. Without loss of generality, let $ \tau = [\tau_1 \dots \tau_n]_i $ for $\tau_j \in \trees_m$. Firstly, by the recurrence \eqref{e:recurrence}, we have that 
\begin{equ}
D^q f_{[\tau_1 \dots \tau_n]_i} : (f_{\lambda_1},\dots,f_{\lambda_q}) =  D^q \bigg(D^n f_{i} : (f_{\tau_1} \dots f_{\tau_n})\bigg) : (f_{\lambda_1},\dots,f_{\lambda_q})\;.
\end{equ}
If we apply the Leibniz formula then the above equals
\begin{equ}\label{e:leibniz2}
\sum_{p,p_1,\dots,p_n}\binom{q}{p,p_1,\dots,p_n} D^{p+n} f_i : \bigg( f_{\lambda_1},\dots,f_{\lambda_p},u_1^{p_1},\dots,u_{n}^{p_n} \bigg)\;,
\end{equ} 
where $\binom{q}{p,p_1,\dots,p_n} = \frac{q!}{p! p_1 ! \dots p_n !}$ and where we sum over all partitions $p + p_1 + \dots + p_n = q$ and where 
\begin{equ}
u_i^{p_i} = D^{p_i} f_{\tau_i} : (f_{\lambda_1^{i}}, \dots , f_{\lambda_{p_i}^i})\;,
\end{equ}
using the notation 
\begin{equ}
(\lambda_1,\dots,\lambda_q) = (\lambda_1,\dots,\lambda_p,\lambda_1^1,\dots,\lambda_{p_1}^1,\dots,\lambda_1^{n},\dots,\lambda_{p_n}^n)\;.
\end{equ}
Since $\tau_i \in \trees_m$, it follows by the induction hypothesis that $u_i^{p_i} = f_{(\lambda_1^{i} \dots \lambda_{p_i}^i)\star \tau_i}$. Hence, \eqref{e:leibniz2} equals
\begin{align}\notag
&\sum_{\lambdatilde_1,\dots,\lambdatilde_n} \bigg(\inner{(\lambda_1^1 \dots \lambda_{p_1}^1)\star \tau_1,\lambdatilde_1} \dots \inner{(\lambda_1^n \dots \lambda_{p_1}^n)\star \tau_n, \lambdatilde_n}\bigg)\\ &\qquad\qquad\qquad\qquad \binom{q}{p,p_1,\dots,p_n}  \times  D^{p+n} f_i : \bigg( f_{\lambda_1},\dots,f_{\lambda_p},f_{\lambdatilde_1},\dots,f_{\lambdatilde_n}\bigg)\notag\\
&=\sum_{\lambdatilde_1,\dots,\lambdatilde_n} \bigg(\inner{(\lambda_1^1 \dots \lambda_{p_1}^1)\star \tau_1,\lambdatilde_1} \dots \inner{(\lambda_1^n \dots \lambda_{p_1}^n)\star \tau_n, \lambdatilde_n}\bigg)\notag \\ &\qquad\qquad\qquad\qquad\qquad\qquad\times\binom{q}{p,p_1,\dots,p_n} f_{[\lambda_1 \dots \lambda_p \lambdatilde_1 \dots \lambdatilde_n]_i} \label{e:leibniz_left}\;,
\end{align}
where we sum over all partitions $p+p_1+ \dots + p_n = q$ and all $\lambdatilde_i \in \trees$. On the other hand, we have that 
\begin{equ}\label{e:leibniz_right}
f_{(\lambda_1 \dots \lambda_q)\star [\tau_1 \dots \tau_n]_i} = \sum_{\sigma \in \trees} \inner{(\lambda_1 \dots \lambda_q) \star [\tau_1 \dots \tau_n]_i , \sigma} f_{\sigma}\;.
\end{equ}
By eliminating those terms that will vanish, this equals
\begin{equ}
\sum_{\rho_1,\dots,\rho_{n+p}} \frac{1}{(n+p)!}\inner{(\lambda_1 \dots \lambda_q) \otimest [\tau_1 \dots \tau_n]_i, \Delta [\rho_1 \dots \rho_{n+p}]_i} f_{[\rho_1 \dots \rho_{n+p}]_i}\;,
\end{equ}
where we sum over all $\rho_i \in \trees_N$ and the factor $\frac{1}{(n+p)!}$ appears due to the fact that $[\rho_1 \dots \rho_{n+p} ]_i$ is identical for different arrangements of the $\rho_i$. By definition, we have that
\begin{align*}
\sum_{p=0}^N\inner{(\lambda_1 \dots \lambda_q) &\otimest [\tau_1 \dots \tau_n]_i, \Delta [\rho_1 \dots \rho_{n+p}]_i}\\  &= \inner{\lambda_1 \dots \lambda_q,  \rho_1^{(1)} \dots \rho_{n+p}^{(1)}} \inner{\tau_1\dots \tau_n,\rho_1^{(2)}\dots \rho_{n+p}^{(2)}}\;.
\end{align*}
Now, each of these terms will vanish unless $\rho_i^{(2)} = 1$ for exactly $p$ factors in $\rho_1^{(2)}\dots \rho_{n+p}^{(2)}$. Moreover, since we are summing over all $\rho_i$, the expression must be symmetric in the $\rho_i$. In particular, we can assume that $\rho_i^{(2)}=1$ for $1 \leq i \leq p$, provided we include the combinatorial factor $\binom{n+p}{p}$. Of course, this implies $\rho_i^{(1)}\otimest \rho_i^{(2)} = \rho_i\otimest 1$ for $1\leq i \leq p$. Hence, \eqref{e:leibniz_right} equals
\begin{align*}
\sum_{\rho_1,\dots,\rho_{n+p}} \frac{1}{(n+p)!}\binom{n+p}{p} \inner{\lambda_1 \dots \lambda_q,  \rho_1 \dots \rho_p \rho_{p+1}^{(1)} \dots \rho_{p+n}^{(1)}} \inner{\tau_1\dots \tau_n,\rho_{p+1}^{(2)}\dots \rho_{p+n}^{(2)}} f_{[\rho_1 \dots \rho_{n+p}]}\;.
\end{align*}
By the symmetry of the expression, we can also simplify this to
\begin{align*}
&\sum_{\rho_1,\dots,\rho_{n+p}} \frac{n!}{(n+p)!}\binom{n+p}{p} \inner{\lambda_1 \dots \lambda_q,  \rho_1 \dots \rho_p \rho_{p+1}^{(1)} \dots \rho_{p+n}^{(1)}} \inner{\tau_1,\rho_{p+1}^{(2)}}\dots \inner{\tau_n,\rho_{p+n}^{(2)}} f_{[\rho_1 \dots \rho_{n+p}]} \\
&= \sum_{\rho_1,\dots,\rho_{n+p}} \frac{1}{p!}\inner{\lambda_1 \dots \lambda_q,  \rho_1 \dots \rho_p \rho_{p+1}^{(1)} \dots \rho_{p+n}^{(1)}} \inner{\tau_1,\rho_{p+1}^{(2)}}\dots \inner{\tau_n,\rho_{p+n}^{(2)}} f_{[\rho_1 \dots \rho_{n+p}]}\;.
\end{align*}
Repeating the same idea on the other factors, this equals
\begin{align*}
&\sum_{\rho_1,\dots,\rho_{n+p}} \frac{1}{p!} \binom{q}{p} \inner{\lambda_1 \dots \lambda_p,  \rho_1 \dots \rho_p } \inner{\lambda_{p+1} \dots \lambda_q,  \rho_{p+1}^{(1)} \dots \rho_{p+n}^{(1)}} \inner{\tau_1,\rho_{p+1}^{(2)}}\dots \inner{\tau_n,\rho_{p+n}^{(2)}} f_{[\rho_1 \dots \rho_{n+p}]}\\
&= \sum_{\rho_1,\dots,\rho_{n+p}} \binom{q}{p} \inner{\lambda_1,  \rho_1} \dots\inner{ \lambda_p,  \rho_p } \inner{\lambda_{p+1} \dots \lambda_q,  \rho_{p+1}^{(1)} \dots \rho_{p+n}^{(1)}} \inner{\tau_1,\rho_{p+1}^{(2)}}\dots \inner{\tau_n,\rho_{p+n}^{(2)}} f_{[\rho_1 \dots \rho_{n+p}]}\;.
\end{align*}
Let $p_1 + \dots + p_n = q-p$ be some partition and let $(\lambda_{p+1},\dots,\lambda_{q}) = (\lambda_1^1,\dots,\lambda_{p_1}^1,\dots,\lambda_1^n,\dots,\lambda_{p_n}^n)$. Then it also follows from the symmetry of the expression that the above equals
\begin{align*}
&\sum_{\rho_1,\dots,\rho_{n+p}} \binom{q}{p}\binom{q-p}{p_1,\dots,p_n} \inner{\lambda_1,  \rho_1} \dots\inner{ \lambda_p,  \rho_p } \inner{\lambda_1^1\dots\lambda_{p_1}^1,  \rho_{p+1}^{(1)} }\dots \inner{\lambda_1^n\dots\lambda_{p_n}^n,  \rho_{p+n}^{(1)} }\\ 
&\qquad\qquad\qquad\qquad \times\inner{\tau_1,\rho_{p+1}^{(2)}}\dots \inner{\tau_n,\rho_{p+n}^{(2)}} f_{[\rho_1 \dots \rho_{n+p}]} \\
& = \sum_{\rho_{p+1},\dots,\rho_{n+p}} \binom{q}{p,p_1,\dots,p_n}\inner{(\lambda_1^1,\dots,\lambda_{p_1}^1)\star \tau_1,  \rho_{p+1} }\dots \inner{(\lambda_1^n,\dots,\lambda_{p_n}^n)\star \tau_n,  \rho_{p+n}}\\&\qquad\qquad\qquad\qquad\qquad\times f_{[\lambda_1 \dots \lambda_p \rho_{p+1}\dots \rho_{p+n}]}\;,
\end{align*}
where we sum over all partitions $p+p_1 + \dots + p_n = q$. We see that \eqref{e:leibniz_right} equals \eqref{e:leibniz_left}, which proves the induction.    
\end{proof}

\subsection{It\^o-Stratonovich correction}\label{subsec:strato}
We can now state and prove the generalised correction formula. In the following, let $f(Y)\cdot dX = \sum_{i=1}^d f_i(Y)dX^i $ for smooth vector fields $f_i: \reals^e \to \reals^e$. As usual, let $f_\tau$ be defined by the recurrence \eqref{e:recurrence} with $f_{\bullet_i} = f_i$ and $f_h = \inner{h,1}\Id + \sum_{\tau \in \trees} \inner{h,\tau}f_\tau$ for any $h\in\hopf^*$. Finally, let $\xbar,\Xbar$ be as in Theorem \ref{thm:branchedgeometric}  
\begin{thm}\label{thm:strato}
Let $\bigy$ with $\inner{1,\bigy}=Y$ be the controlled rough path solution to the RDE
\begin{equ}\label{e:convertRDE}
dY_t = f(Y_t)\cdot dX_t\;,
\end{equ}
driven by a branched rough path $\X$ over $X$. Then $\bigy$ also solves the RDE 
\begin{equ}\label{e:conversion_RDE_1}
dY_t = \fbar(Y_t) \cdot d\xbar_t\;,
\end{equ} 
driven by $\Xbar$, where $\fbar(Y)\cdot d\xbar = \sum_{\tau \in \trees_N} f_\tau(Y)d\xbar^\tau$.\end{thm}
\begin{proof}
As in \eqref{e:ito_strato1}, we have that  
\begin{equ}
\sum_{\tau \in \trees_N} f_{\tau}(Y_s)\inner{\X_{st},\tau} = \sum_{\sigma \in U_{N,N}} f_{\psi^*(\sigma)}(Y_s) \innerprod{\Xbar_{st},\sigma}\;.
\end{equ} 
Therefore, as stated in Remark \ref{rmk:different_VS}, if we can show that the coefficients $f_{\psi^*(\sigma)}$ satisfy 
\begin{equ}
f_{\psi^*(\tau_1 \otimes \dots \otimes \tau_n)} = f_{\psi^*(\tau_1)}\cdot Df_{\psi^*(\tau_2 \otimes \dots \otimes \tau_n)}\;,
\end{equ}
for all $\tau_1 \otimes \dots \otimes \tau_n \in U_{N,N}$ then Proposition \ref{prop:georecurrence} implies that $\bigy$ also solves the RDE
\begin{equ}\label{e:conversion_RDE_2}
dY_t =  \sum_{\tau \in \trees_N} f_{\psi^*(\tau)}(Y_t) d\xbar_t^\tau \;,
\end{equ}
driven by the geometric rough path $\Xbar$. From the definition of $\psi$ found in \eqref{e:psi_definition}, one can easily check that $\psi^*(\tau) = \tau$, so that \eqref{e:conversion_RDE_1} and \eqref{e:conversion_RDE_2} are indeed the same RDE. 
\par
From Lemma \ref{lem:psiexist}, we know that $(\psi\otimes\psi)\Delta = \Deltabar \psi$, the dual of this statement implies that
\begin{align*}
\psi^* (\tau_1\otimes \dots \otimes \tau_n) &= \psi^*(\tau_1)\star \psi^*(\tau_2 \otimes\dots\otimes \tau_n)\\ 
&= \sum_{\sigma_1 \in \trees_N} \innerprod{\psi(\sigma_1),\tau_1} \big(\sigma_1 \star \psi^*(\tau_2 \otimes\dots\otimes \tau_n)\big)\;.
\end{align*}
In light of this, the theorem follows almost immediately from Lemma \ref{lem:L_GL}, where we take $\lambda_1 \dots \lambda_q=\sigma_1$ and $h=\sigma_2 \star \dots \star \sigma_n$. We have that
\begin{align*}
f_{\psi^* (\tau_1\otimes \dots \otimes \tau_n)} &= \sum_{\sigma_1\in\trees_N}\inner{\psi(\sigma_1),\tau_1}  f_{ \sigma_1 \star \psi^*(\tau_2 \otimes\dots\otimes \tau_n)}\\
&=  \sum_{\sigma_1\in\trees_N}\inner{\psi(\sigma_1),\tau_1} f_{\sigma_1}\cdot  Df_{ \psi^*(\tau_2 \otimes\dots\otimes \tau_n)}\\
&= f_{\psi^*(\tau_1)}\cdot Df_{\psi^*(\tau_2 \otimes\dots\otimes \tau_n)}\;.
\end{align*}
This proves the recurrence \eqref{e:geometricrecurrence} and hence completes the proof.
\end{proof}
\begin{rmk}
As with the geometric rough path $\Xbar$, there is some redundancy in the extended vector field $\fbar$. In fact, it is in general possible to choose another geometric rough path $\Xhat$, such that $Y$ also solves an RDE driven by $\Xhat$ that features \emph{fewer} vector fields. For example, if $1/3< \gamma \leq 1/2$, we have that
\begin{align*}
\delta Y_{st} &= f_i(Y_s)\innerprod{\Xbar_{st},\mytreeone{i}\;\;} + f_j^{\alpha}(Y_s) \del^{\alpha} f_i(Y_s) \innerprod{\Xbar_{st},\mytreeone{j}\;\; \otimes\mytreeone{i}\;\;}\\ 
&+ f_j^{\alpha}(Y_s) \del^{\alpha} f_i(Y_s) \innerprod{\Xbar_{st}, \mytreeoneone{j}{i}\;\;}\; + r_{st}\;,
\end{align*}
for $|r_{st}|= \liloh(|t-s|)$. Let us now define $\innerprod{\Xhat,\mytreeone{i}\;\;} = \innerprod{\Xbar,\mytreeone{i}\;\;}$ and 
\begin{equ}
\innerprod{\Xhat,\mytreeone{j}\;\;\otimes\mytreeone{i}\;\;} = \innerprod{\Xbar,\mytreeone{j}\;\;\otimes\mytreeone{i}\;\; + \mytreeoneone{j}{i}\;\;- \mytreeoneone{i}{j}\;\;}\;.
\end{equ}
for all $i,j=1,\dots,d$, and finally, we set
\begin{equ}
\innerprod{\Xhat,\mytreeone{ij}\;\;\;} \defin \innerprod{\Xbar,\mytreeoneone{j}{i}\;\; + \mytreeoneone{i}{j}\;\;}
\end{equ}
for all $i\leq j$. Since we have only changed the higher order components of $\Xbar$ by adding an anti-symmetric $2\gamma$-H\"older path, $\Xhat$ remains geometric. Moreover, we see that
\begin{align*}
\delta Y_{st} &= f_i(Y_s)\innerprod{\Xhat_{st},\mytreeone{i}\;\;} + f_j^{\alpha}(Y_s)\del^\alpha f_i(Y_s) \innerprod{\Xhat_{st},\mytreeone{j}\;\;\otimes \mytreeone{i}\;\;}\\ &+ \frac{1}{2}\big(f_l^{\alpha}(Y_s)\del^\alpha f_k(Y_s)  + f_k^{\alpha}(Y_s)\del^\alpha f_l(Y_s)  \big)\innerprod{\Xhat_{st},\mytreeone{kl}\;\;\;} + r_{st}\;.
\end{align*}
Hence, $Y$ solves the RDE
\begin{equ}
dY_t = f_i(Y_t)d\xhat^i_t +   \frac{1}{2}\big(f_l^{\alpha}(Y_t)\del^\alpha f_k(Y_t)  + f_k^{\alpha}(Y_t)\del^\alpha f_l(Y_t)  \big)d\xhat^{kl}\;,
\end{equ}
where we only sum over $k\leq l$. This is clearly a simpler RDE than the one obtained in Theorem \eqref{thm:strato}, since we sum over a smaller index set than $\trees_2$. It is also more reminiscent of the usual It\^o-Stratonovich correction. 
\par
To put this another way, suppose $X$ is a Brownian motion, or indeed any path for which there is a \emph{canonical} geometric rough path lying above it \cite{victoir10,friz10}; in the case of Brownian motion, this corresponds to constructing Stratonovich integrals. If the extension $\Xbar$ were constructed using this canonical geometric rough path above $X$, then we would recover the classical It\^o-Stratonovich correction. In particular, the antisymmetric part $\frac{1}{2}\innerprod{\Xbar_{st},\mytreeoneone{j}{i}\;\;- \mytreeoneone{i}{j}\;\;}$ would vanish and the symmetric part $\frac{1}{2}\innerprod{\Xbar_{st},\mytreeoneone{j}{i}\;\;+ \mytreeoneone{i}{j}\;\;}$ would be the quadratic variation.      
\par
The simplification is quite easy in the case $N=2$, but the procedure becomes much more complicated for larger $N$, and we can see no natural method of generalising this simplification for larger $N$.
\end{rmk}

\bibliographystyle{./Martin}
\bibliography{ito}

\end{document}